\documentclass{amsart}
\usepackage{amsmath, amssymb, amscd, amsthm, amsfonts, amstext, amsbsy,  mathrsfs, hyperref,  upgreek, mathtools, stmaryrd, enumitem, tensor, comment, lineno, tikz-cd}
\hypersetup{colorlinks=true}
\usepackage[shadow, textsize=footnotesize,color=blue!40
, bordercolor=blue!80]{todonotes}
\numberwithin{equation}{section}

\usepackage{xcolor, color, soul}
\sethlcolor{yellow}

\renewcommand{\subset}{\subseteq}

\newcommand{\B}{\mathcal{B}}

\newcommand{\F}{\mathcal{F}}

\newcommand{\GG}{\mathbb{G}}

\newcommand{\QQ}{\mathbb{Q}}
\newcommand{\RR}{\mathbb{R}}

\newcommand{\ZZ}{\mathbb{Z}}

\newcommand{\pre}[2]{\tensor[^{#1}]{#2}{}}

\newcommand{\Nbhd}{\mathbf{N}}
\newcommand{\Bor}{\mathsf{Bor}}
\newcommand{\pI}{\mathrm{I}}
\newcommand{\pII}{\mathrm{II}}
\newcommand{\markdef}[1]{\textbf{#1}}

\newcommand{\pow}{\mathscr{P}}

\newcommand{\leng}{\operatorname{lh}}

\DeclareMathOperator{\Deg}{Deg}
\DeclareMathOperator{\conc}{\mathbin{ {}^\smallfrown {} }}

\newcommand{\Cho}{\mathrm{G}}
\newcommand{\fCho}{f\mathrm{G}}
\newcommand{\SC}{\mathrm{SC}}
\newcommand{\SFC}{f\mathrm{SC}}

\newcommand{\ZFC}{{\sf ZFC}}

\newcommand{\PSP}{{\sf PSP}}

\newcommand{\bPSP}{\( \Bor_\kappa \)-{\sf PSP}}

\newcommand{\Lev}{\mathrm{Lev}}

\newcommand{\Succ}{\mathrm{Succ}}

\DeclareMathOperator{\cl}{cl}

\newtheorem{theorem}{Theorem}[section]
\newtheorem{lemma}[theorem]{Lemma}
\newtheorem{corollary}[theorem]{Corollary}
\newtheorem{proposition}[theorem]{Proposition}

\theoremstyle{definition}
\newtheorem{claim}{Claim}[theorem]
\newtheorem*{claim*}{Claim}
\newtheorem{definition}[theorem]{Definition}
\newtheorem{fact}[theorem]{Fact}
\newtheorem{example}[theorem]{Example}
\newtheorem{question}[theorem]{Question}
\newtheorem*{question*}{Question}
\theoremstyle{remark}
\newtheorem{remark}[theorem]{Remark}

\newenvironment{enumerate-(a)}{\begin{enumerate}[label={\upshape (\alph*)}, leftmargin=2pc]}{\end{enumerate}}

\newenvironment{enumerate-(a)-r}{\begin{enumerate}[label={\upshape (\alph*)}, leftmargin=2pc,resume]}{\end{enumerate}}

\newenvironment{enumerate-(A)}{\begin{enumerate}[label={\upshape (\Alph*)}, leftmargin=2pc]}{\end{enumerate}}

\newenvironment{enumerate-(A)-r}{\begin{enumerate}[label={\upshape (\Alph*)}, leftmargin=2pc,resume]}{\end{enumerate}}

\newenvironment{enumerate-(i)}{\begin{enumerate}[label={\upshape (\roman*)}, leftmargin=2pc]}{\end{enumerate}}

\newenvironment{enumerate-(i)-r}{\begin{enumerate}[label={\upshape (\roman*)}, leftmargin=2pc,resume]}{\end{enumerate}}

\newenvironment{enumerate-(I)}{\begin{enumerate}[label={\upshape (\Roman*)}, leftmargin=2pc]}{\end{enumerate}}

\newenvironment{enumerate-(I)-r}{\begin{enumerate}[label={\upshape (\Roman*)}, leftmargin=2pc,resume]}{\end{enumerate}}

\newenvironment{enumerate-(1)}{\begin{enumerate}[label={\upshape (\arabic*)}, leftmargin=2pc]}{\end{enumerate}}

\newenvironment{enumerate-(1)-r}{\begin{enumerate}[label={\upshape (\arabic*)}, leftmargin=2pc,resume]}{\end{enumerate}}

\newenvironment{itemizenew}{\begin{itemize}[leftmargin=2pc]}{\end{itemize}}

\begin{document}

\title{Generalized Polish spaces at regular uncountable cardinals}
\date{\today}

\author{Claudio Agostini}
\address[Claudio Agostini]
{TU Wien,
Institute of Discrete Mathematics and Geometry,
Wiedner Hauptstrasse 8-10/104, 1040 Wien, Austria}
\email[Claudio Agostini]{claudio.agostini@tuwien.ac.at}

\author{Luca Motto Ros}
\address[Luca Motto Ros]
{Universit\`a degli Studi di Torino,
Dipartimento di Matematica ``G. Peano'',
Via Carlo Alberto 10, 10123 Torino, Italy}
\email[Luca Motto Ros]{luca.mottoros@unito.it}

\author{Philipp Schlicht}
\address[Philipp Schlicht]{School of Mathematics,
University of Bristol,
Fry Building.
Woodland Road,
Bristol, BS8~1UG, UK}
\email[Philipp Schlicht]{philipp.schlicht@bristol.ac.uk}

\thanks{We are grateful to the referee for their helpful comments. The work of the first- and second-listed authors was supported by the Gruppo 
Nazionale per le Strutture Algebriche, Geometriche e le loro Applicazioni (GNSAGA) of 
the Istituto Nazionale di Alta Matematica (INdAM) of Italy, and by project PRIN 2017 
``Mathematical Logic: models, sets, computability'', prot. 2017NWTM8R.
The third-listed author has received funding through EPSRC grant number EP/V009001/1, FWF grant number I4039 and the European Union's Horizon 2020 research and innovation programme under the Marie Sk\l odowska-Curie grant agreement No 794020 (IMIC). 
For the purpose of open access, the authors have applied a ‘Creative Commons Attribution' (CC BY) public copyright licence to any Author Accepted Manuscript (AAM) version arising from this submission.}

\begin{abstract}
In the context of generalized descriptive set theory, we systematically compare and 
analyze various notions of Polish-like spaces and standard \( \kappa \)-Borel spaces for 
\( \kappa \) an uncountable (regular) cardinal satisfying \( \kappa^{< \kappa} = \kappa \). 
As a result, we obtain a solid framework where one can develop the theory in full 
generality. We also provide natural characterizations of the generalized Cantor and Baire 
spaces. Some of the results obtained considerably extend previous work 
from~\cite{CoskSchlMR3453772,Gal19,LuckeSchlMR3430247}, and answer some 
questions contained therein.
\end{abstract}

\maketitle


\section{Introduction}

Generalized descriptive set theory is a very active field of research which in recent years received a lot
of attention because of its deep connections with other areas such as model theory, determinacy and
higher pointclasses from classical descriptive set theory, combinatorial set theory, classification of
uncountable structures and non-separable spaces, and so on. Basically, the idea is to
replace \( \omega \) with an uncountable cardinal in the definitions of the Baire space
\( \pre{\omega}{\omega} \) and Cantor space \( \pre{\omega}{2} \), as well as in all other
topologically-related notions. For example, one considers \( \kappa \)-Borel sets (i.e.\ sets in the
smallest \( \kappa^+ \)-algebra generated by the topology) instead of Borel sets, \( \kappa \)-Lindel\"of
spaces (i.e.\ spaces such that all their open coverings admit a \( {<} \kappa \)-sized subcovering) instead of
compact spaces, \( \kappa \)-meager sets (i.e.\ unions  of \( \kappa \)-many nowhere dense sets) instead
of meager sets, and so on. See~\cite{Friedman:2011nx,AM} for a general introduction and the basics of this subject.

The two spaces lying at the core of the theory are then:
\begin{enumerate-(1)}
\item
 the \markdef{generalized Baire space}
 \[
 \pre{\kappa}{\kappa} = \{ x \mid x \colon \kappa \to \kappa \}
 \]
of all sequences with values in \( \kappa \) and length \( \kappa \), equipped with the so-called \markdef{bounded topology} $\tau_b$, i.e.\ the topology generated by the sets of the form
\[
\Nbhd_s = \{ x \in \pre{\kappa}{\kappa} \mid s \subseteq x \}
 \]
with \( s \) ranging in the set \( \pre{<\kappa}{\kappa} \) of sequences with values in \( \kappa \) and length \( {<} \kappa \);
\item
the \markdef{generalized Cantor space}
\[
\pre{\kappa}{2} = \{ x \mid x \colon \kappa \to 2 \}
 \]
of all binary sequences of length \( \kappa \), which is a closed subset of \( \pre{\kappa}{\kappa} \) and is thus equipped with the relative topology.
\end{enumerate-(1)}
When $\kappa$ is regular, these topologies are \markdef{\( \kappa \)-additive}, i.e. they are closed under intersections of  length \( < \kappa \). 
Since the classical Cantor and Baire spaces are second-countable, it is natural to desire that, accordingly,
\( \pre{\kappa}{\kappa} \) and \( \pre{\kappa}{2} \) have weight \( \kappa \): this amounts to require that
\( \kappa^{< \kappa} = \kappa \) or, equivalently, that \( \kappa \) is regular and such that
\( 2^{< \kappa} = \kappa \). Thus such assumption is usually adopted as one of the basic conditions to
develop the theory, and this paper is no exception.

Despite the achievements already obtained by generalized descriptive set theory, there is still a missing
ingredient. The success and strong impact experienced by classical descriptive set theory in other areas
of mathematics is partially due to its wide applicability: the theory is developed for arbitrary
completely metrizable second-countable (briefly: \markdef{Polish}) \markdef{spaces} and for \markdef{standard Borel spaces}, which are
ubiquitous in most mathematical fields. In contrast, generalized descriptive set theory so far concentrated
(with a few exceptions) only on \( \pre{\kappa}{\kappa} \) and \( \pre{\kappa}{2} \), and this constitutes a
potential limitation. Our goal is to fill this gap by considering various generalizations 
of Polish and standard Borel spaces already proposed in the literature, adding a few more natural options, and
then systematically compare them from various points of view (see Figure~\ref{fig:sum_up_relationships}). 
Some of these results substantially extend and improve previous work appeared in~\cite{CoskSchlMR3453772,Gal19}.

Our analysis reveals that when moving to uncountable cardinals $\kappa$, there is no 
preferred option among the possible generalizations of Polishness. Depending on which 
properties one decides to focus on, certain classes behave better than others, but there 
is no single class simultaneously sharing all nice features typically enjoyed by the 
collection of (classical) Polish spaces. For example, if one is interested in maintaining the 
usual closure properties of the given class (e.g.\ products, closed and \( G_\delta \) 
subspaces, and so on), then the ``right'' classes are arguably those of 
(\( \kappa \)-additive) \( \SFC_\kappa \)-spaces or \( \GG \)-Polish spaces---see 
Definitions~\ref{def:strongfairChoquet} and~\ref{def:K-Polish}, 
Theorem~\ref{thm:charK-Polish}, and 
Theorems~\ref{thm:closurepropertiesSFC},~\ref{thm:closureproperties}, 
and~\ref{thm:closurecontopenimages}.  On the other hand, if one is interested in an 
analogue of the Cantor-Bendixson theorem for perfect spaces, then one should better 
move to the class of (\( \kappa \)-additive) \( \SC_\kappa \)-spaces---see 
Definition~\ref{def:strongChoquet}, \cite[Proposition 3.1]{CoskSchlMR3453772}, and 
Theorem~\ref{thm:PSP for perfect SC_kappa kappa-additive}.

These different possibilities and behaviors are reconciled at the level of 
\( \kappa \)-Borel sets:  all the proposed generalizations give rise to the same class of 
spaces up to \( \kappa \)-Borel isomorphism (Theorems~\ref{thm:char_up_to_kBorel_iso} 
and~\ref{th:standardBorel}), thus they constitute  a natural
and solid setup to work with. We also provide a mathematical explanation of the special 
role played in the theory by the generalized Cantor and Baire spaces. On the one hand 
they admit nice characterizations which are analogous to the ones obtained in the 
classical setting by Brouwer and Alexandrov-Urysohn 
(Theorems~\ref{thm:charBairespaceallcardinals}, \ref{thm:charCantornonweaklycompact}, and~\ref{thm:charCantorweaklycompact}). 
On the other hand, when restricting to \( \kappa \)-additive spaces all our classes can be 
described, up to homeomorphism, as collections of simply definable subsets of 
\( \pre{\kappa}{\kappa} \) and \( \pre{\kappa}{2} \) (Theorems~\ref{thm:charK-Polish}, \ref{thm:SC-space+G-Polish}, \ref{thm:charK-Polish-Lindelof}, and~\ref{thm:char-spherically-complete-K-Polish-Lindelof}). 

In a sequel to this paper, we will provide many more concrete examples of (not 
necessarily \( \kappa \)-additive) Polish-like spaces, thus showing that such classes are 
extremely rich and not limited to simple subspaces of \( \pre{\kappa}{\kappa} \). In 
combination with the more theoretical observations presented in this paper, we believe 
that our results provide a wide yet well-behaved setup for developing generalized 
descriptive set theory, thus opening the way to fruitful applications to other areas of 
mathematics.


\section{Polish-like spaces} \label{sec:relationship_SC_SFC_Polish}

\subsection{Spaces, games, and metrics}

Throughout the paper we work in \( \ZFC \) and assume that \( \kappa \) is an 
uncountable regular cardinal satisfying \( 2^{< \kappa} = \kappa \) (equivalently: 
\( \kappa^{< \kappa} = \kappa \)).
Unless otherwise specified, from now on all topological spaces are assumed to be regular and Hausdorff, and we will refer to them just as ``spaces''.
In this framework, (classical) Polish spaces can 
equivalently be defined as:
\begin{enumerate}[label={(Pol.\ \upshape \arabic*)}, leftmargin=4.5pc]
\item \label{intro:pol1}
completely metrizable second-countable spaces;
\item \label{intro:pol2}
strong Choquet second-countable spaces, where strong Choquet means that 
the second player has a winning strategy in a suitable topological game, called strong 
Choquet game, on the given space (see below for the precise definition).
\end{enumerate}
Consider now pairs \( ( X, \mathscr{B} )\) with \( X \) a nonempty set and \( \mathscr{B} \) 
a \( \sigma \)-algebra on \( X \). Such pairs are called Borel spaces if \( \mathscr{B} \)
 is countably generated and separates points%
\footnote{A family \( \mathscr{B} \subseteq \pow(X) \) separates points if for all distinct 
\( x,y \in X \) there is \( B \in \mathscr{B} \) with \( x \in B \) and \( y \notin B \).} 
or, equivalently, if there is a metrizable second-countable topology on \( X \) generating 
\( \mathscr{B} \) as its Borel \(\sigma\)-algebra. Standard Borel spaces can then 
equivalently be defined as:
\begin{enumerate}[label={(St.Bor.\ \upshape \arabic*)}, leftmargin=4.5pc]
\item \label{intro:stbor1}
Borel spaces \( ( X, \mathscr{B} ) \) such that there is a Polish topology on \( X \) 
generating \( \mathscr{B} \) as its Borel \(\sigma\)-algebra;
\item \label{intro:stbor2}
Borel spaces which are Borel isomorphic to a Borel subset of \( \pre{\omega}{\omega} \) 
(or any other uncountable Polish space, including \( \pre{\omega}{2} \)).
\end{enumerate}

In~\cite{MottoRos:2013}, a notion of standard \( \kappa \)-Borel space was introduced by 
straightforwardly generalizing the definition given by~\ref{intro:stbor2}. Call a pair 
\( (X, \mathscr{B}) \) a \( \kappa \)-Borel space if \( \mathscr{B} \) is a 
\( \kappa^+ \)-algebra on \( X \) which separates points and admits a $\leq \kappa$-sized
basis. The elements of \( \mathscr{B} \) are then called \markdef{\( \kappa \)-Borel} sets 
of \( X \). If \( (X, \mathscr{B}) \) is a \( \kappa \)-Borel space and \( Y \subseteq X \), then 
setting \( \mathscr{B} \restriction Y = \{ B \cap Y \mid B \in \mathscr{B} \} \) we get that 
\( (Y, \mathscr{B} \restriction Y) \) is again a \( \kappa \)-Borel space.
If \( (X,\mathscr{B}) \) and \( (X', \mathscr{B}') \) are \( \kappa \)-Borel spaces, we say that 
a function \( f \colon X \to X' \) is \markdef{\( \kappa \)-Borel} (\markdef{measurable}) if
 \( f^{-1}(B) \in \mathscr{B} \) for all \( B \in \mathscr{B}' \). A \markdef{\( \kappa \)-Borel 
 isomorphism} between \( (X,\mathscr{B}) \) and \( (X', \mathscr{B}') \) is a bijection \( f \)  
 such that both \( f \) and \( f^{-1} \) are \( \kappa \)-Borel; 
 two \( \kappa \)-Borel spaces are then 
 \markdef{\( \kappa \)-Borel isomorphic} if there is a \( \kappa \)-Borel isomorphism 
 between them.
 Finally, a 
 \markdef{\( \kappa \)-Borel embedding} \( f \colon X \to X' \) is an injective function 
 which is a \( \kappa \)-Borel isomorphism between \( (X, \mathscr{B}) \) and 
 \( (f(X), \mathscr{B'} \restriction f(X)) \).  Notice that every \( T_0 \) topological space 
 \( (X,\tau) \) of weight $\leq \kappa$  can be seen as a \( \kappa \)-Borel space in a canonical 
 way by pairing it  with the collection
\[
\Bor_\kappa(X,\tau)
\]
of all its \( \kappa \)-Borel subsets, i.e.\ with the smallest \( \kappa^+ \)-algebra  
generated by its topology. (We sometimes remove \( \tau \) from this notation if clear 
from the context.) If not specified otherwise, we are always tacitly referring to such 
\( \kappa^+ \)-Borel structure when dealing with \( \kappa \)-Borel isomorphisms and 
\( \kappa \)-Borel embeddings between topological spaces.

 We are now ready to generalize~\ref{intro:stbor2}.

\begin{definition}\label{def:standardBorel}
A \( \kappa \)-Borel space \( (X, \mathscr{B}) \) is \markdef{standard}%
\footnote{Our definition of a standard \( \kappa \)-Borel space is slightly different yet equivalent to the one considered in~\cite{MottoRos:2013}. Indeed, the difference is that in~\cite[Definition 3.6]{MottoRos:2013} a \( \leq \kappa \)-weighted topology generating the standard \( \kappa \)-Borel structure is singled out---see also the discussion after Corollary~\ref{cor:topgrneratingstandardBorel}.}
 if
it is \( \kappa \)-Borel isomorphic to a \( \kappa \)-Borel subset of \( \pre{\kappa}{\kappa} \).
\end{definition}

Generalizations of~\ref{intro:stbor1} were instead not considered 
in~\cite{MottoRos:2013} 
because at that time no natural generalization of the concept of a Polish space was 
introduced yet. But clearly, once we are given a notion of a Polish-like space for 
\( \kappa \) (e.g.\ the ones we are going to consider below, namely  
\( \SC_\kappa \)-spaces, \( \SFC_\kappa \)-spaces, or \( \GG \)-Polish spaces), we can 
accordingly generalize~\ref{intro:stbor1} by considering those \( \kappa \)-Borel spaces 
which admit a topology of the desired type generating \( \mathscr{B} \) as its 
\( \kappa^+ \)-algebra of \( \kappa \)-Borel sets. This yields to several formally different 
definitions: we will however show 
that they all coincide, so that there is no need to notationally and terminologically 
distinguish them at this point.

We now move to some natural generalizations of Polishness.
 In~\cite{CoskSchlMR3453772}, the authors considered a natural generalization 
 of~\ref{intro:pol2} to uncountable regular cardinals \( \kappa \) in order to obtain a notion of 
 ``Polish-like'' spaces, called therein strong \( \kappa \)-Choquet spaces. Let us recall the 
 relevant definitions. 
The (classical) Choquet game \( \Cho_\omega(X) \) on a topological space \( X \) is the 
game where two players \( \pI \) and \( \pII \) alternatively pick nonempty open sets 
\( U_n\) and \( V_n \)
\[
\begin{array}{c|cccccc}
\pI & U_0 &  & U_1 & & \dotsc & \\
\hline
\pII & & V_0 & & V_1 & & \dotsc
\end{array}
 \]
so that \( U_{n+1} \subseteq V_n \subseteq U_n \); player \( \pII \) wins the run if  the set \( \bigcap_{n \in \omega} U_n = \bigcap_{n \in \omega} V_n \) is nonempty.
The strong Choquet game \( \Cho^s_\omega(X) \) is the variant of \( \Cho_\omega(X) \) where \( \pI \) additional plays points \( x_n \in U_n \)
\[
\begin{array}{c|cccccc}
\pI & (U_0, x_0) &  & (U_1, x_1) & & \dotsc & \\
\hline
\pII & & V_0 & & V_1 & & \dotsc
\end{array}
 \]
 and \( \pII \) ensures that \( x_n \in V_n \subseteq U_n \); the winning condition stays the same.

It is (almost) straightforward to generalize such games to uncountable \( \kappa \)'s: just 
let players \( \pI \) and \( \pII \) play for \( \kappa\)-many rounds, and still declare \( \pII \) 
as the winner of the run if the final intersection \( \bigcap_{\alpha < \kappa} U_\alpha = \bigcap_{\alpha < \kappa} V_\alpha \) is nonempty.
However, since \( \kappa > \omega \) we now have to decide what should happen at limit levels
\( \gamma < \kappa \). Firstly, since the space \( X \) is not necessarily \( \kappa \)-additive we require
\( U_\gamma, V_\gamma \) to be just open \emph{relatively to what has been played so far},
i.e.\ relatively to \( \bigcap_{\alpha < \gamma} U_\alpha = \bigcap_{\alpha< \gamma} V_\alpha \) (this obviously
applies to all rounds with index \( \gamma \geq \omega \), not only to the limit ones). A more subtle
issue is deciding who wins the game if at some limit \( \gamma < \kappa \) we already have
\( \bigcap_{\alpha < \gamma} U_\alpha = \bigcap_{\alpha < \gamma} V_\alpha = \emptyset \), so that the
game cannot continue from that round on. Following~\cite{CoskSchlMR3453772}, the \markdef{\( \kappa \)-Choquet game} \( \Cho_\kappa(X) \) and the \markdef{strong} \markdef{\( \kappa \)-Choquet game} \( \Cho^{s}_\kappa(X) \) on \( X \) are defined
by letting \( \pI \) win in such situations.  In other words, \( \pII \) has to ensure that for all limit
\( \gamma \leq \kappa \) (thus including, in particular,  the final stage \( \gamma = \kappa \)), the intersection
\( \bigcap_{\alpha < \gamma} U_\alpha = \bigcap_{\alpha< \gamma} V_\alpha \) is nonempty. This leads
to the following definition.

\begin{definition} \label{def:strongChoquet}
A space \( X \) is called \markdef{strong \( \kappa \)-Choquet} (or \markdef{\( \SC_\kappa \)-space}) if 
it has%
\footnote{Notice that we are deliberately allowing our spaces to have weight strictly 
smaller than \( \kappa \). Although this might sound unnatural at first glance, it allows us 
to state some of our results in a more elegant form and is perfectly coherent with what 
is done in the classical case, where one includes among Polish spaces also those of finite 
weight.}
weight \( \leq \kappa \) and
player \( \pII \) has a winning strategy in \( \Cho^s_\kappa(X) \).
\end{definition}

The other natural option, not yet considered so far in the literature, is to make the game more fair by
deciding that \( \pI \) partially shares the burden of having a nonempty intersection and takes care of
limit levels \( \gamma < \kappa \). In other words: \( \pII \) wins if he can guarantee that
\( \bigcap_{\alpha< \kappa} U_\alpha = \bigcap_{\alpha< \kappa} V_\alpha \neq \emptyset \),
\emph{provided that for all limit \( \gamma < \kappa \) the intersection
\( \bigcap_{\alpha < \gamma} U_\alpha = \bigcap_{\alpha< \gamma} V_\alpha \) is nonempty} (if this fails
at some limit stage before \( \kappa \), then the run terminates and \( \pII \) automatically wins). We call this version of the Choquet game
\markdef{fair \( \kappa \)-Choquet game} and denote it by \( \fCho_\kappa(X) \), while its further variant
with player \( \pI \) additionally choosing points is called \markdef{strong fair \( \kappa \)-Choquet
game} and is denoted by \( \fCho^s_\kappa(X) \), accordingly.

\begin{definition} \label{def:strongfairChoquet}
A space \( X \) is called \markdef{strong fair \( \kappa \)-Choquet} (or  \( \SFC_\kappa \)\markdef{-space}) if
it has weight \( \leq \kappa \)
and
player \( \pII \) has a winning strategy in \( \fCho^s_\kappa(X) \).
\end{definition}

Since it is more difficult for player \( \pII \) to win the strong \( \kappa \)-Choquet game 
than its fair variant, it is clear from the definition that every \( \SC_\kappa \)-space is in 
particular an \( \SFC_\kappa \)-space. Moreover, both \( \pre{\kappa}{\kappa} \) and 
\( \pre{\kappa}{2} \) are trivially \( \SC_\kappa \)-spaces (any legal strategy where \( \pII \) 
plays basic open sets is automatically winning in the corresponding strong 
\( \kappa \)-Choquet games), and thus they are also \( \SFC_\kappa \)-spaces.

\begin{remark} \label{rmk:strategyrangeinbasis}
Although it is not part of the rules in Choquet-like games, in the above definitions one 
could equivalently require the players to pick only open sets from any given basis of the 
topological space (possibly intersected with all previous moves, if the space is not 
$\kappa$-additive)---see~\cite[Lemma 2.5]{CoskSchlMR3453772}. This restriction will 
turn out to be useful in some of the proofs below.
\end{remark}

We next move to generalizations of~\ref{intro:pol1}. This requires to find suitable 
analogues of metrics over the real line for spaces that are not necessarily first countable. 
One solution is to consider metrics over a structure other than $\RR$. 
Consider a totally ordered%
\footnote{This means that the order \( \leq_\GG \) is linear and translation-invariant (on both sides).} 
(Abelian) group 
\[
\GG = \langle G, +_\GG, 0_\GG, \leq_\GG \rangle
\]
with \markdef{degree} \( \Deg(\GG) = \kappa \), where \( \Deg(\GG) \) denotes
the coinitiality of the positive cone \( \GG^+ = \{ \varepsilon \in \GG \mid  0_\GG <_\GG \varepsilon \}  \) of \( \GG \).%
\footnote{This is also called the \markdef{base number} of \( \GG \) in \cite{Gal19} and the \markdef{character} of \( \GG \) in~\cite{Sik50}.}
A \markdef{\( \GG \)-metric} on a nonempty space \( X \) is then a function \( d \colon X^2 \to \GG^+\cup\{0_\GG\} \) satisfying the usual rules of a distance function: for all \( x,y,z \in X \)
\begin{itemizenew}
\item
\( d(x,y) = 0_\GG \iff x = y \)
\item
\( d(x,y) = d(y,x) \)
\item
\( d(x,z) \leq_\GG d(x,y) +_\GG d(y,z) \).
\end{itemizenew}
Every \( \GG \)-metric space \( (X,d) \) is naturally equipped with the (\( d \)-)topology generated by its open balls
\[
B_d(x,\varepsilon) = \{ y \in X \mid d(x,y) <_\GG \varepsilon \},
 \]
where \( x \in X \) and \( \varepsilon \in \GG^+ \). If \( X \) is already a topological space, we say that the \( \GG \)-metric \( d \) is compatible with the topology of \( X \) if the latter coincides with the \( d \)-topology. 
A topological space is called \markdef{$\GG$-metrizable} if it admits a compatible $\GG$-metric.

Let \( (X,d) \) be a \( \GG \)-metric space. A sequence%
\footnote{Notice that when speaking about Cauchy sequences and Cauchy-completeness we always refer to sequences of length \( \kappa = \mathrm{Deg}(\GG) \).}
\( (x_i)_{i < \kappa} \) of points from \( X \) is \mbox{(\( d \)-)Cauchy} if
\[
\forall \varepsilon \in \GG^+ \, \exists \alpha < \kappa \, \forall \beta,\gamma \geq \alpha \, (d(x_\beta , x_\gamma) <_\GG \varepsilon).
 \]
The space \( (X,d) \) (or the \( \GG \)-metric \( d \)) is Cauchy-complete if every Cauchy sequence \( (x_i)_{i < \kappa} \) converges to some (necessarily unique) \( x \in X \), that is,
\[
\forall \varepsilon \in \GG^+ \, \exists \alpha < \kappa\, \forall \beta \geq \alpha \, (d(x_\beta , x) <_\GG \varepsilon ).
\]

We are now ready to generalize~\ref{intro:pol1}.

\begin{definition} \label{def:K-Polish}
A space \( X \) is \markdef{\( \GG \)-Polish} if it is completely \( \GG \)-metrizable and
 has weight (equivalently, density character) \(\leq \kappa \).
\end{definition}

\begin{remark}
These definitions are not new. Spaces with generalized metrics taking values in a 
structure different from $\RR$ have been introduced in \cite{Sik50} and have been 
widely studied since then, see for 
example \cite{Shu_Tang1964, ReichelMR0458373, NyikosReichelMR1173266}. To the 
best of our knowledge, the systematic study of \emph{completely} $\GG$-metrizable 
spaces is instead of more recent interest, and so far it has been developed mainly 
in~\cite{Gal19}.
\end{remark}

Clearly, \( \GG \)-Polish spaces are closed under closed subspaces. 
Moreover, the space \( \pre{\kappa}{\kappa} \) (endowed with the bounded topology) is always \( \GG \)-Polish, as witnessed by the \( \GG \)-metric
\begin{equation} \label{eq:metricBaire}
d(x,y) =
\begin{cases}
0_\GG & \text{if }  x= y \\
r_\alpha & \text{if }x \restriction \alpha = y \restriction \alpha \text{ and } x(\alpha) \neq y(\alpha)
\end{cases}
 \end{equation}
where \( (r_\alpha)_{\alpha < \kappa} \) is a strictly decreasing sequence coinitial in \( \GG^+ \) (the choice of such a sequence is irrelevant).
It follows that all closed subspaces of \( \pre{\kappa}{\kappa} \), notably including 
\( \pre{\kappa}{2} \), are \( \GG \)-Polish for any $\GG$ as above. Notice also that 
commutativity of the group operation is not strictly needed in order to define the metric, but it is usually required to 
ensure that \( \GG \) itself form a \( \GG \)-metric space with distance function \( d(x,y) = | x -_\GG y |_{\GG} \).
Sometimes it is further required that \( \GG \) is Cauchy-complete with respect to the above metric: in this case \( \GG \) itself would become \( \GG \)-Polish.

We decided to work with the theory of metrics over a totally ordered Abelian group $\GG$ since it is arguably the most common choice in literature.
However, other choices are possible. For example, Reichel in \cite{ReichelMR0458373} 
studied metrics with values in a totally ordered Abelian semigroup with minimum. 
Coskey and Schlicht in \cite{CoskSchlMR3453772} considered (ultra)metrics with values 
in a linear order (where the operation $+_\GG$ is the \textit{minimum} function). Or 
$\GG$ can be non-Abelian as well. All these choices would essentially lead to the same 
results presented here for Abelian groups: see 
Remark~\ref{rmk:choice_of_structure_for_GG}. The reason why we decided to follow the 
common practice of sticking to totally ordered Abelian groups is that metrics over 
groups grant most of the properties of standard metrics.
For example, it is easy to show that for every $x\in X$ and every sequence 
$( r_\alpha )_{\alpha<\kappa}$ coinitial in $\GG^+$, the family
$\{B_d(x,r_\alpha)\mid \alpha<\kappa\}$ is a local basis of $x$ which is well-ordered by reverse 
inclusion $\supseteq$. If one wants to consider metrics taking values in less structured 
sets, like monoids or semigroups, this condition must be explicitly added to the axioms 
that define the metric (see e.g.\ \cite{ReichelMR0458373}).

We conclude this section by addressing another natural question: is there any advantage 
in choosing a particular totally ordered Abelian group \( \GG \) over the others?
In the countable case, $\RR$ plays a key role among all the possible choices of range for 
the metrics: for example, every connected (real-valued) metric space does not admit a 
metric with range contained in $\QQ$. In the uncountable case, the situation is the 
opposite: different choices of \( \GG \) almost always lead to the same class of spaces, 
making less relevant the actual choice of the range of the metrics.
For example, it is well-known that given an uncountable regular cardinal $\kappa$ 
and two totally ordered Abelian groups $\GG$ and $\GG'$ of degree 
$\Deg(\GG)=\Deg(\GG')=\kappa$, a space of weight $\leq \kappa$ is 
$\GG$-metrizable, if and only if it is $\GG'$-metrizable if and only if it is 
$\kappa$-additive (see Theorem~\ref{thm:charK-metrizable}, which is taken 
from~\cite{Sik50}, but see also~\cite{Shu_Tang1964}). 
In Theorem~\ref{thm:charK-Polish} and Corollary~\ref{cor:Gisirrelevant}, we show that a 
similar statement holds for \emph{completely} \( \GG \)-metrizable spaces, hence the notion of 
$\GG$-Polish as well is independent from the choice of the actual $\GG$.

The fact that there is no preferred structure for the range of our generalized metrics 
implies that every possible generalization-to-level-\( \kappa \) of the reals yields to an 
example of 
$\GG$-Polish 
space (as long as this generalization preserves properties like being Cauchy-complete 
with respect to its canonical metric over itself).
For example, this applies to the long reals introduced by Klaua in \cite{klauaMR0146090} 
and studied by Asper\'o and Tsaprounis in \cite{AsperoTsaprounisMR3780585}, or to the 
generalization of $\RR$ introduced in \cite{Gal19} using the surreal numbers. 
See also \cite{DalesWoodinMR1420859} for other examples of $\GG$-Polish spaces, as 
well as methods to construct Cauchy-complete totally ordered fields.

\subsection{Relationships}

The goal of this subsection is to compare the proposed classes of Polish-like (topological) 
spaces; in Section~\ref{sec:standardBorel} we will extend our analysis to encompass the 
various generalizations of standard (\( \kappa \)-)Borel spaces.

\begin{definition}
Let \( X \) be a  space. A set \( A \subseteq X \) is \( G^\kappa_\delta \) if it can be written as a \( \kappa \)-sized intersection of open sets of \( X \); it is \( F^\kappa_\sigma \) if it can be written as a \( \kappa \)-sized union of closed sets of \( X \) or, equivalently, if its complement is \( G^\kappa_\delta \).
A $G^\kappa_\delta$ set is called \markdef{proper} if it is not $F^\kappa_\sigma$.
\end{definition}

It is easy to construct $\SFC_\kappa$-subspaces of, say, the generalized Cantor space \( \pre{\kappa}{2} \) which are properly \( G^\kappa_\delta \): consider e.g.\ 
\begin{equation} \label{eq:Gdeltaofcantor} \{ x \in \pre{\kappa}{2} \mid \forall \alpha \exists \beta \geq \alpha \, (x(\beta) = 1) \} . 
\end{equation}
As in the classical case, this specific example is particularly relevant.

\begin{fact}\label{fct:Baire_is_G_delta_of_Cantor}
The generalized Baire space $\pre{\kappa}{\kappa}$ is homeomorphic to the $G^\kappa_\delta$ subset of  $\pre{\kappa}{2}$ from~\eqref{eq:Gdeltaofcantor}.
\end{fact}

The following is a well-known fact, but we reprove it here for the reader's convenience.

\begin{lemma}\label{lem:closed_are_G_delta}
Every closed subset \( C \) of a space\footnote{Recall that all spaces are tacitly assumed to be regular Hausdorff.} $X$ of weight $\leq \kappa$ is $G^\kappa_\delta$ in $X$.
\end{lemma}

\begin{proof}
Let $\B$ be a basis for \( X \) of size $\leq \kappa$.
By regularity of \( X \), for every $x\in X\setminus C$ there is $U\in \B$ such that $x\in U$ and $\cl(U)\subset X\setminus C$.
Thus
\[ 
C=\bigcap \{X\setminus \cl(U)\mid {U \in \B} \wedge {\cl(U) \cap C = \emptyset}\}. \qedhere
 \] 
\end{proof}

\begin{proposition} \label{prop:G_deltasubspaces}
If \( X \) is an  \( \SFC_\kappa \)-space and \( Y \subseteq X \) is \( G^\kappa_\delta \), then \( Y \) is an  \( \SFC_\kappa \)-space as well.
\end{proposition}

\begin{proof}
Let \( O_\alpha \subseteq X \) be open sets such that \( Y = \bigcap_{\alpha < \kappa} O_\alpha \) and fix a winning strategy \( \tau \) for \( \pII \) in \( \fCho^s_\kappa(X) \).
We define (by  recursion on the round) a strategy for \( \pII \) in \( \fCho^s_\kappa(Y) \) as follows. 
Suppose that until a certain round $\alpha<\kappa$, player $\pI$ has played a sequence \( \langle  ({U}_\beta, x_\beta) \mid \beta\leq \alpha \rangle \)  following the rules of \( \fCho^s_\kappa(Y) \).
Each set \( U_\beta \) is open in $Y$ relatively to the intersection of all previous moves, 
hence it can be seen as the intersection of $Y$ (and all previous moves of $\pI$) with 
some  open set of $X$. Proceeding recursively, we can thus associate to each $U_\beta$ 
a set $\tilde{U}_\beta \subseteq O_\beta$ such that \( U_\beta = \tilde{U}_\beta \cap Y \), 
where $\tilde{U}_\beta$ is open in $X$ relatively to the intersection 
$\bigcap_{\zeta<\beta}\tilde{U}_\zeta$ of all previous sets  (this can be done because 
each \( O_\beta \) is open in $X$).
Then \( \langle  (\tilde{U}_\beta, x_\beta) \mid \beta\leq \alpha \rangle \) is a legal sequence of moves for $\pI$ in \( \fCho^s_\kappa(X) \). 
If \( V_\alpha \) is what \(\tau\) requires \( \pII \) to play against 
\( \langle  (\tilde{U}_\beta , x_\beta) \mid \beta\leq \alpha \rangle \) in 
\( \fCho^s_\kappa(X) \), we get that \( V_\alpha \cap Y \neq \emptyset \), as witnessed by 
\( x_\alpha \) itself, and \( V_\alpha \subseteq \tilde{U}_\alpha \subseteq O_\alpha \): so 
we can let \( \pII \) respond to \( \pI \)'s move in the game \( \fCho^s_\kappa(Y) \) on 
\( Y \) with \( V_\alpha \cap Y \). 
By construction, the resulting strategy for \( \pII \) is legal with respect to the rules of 
\( \fCho^s_\kappa(Y) \). Moreover, if for all limit \( \gamma < \kappa \) the intersection 
\( \bigcap_{\alpha < \gamma} (V_\alpha \cap Y) \) is nonempty, then so is 
\( \bigcap_{\alpha< \gamma} V_\alpha \):
since \( \tau \) is winning in \( \fCho^s_\kappa(X) \), this means that \( \bigcap_{\alpha< \kappa} V_\alpha \neq \emptyset \), whence by \( V_\alpha \subseteq O_\alpha \) we also get
\[
\bigcap_{\alpha < \kappa} (V_\alpha \cap Y) = \bigg(\bigcap_{\alpha < \kappa} V_ \alpha \bigg) \cap Y =  \bigcap_{\alpha < \kappa} V_ \alpha  \cap  \bigcap_{\alpha < \kappa} O_\alpha = \bigcap_{\alpha< \kappa} V_\alpha \neq \emptyset. \qedhere
 \]
\end{proof}

Recall that given an infinite cardinal \( \nu \), a topological space \( X \) is called \markdef{\( \nu \)-additive} if its topology is closed under intersections of  length \( < \nu \).
Every topological space is \( \omega \)-additive, and as already noticed the generalized Baire 
and Cantor spaces \( \pre{\kappa}{\kappa}, \pre{\kappa}{2} \) are both 
\( \kappa \)-additive when \( \kappa \) is regular. Moreover, if \( X \) is regular and 
\( \nu \)-additive for some \( \nu > \omega \), then \( X \) is zero-dimensional  (i.e.\ it has a 
basis consisting of clopen sets). Indeed, fix a point \( x \in X \) and an open neighborhood 
\( U \) of it. Using regularity, recursively construct a sequence \( (U_n)_{n \in \omega} \) of 
open neighborhoods of \( x \) such that \( U_0 = U \) and \( \cl (U_{n+1}) \subseteq U_n \). 
Then \( V = \bigcap_{n \in \omega} U_n = \bigcap_{n \in \omega} \cl(U_n) \) contains 
\( x \), it is closed, and it is also open by \( \nu \)-additivity (here we use 
\( \nu > \omega \)). Thus \( X \) admits a basis consisting of clopen sets, as required. Notice 
also that if \( X \) has weight \( \kappa \), then such a clopen basis can be taken of size 
\( \kappa \) as well.

Recall also the correspondence between closed subsets of \( \pre{\kappa}{\kappa} \) and 
trees on \( \kappa \). Given an ordinal  \(\gamma\) and a nonempty set \( A \), we denote 
by \( \pre{\gamma}{A} \) the set of all sequences of length \( \gamma \) and values in 
\( A \). We then set \( \pre{< \kappa}{\kappa} = \bigcup_{\gamma < \kappa} \pre{\gamma}{\kappa} \), and for \( s \in \pre{< \kappa}{\kappa} \) we let \( \leng(s) \) be the length of 
\( s \), that is, the unique ordinal \( \gamma < \kappa \) such that \( s \in \pre{\gamma}{\kappa} \). The concatenation between two sequences \( s,t \) is denoted by 
\( s {}^\smallfrown{} t \), and to simplify the notation we just write 
\( s {}^\smallfrown{} i \) and \( i {}^\smallfrown{} s \) if \( t = \langle i \rangle \) is a 
sequence of length \( 1 \). 
If \( \alpha \leq \leng(s) \), we denote by \( s \restriction \alpha \) the restriction of \( s \) 
to its first \( \alpha \)-many digits. We write \( s \subseteq t \) to say that \( s \) is an initial 
segment of \( t \), that is, \( \leng(s) \leq \leng(t) \) and \( s = t \restriction \leng(s) \). The 
sequences \( s \) and \( t \) are \markdef{comparable} if \( s \subseteq t \) or 
\( t \subseteq s \), and \markdef{incomparable} otherwise.
A set \( T \subseteq \pre{< \kappa}{\kappa} \) is called \markdef{tree} if it is closed under 
initial segments. For \( \alpha < \kappa \) we denote by \( \mathrm{Lev}_\alpha(T) \) the 
\(\alpha\)-th level of the tree \( T \), namely,
\[ 
\mathrm{Lev}_\alpha(T) = \{ t \in T \mid \leng(t) = \alpha \} . 
\]
Given \( s \in T \), we also define the localization of \( T \) at \( s \) as 
\[ 
T_s = \{  t \in T \mid t \text{ is comparable with } s \} . 
\]
The bounded topology on \( \pre{\kappa}{\kappa} \) is the unique topology on such a space 
with the following property:
a set \( C \subseteq \pre{\kappa}{\kappa} \) is closed if and only it there is some tree
\( T \subseteq \pre{<\kappa}{\kappa} \) such that \( C = [T] \), where  the \markdef{body} \( [T] \) of the tree \( T \) is defined by
\[
[T] = \{  x \in \pre{\kappa}{\kappa} \mid \forall \alpha < \kappa \, ( x \restriction \alpha \in T) \}.
 \]
The above tree  \( T \) can always be required to be \markdef{pruned}, that is, such that for all 
\( s \in T \) and \( \leng(s) \leq \alpha < \kappa \) there is \( s \subseteq t \in T \) such that 
\( \leng(t) = \alpha \), i.e.\ \( \mathrm{Lev}_\alpha(T_s) \neq \emptyset \) for all \( \alpha < \kappa \). Indeed, if \( C \) is closed, then the tree 
\( T_C = \{ x \restriction \alpha \mid x \in C \wedge \alpha < \kappa \} \) is pruned 
and such that \( C = [T_C] \). Sometimes, one needs to consider a further closure property for trees. We say that the tree \( T \) is \markdef{\( {<} \kappa \)-closed} if for all 
sequences \( s \in \pre{\gamma}{\kappa} \) with \( \gamma < \kappa \) limit, if  
\( s \restriction \alpha \in T \) for all \( \alpha < \gamma \), then \( s \in T \) as well. A tree 
\( T \) is called \markdef{superclosed} if it is pruned and \( {<} \kappa \)-closed; this in 
particular implies that if \( s \in T \), then \( \Nbhd_s \cap [T] \neq \emptyset \) or, 
equivalently, \( [T_s ] \neq \emptyset \). Not all closed subsets of \( \pre{\kappa}{\kappa} \) are the body of a superclosed tree: consider e.g.\ the set 
\begin{equation} \label{eq:finitelymanyzeroes}
X_0 =  \{ x \in \pre{\kappa}{2} \mid |\{ \alpha<\kappa \mid x(\alpha) = 0 \}| < \aleph_0 \}  .
 \end{equation}
This justifies the following terminology: a closed \( C \subseteq \pre{\kappa}{\kappa} \) is called \markdef{superclosed} if \( C = [T] \) for some  superclosed tree \( T \).

Sikorski proved in~\cite[Theorem (x)]{Sik50} that every regular $\kappa$-additive space 
of weight $\leq \kappa$ is homeomorphic to a subspace of $\pre{\kappa}{2}$, and that 
the latter is $\GG$-metrizable. 
We can sum up his results as follows, where we additionally
use Fact~\ref{fct:Baire_is_G_delta_of_Cantor} to further add item~\ref{thm:charK-metrizable-4} to the list of equivalent conditions. 

\begin{theorem}[{\cite[Theorem (viii)-(x)]{Sik50}}] \label{thm:charK-metrizable}
For any space \( X \) of weight $\leq \kappa$ 
and any totally ordered Abelian group \( \GG \) with \( \Deg(\GG) = \kappa \)
the following are equivalent:
\begin{enumerate-(a)}
\item\label{thm:charK-metrizable-1}
\( X \) is \( \kappa \)-additive;
\item\label{thm:charK-metrizable-2}
\( X \) is \( \GG \)-metrizable;
\item \label{thm:charK-metrizable-3}
\( X \) is homeomorphic to a subset of \( \pre{\kappa}{2} \);
\item \label{thm:charK-metrizable-4}
\( X \) is homeomorphic to a subset of \( \pre{\kappa}{\kappa} \).
\end{enumerate-(a)}
\end{theorem}

Since conditions~\ref{thm:charK-metrizable-1}, \ref{thm:charK-metrizable-3}, and~\ref{thm:charK-metrizable-4} do not refer to \( \GG \) at all, 
this shows in particular that the choice of the actual group in the definition of the generalized metric is irrelevant. 
We are now going to prove that analogous results hold also for \( \SFC_\kappa \)-spaces, \( \SC_\kappa \)-spaces, and $\GG$-Polish spaces (see Theorems~\ref{thm:charK-Polish} and~\ref{thm:SC-space+G-Polish}).

\begin{proposition}
 \label{prop:charadditiveSCandSFCkappaspaces}
Let $X$ be a \( \kappa \)-additive \( \SFC_\kappa \)-space. Then \( X \) is homeomorphic to a closed \( C \subseteq \pre{\kappa}{\kappa} \).
If furthermore $X$ is an \( \SC_\kappa \)-space, then \( C \) can be taken to be superclosed.
\end{proposition}

\begin{proof}
We prove the two statements simultaneously.
Let \( ( B_\alpha )_{ \alpha < \kappa } \) be an enumeration of a clopen basis \( \mathcal{B} \) of \( X \), possibly with repetitions.
Depending on whether \( X \) is an \( \SC_\kappa \)-space or just an \( \SFC_\kappa \)-space, let \(\sigma\) be a winning strategy for player \( \pII \) in \( \Cho^s_\kappa(X) \) or \( \fCho^s_\kappa(X) \).
By Remark~\ref{rmk:strategyrangeinbasis},
without loss of generality we can assume that the range of \(\sigma\) is contained in \( \mathcal{B} \).
To simplify the notation, given an ordinal \( \beta \), let \( \Succ(\beta) \) be the collection of all successor ordinals \( \leq \beta \). Set also
\[
\pre{<\Succ(\kappa)}{\kappa} = \{ s \in \pre{< \kappa}{\kappa} \mid \leng(s) \in \Succ(\kappa) \}.
 \]
We will construct a family of the form
\[
\mathcal{F}=\left \{  \langle x_s, U_s, V_s, \hat{V}_s \rangle \mid s \in \pre{<\Succ(\kappa)}{\kappa} \right\},
\]
and set for every \( t \in \pre{\leq \kappa}{\kappa} = \pre{< \kappa}{\kappa} \cup \pre{\kappa}{\kappa} \) with \( \leng(t) = \gamma \leq \kappa \),
\begin{equation}  \label{eq:V(T)}
V(t) = \bigcap_{\alpha \in \Succ(\gamma)}\hat{V}_{t \restriction \alpha}.
 \end{equation}
(In particular, when \( \gamma = 0 \) we get \( V(\emptyset) = X \) because \( \Succ(0) = \emptyset \).) The family \( \mathcal{F} \) will be designed so that for any \( \gamma < \kappa \) and \( s \in \pre{\gamma+1}{\kappa} \) the following properties are satisfied:

\begin{enumerate-(i)}
\item \label{def}
 $x_s\in X$, and $U_s$, $V_s$, $\hat{V}_s$ are all clopen in \( X \).
\item \label{partialplay}
If \( V(s ) \neq \emptyset \),
then the sequence
\( \langle (U_{s \restriction \alpha}, x_{s \restriction \alpha} ), V_{s \restriction \alpha} \mid \alpha\in\Succ(\gamma+1) \rangle \)
is a (partial) run in the strong (fair) \( \kappa \)-Choquet game on \( X \) in which \( \pII \) follows $\sigma$. 
\item \label{compatibility}
Either
$ \hat{V}_s \subseteq B_\gamma $ or $\hat{V}_s \cap B_\gamma = \emptyset $.
\item \label{chain}
\( \hat{V}_s\subseteq V_s \subseteq  U_s \subseteq V(s \restriction \gamma) \).
\item\label{partition}
$\{ \hat{V}_s \mid s \in \pre{\gamma+1}{\kappa} \} $ is a partition%
\footnote{In our terminology, an indexed family \( \{ A_i \mid i \in I \} \) of subsets of \( X \) is a partition of \( X \) if \(\bigcup_{i \in I} A_i = X \) and \( A_i \cap A_j = \emptyset \) for distinct \( i,j \in I \). 
In particular, some of the \( A_i \)'s might be empty and for \( i \neq j \) we have \( A_i = A_j \) if and only if both \( A_i \) and \( A_j \) are empty.}
 of \( X \).
\end{enumerate-(i)}

Condition~\ref{chain} implies that 
\begin{equation} \label{eq:monotone} 
\hat{V}_s \subseteq \hat{V}_{s \restriction \alpha} 
\end{equation} 
for every \( s \in \pre{<\Succ( \kappa)}{\kappa} \) and \( \alpha \in \Succ(\leng(s)) \). Together with condition~\ref{partition}, this entails that
\begin{enumerate}[label={\upshape (v')}, leftmargin=2pc]
\item \label{partition'}
For any $ \gamma < \kappa $, successor or not, $\{ V(t) \mid t \in \pre{\gamma}{\kappa} \} $ is a partition of \( X \).
\end{enumerate}
From condition~\ref{partition'} and~\eqref{eq:monotone} it easily follows that if \( t,t' \in \pre{<\kappa}{\kappa} \) are such that \( V(t) \cap V(t') \neq \emptyset \), then \( t  \) and \( t' \) are comparable.
The inclusion~\eqref{eq:monotone} also implies that  if \( \leng(t) \) is a successor ordinal, then \( V(t) = \hat{V}_t \). If instead \( \gamma =  \leng(t) \leq \kappa \) is limit,  then
\begin{equation} \label{eq:V(t)limitcase}
V(t) =\bigcap_{\alpha\in \Succ(\gamma)} U_{t \restriction \alpha}=\bigcap_{\alpha\in \Succ(\gamma)} V_{t \restriction \alpha}
\end{equation}
by condition~\ref{chain} again. Notice also that the additional properties discussed in 
this paragraph have a local (i.e.\ level-by-level) nature: for example, to 
have~\ref{partition'} at some level \( \gamma \), it is enough to have 
conditions~\ref{chain} and~\ref{partition} at all levels  \( \gamma' \leq \gamma \).

Given \( \mathcal{F} \) as above, one obtains  the required homeomorphism of \( X \) with 
a (super)closed set \( C \subseteq \pre{\kappa}{\kappa} \) as follows.
Since \( X \) is Hausdorff, if \( \leng(t) = \kappa \) then $V(t)$ has at most one element by 
condition~\ref{compatibility}.
Consider the tree
\[
T = \{ t \in \pre{< \kappa}{\kappa} \mid V(t) \neq \emptyset \}.
\]
It is pruned by condition~\ref{partition'} and the comment following it. Furthermore, if 
\( X \) is an \( \SC_\kappa \)-space (i.e.\ \(\sigma\) is a  winning in the game 
\( \Cho^s_\kappa(X) \)), then \( T \) is also \( {<} \kappa \)-closed by 
condition~\ref{partialplay} and equation~\eqref{eq:V(t)limitcase}. 

We now prove that the (super)closed set \( C = [T] \) is homeomorphic to \( X \).
Since \(\sigma\) is a winning strategy in the strong (fair) $\kappa$-Choquet game,   the 
set \( V(t)\) is nonempty for every $t\in [T]$ by condition~\ref{partialplay} and 
equation~\eqref{eq:V(t)limitcase} again, thus it contains exactly one point: let 
$f \colon [T] \to X$ be the map that associates to every $t\in [T]$ the unique element in $V(t)$.
We claim that $f$ is a homeomorphism.

\begin{claim} \label{claim:bijection}
$f$ is bijective.
\end{claim}

\begin{proof}
To see that $f$ is injective, let \( t , t' \in [T] \) be distinct and \( \alpha< \kappa \) be such 
that \( t \restriction \alpha \neq t' \restriction \alpha \). By condition~\ref{partition'}
we have
\( V(t \restriction \alpha) \cap V(t' \restriction \alpha) = \emptyset \), and hence 
\( f(t) \neq f(t') \) because  \( f(t) \in V(t) \subseteq V(t \restriction \alpha) \) and 
\( f(t') \in V(t') \subseteq V(t' \restriction \alpha) \). To see that \( f \) is also surjective, fix 
any \( x \in X \).
By~\ref{partition'} again (and the comment following it), for each \( \alpha < \kappa \) 
there is a unique \( t_{\alpha}  \) of length \(\alpha\) with \( x \in V(t_\alpha) \), and 
moreover $t_\alpha\subseteq t_\beta$ for all $\alpha \leq \beta < \kappa$.
Let $t = \bigcup_{\alpha<\kappa} t_\alpha$, so that \( x \in V(t) = \bigcap_{\alpha < \kappa} V(t_\alpha) =
\bigcap_{\alpha < \kappa} V(t \restriction \alpha) \): then \( x \) itself witnesses \( t \in [T] \), and  $f(t)=x$.
\end{proof}

\begin{claim}
\( f \) is a homeomorphism.
\end{claim}

\begin{proof}
Observe that  by definition of $f$, its surjectivity,  and condition~\ref{partition'}, 
\begin{equation} \label{eq:preimage}
f (\Nbhd_s \cap [T] ) = V(s) =  \hat{V}_s
\end{equation}
for all $s\in T$ with \( \leng(s) \in \Succ(\kappa) \). Since 
\( \{ \Nbhd_s \cap [T] \mid s\in T \cap \pre{<\Succ(\kappa)}{\kappa} \} \) is a basis for the 
relative topology of \( [T] \), and since $\{\hat{V}_s\mid s\in T \cap \pre{<\Succ(\kappa)}{\kappa} \}$ is a basis for $X$ by conditions~\ref{def},~\ref{compatibility}, and \ref{partition}, then $f$ and $f^{-1}$ are continuous.
\end{proof}

It remains to construct the required family \( \mathcal{F} \) by recursion on \( \gamma < \kappa \).
We assume that for every \( t \in \pre{< \kappa}{\kappa} \) with \( \leng(t) = \gamma \) 
and all \( \alpha \in \Succ(\gamma) \), the elements \( x_{t \restriction \alpha} \), 
\( U_{t \restriction \alpha} \), \( V_{t \restriction \alpha} \), and \( \hat{V}_{t \restriction \alpha} \) have been defined so that conditions~\ref{def}--\ref{partition} are satisfied up 
to level \( \gamma \) (when \( \gamma > 0 \) this is the inductive hypothesis, while if 
\( \gamma = 0 \) the assumption is obviously vacuous because \( \Succ(0) \) is empty): our 
goal is to define \( x_{t {}^\smallfrown{} i}\), \( U_{t {}^\smallfrown{} i} \), 
\( V_{t {}^\smallfrown{} i} \), and \( \hat{V}_{t {}^\smallfrown{} i} \) for all \( t \) as above 
and \( i < \kappa \) in such a way that conditions~\ref{def}--\ref{partition} are preserved.

Recall the definition of the sets \( V(t) \) from equation~\eqref{eq:V(T)}.
If \( V(t) = \emptyset \), then we set \( U_{t {}^\smallfrown{} i} = V_{t {}^\smallfrown{} i} = \hat{V}_{t {}^\smallfrown{} i} = \emptyset \) for all \( i < \kappa \) and let \( x_{t {}^\smallfrown{} i} \) be an arbitrary point in \( X \).
Assume now that \( V(t) \neq \emptyset \).
Notice that \( V(t) \) is clopen:
if \( \gamma > 0 \) this follows from \( \kappa \)-additivity of \( X \) and the fact that \( \hat{V}_{t \restriction \alpha} \) is clopen for every \( \alpha\in \Succ(\gamma) \) by~\ref{def}, while if \(\gamma = 0 \) then \( V(\emptyset) = X \) by definition.  
By condition~\ref{partialplay}, the sequence \( \langle (U_{t \restriction \alpha}, x_{t \restriction \alpha} ), V_{t \restriction \alpha} \mid \alpha\in\Succ(\gamma) \rangle \) is a partial run in the corresponding Choquet-like game in which \( \pII \) is following \(\sigma\).
 We let such run continue for one more round by letting \( \pI \) play some \( (U,x) \) with 
 \( U \) clopen and \( x \in U \subseteq V(t) \), and \( \pII \) reply with some  
 \( V \in \mathcal{B} \) following the winning strategy \(\sigma\), so that in particular 
 \( x \in V \subseteq U \). Let \( \{ V_j \mid j < \delta \} \) be the collection of all those sets 
 \( V \) that can be obtained in this way: even if there are possibly more than 
 \( \kappa \)-many moves for \( \pI \) as above, there are at most \( \kappa \)-many replies 
 of \( \pII \) because \( |\mathcal{B}| \leq \kappa \), hence $\delta\leq \kappa$. For each 
 \( j < \delta \) we then choose one of player \( \pI \)'s moves \( (U_j,x_j) \) yielding \( V_j \) 
 as \( \pII \)'s reply. In particular, \( x_j \in V_j \subseteq U_j \). Let \( ( \hat{V}_i )_{ i < \nu } \) (where \( \nu \leq \kappa \)) be an enumeration without repetitions of the nonempty 
 sets in
 \[
\left \{ \left( V_j \setminus \bigcup\nolimits_{\ell < j} V_\ell \right) \cap B_\gamma \,\, \middle| \,\, j < \delta \right \} \cup \left \{ \left( V_j \setminus \bigcup\nolimits_{\ell < j} V_\ell \right) \setminus B_\gamma \,\, \middle| \,\, j < \delta \right \},
 \]
and for each \( i < \nu \) let \( j(i) < \delta \leq \kappa \) be such that \( \hat{V}_i \subseteq V_{j(i)} \). 
Notice that the \( \hat{V}_i \)'s are clopen by \( \kappa \)-additivity again. Finally, set
\begin{align*}
x_{t {}^\smallfrown{} i} & = x_{j(i)} &&&
U_{t {}^\smallfrown{} i} & = U_{j(i)} &&&
V_{t {}^\smallfrown{} i} & = V_{j(i)} &&&
\hat{V}_{t {}^\smallfrown{} i} & = \hat{V}_i
\end{align*}
if \( i < \nu \), and  \( U_{t {}^\smallfrown{} i} = V_{t {}^\smallfrown{} i} = \hat{V}_{t {}^\smallfrown{} i} = \emptyset \) with \( x_{t {}^\smallfrown{} i} \) an arbitrary point of \( X \) if \( \nu \leq i < \kappa \).

It is not hard to see that conditions~\ref{def}--\ref{chain} are preserved by  construction.
As for  condition~\ref{partition}, by inductive hypothesis (or \( V(\emptyset) = X \) if \( \gamma = 0 \)) we get~\ref{partition'} at level \( \gamma \), that is, 
\( \{ V(t) \mid  t \in \pre{\gamma}{\kappa} \} \)
is a partition of \( X \). Thus the desired result straightforwardly follows from the fact 
that the \( V_j \)'s cover \( V(t) \) because in our construction player \( \pI \) can play any 
\( x \in V(t) \) in her last round (paired with a suitable clopen set \( U \) such that 
\( x \in U \subseteq V(t) \), which exists because \( V(t) \) is clopen).
\end{proof}

We now consider the problem of simultaneously embedding two \( \kappa \)-additive \( \SFC_\kappa \)-spaces \( X' \subseteq X \) into \( \pre{\kappa}{\kappa} \). Applying Proposition~\ref{prop:charadditiveSCandSFCkappaspaces} to \( X \) we get a closed \( C \) and a homeomorphism \( f \colon C \to X \). If $X'$ is a closed in $X$, it follows that also \( C' = f^{-1}(X') \) is closed in \( C \) and hence in \( \pre{\kappa}{\kappa} \). However, when \( X' \) is an \( \SC_\kappa \)-space we would like to have that \( C' \) is superclosed. 
To this aim we need to modify our construction.

\begin{proposition} \label{prop:simultaneousembeddings}
Let \( X \) be a \( \kappa \)-additive \( \SFC_\kappa \)-space and \( X' \subseteq X \) be a closed \( \SC_\kappa \)-subspace. Then there is a closed \( C \subseteq \pre{\kappa}{\kappa} \) and a homeomorphism \( f \colon C \to X \) such that \( C' = f^{-1}(X')  \) is superclosed.
 \end{proposition}
 
\begin{proof}
The idea is to apply the argument from the previous proof but starting with a strategy $\sigma$ that is winning for \( \pII \) in $\fCho_\kappa^s(X)$ and, when ``restricted'' to $X'$, it is also winning in $\Cho_\kappa^s(X')$. Let \( \B \) be a basis for \( X \) of size \( \leq \kappa \).

\begin{claim}\label{claim:common_strategy}
There is a winning strategy $\sigma$ for player $\pII$ in $\fCho^s_\kappa(X)$ with range in $\B$ such that for any $\gamma\leq \kappa$ and for any (possibly partial) run $\langle (U_\alpha,x_\alpha), V_\alpha\mid \alpha< \gamma\rangle$ in $\fCho^s_\kappa(X)$ where player $\pII$ followed $\sigma$, one has $\bigcap_{\alpha<\gamma} V_\alpha\cap X'\neq \emptyset$ if and only if $V_\alpha\cap X'\neq\emptyset$ for every $\alpha<\gamma$.
\end{claim}

\begin{proof}[Proof of the claim]
Let \( \sigma' \) be an arbitrary winning strategy for \( \pII \) in \( \Cho^s_\kappa(X') \), and let \( \sigma'' \) be a winning strategy for \( \pII \) in \( \fCho^s_\kappa(X) \) with range contained in 
\( \B \). Define the strategy \(\sigma\) as follows. Suppose that at stage \( \alpha < \kappa \) player \( \pI \) has played the sequence \( \langle (U_\beta, x_\beta) \mid \beta \leq \alpha \rangle \) in the game \( \fCho^s_\kappa(X) \).  
\begin{enumerate-(1)}
\item \label{case1mixingstrategies}
As long as all points \( x_\beta \) belongs to \( X' \),  player \( \pII \) considers the auxiliary partial play \mbox{\( \langle (U_\beta  \cap X', x_\beta) \mid \beta \leq \alpha \rangle \)} of \( \pI \) in \( \Cho^s_\kappa(X') \) and she uses \( \tau' \) to get her next move \( V'_\alpha \) in the game \( \Cho^s_\kappa(X') \). Since \( V'_\alpha \) is open in \( X' \), there is \( W \) open in \( X \) such that \( V'_\alpha = W \cap X' \): let \( \pII \) play  any \( V_\alpha \in \mathcal{B} \) such that \( x_\alpha \in V_\alpha \subseteq W \cap  \bigcap_{\beta \leq \alpha} U_\beta \) as her next move in the game \( \fCho^s_\kappa(X) \) (this is possible because \( W \cap  \bigcap_{\beta \leq \alpha} U_\beta \) is open by \( \kappa \)-additivity).
\item \label{case2mixingstrategies}
If \( \alpha \) is smallest such that \( x_\alpha \notin X' \), from that point on player \( \pII \) uses her strategy \( \sigma''\) pretending that \( (U_\alpha \setminus X',x_\alpha) \) was the first move of \( \pI \) in a new run of \( \fCho^s_\kappa(X) \).
\end{enumerate-(1)}
We claim that \(\sigma\) is as required, so
fix any \( \gamma \leq \kappa \). Let \( \langle (U_\alpha,x_\alpha), V_\alpha\mid \alpha< \gamma\rangle \) 
be a partial run in which \( \pII \) followed \(\sigma\) and assume that \( V_\alpha \cap X' \neq \emptyset \) for every \( \alpha < \gamma \). By~\ref{case2mixingstrategies} this implies 
that \( x_\alpha \in X' \) for all \( \alpha < \gamma \). If \( \gamma = \alpha+1 \) is a 
successor ordinal, then \( \bigcap_{\beta<\gamma} V_\beta\cap X' = V_\alpha \cap X' \neq \emptyset \) by assumption. Assume instead that \(\gamma\) is limit. By \( x_\alpha \in X' \) and~\ref{case1mixingstrategies}, for all \( \alpha < \gamma \) 
we have
\begin{equation} \label{eq:mixedstrategies}  
U_{\alpha+1}  \cap X' \subseteq V_\alpha \cap X' \subseteq V'_\alpha \subseteq U_\alpha \cap X' , 
\end{equation}
where \( V'_\alpha \subseteq X' \) is again \(\pII \)'s reply to the 
partial play \( \langle (U_\beta  \cap X', x_\beta) \mid \beta \leq \alpha \rangle \) of 
\( \pI \) in \( \Cho^s_\kappa(X') \) according to \( \sigma' \). It follows that 
\( \langle (U_\alpha \cap X' ,x_\alpha), V'_\alpha\mid \alpha< \gamma\rangle \)
is a (legal) 
partial run in \( \Cho^s_\kappa(X') \) where \( \pII \) followed \( \sigma' \), and since the 
latter is winning in such game we get \( \bigcap_{\alpha < \gamma} V_\alpha \cap X' = \bigcap_{\alpha < \gamma} V'_\alpha \neq \emptyset \) (the first equality follows from~\eqref{eq:mixedstrategies} and the fact that \(\gamma\) is limit). This also implies that \(\sigma\) wins \( \fCho^s_\kappa(X) \) in all runs where \( V_\alpha \cap X' \neq \emptyset \) for all \( \alpha < \kappa \); on the other hand, when this is not the case and \( \alpha < \kappa \) is smallest such that \( V_\alpha \cap X' = \emptyset \), then the tail of the run from level \( \alpha \) on is a (legal) run in \( \fCho^s_\kappa(X) \) in which \( \pII \) followed \( \sigma'' \), thus \( \pII \) won as well. This shows that \( \sigma \) is winning for \( \pII \) in \( \fCho^s_\kappa(X) \) and concludes the proof.
\end{proof}

Starting from \( \sigma \) as in Claim~\ref{claim:common_strategy}, argue as 
in the proof of Proposition~\ref{prop:charadditiveSCandSFCkappaspaces} to 
build a family \( \mathcal{F}=\left \{ \langle x_s, U_s, V_s, \hat{V}_s \rangle \mid s \in \pre{<\Succ(\kappa)}{\kappa} \right\} \) 
and a homeomorphism 
\( f \colon C \to X \), where \( C = [T] \) is the closed subset of $\pre{\kappa}{\kappa}$ defined by the tree \( T = \{ t \in \pre{< \kappa}{\kappa} \mid V(t) \neq \emptyset \} \), and \( f(t) \) is the unique point in \( V(t) \) for all \( t \in [T] \). 
Consider now the tree  defined by
\[ 
T' = \{ t \in \pre{< \kappa}{\kappa} \mid V(t) \cap X' \neq \emptyset \}.
 \] 
Clearly \( T' \subseteq T \). Moreover, for every $t\in T'$ we have $\Nbhd_t\cap [T']\neq\emptyset$: indeed, if \( t \in T' \), then there is 
\( x \in V(t) \cap X' \), hence \( f^{-1}(x) \supseteq t \) and by construction \( x \) witnesses 
\( f^{-1}(x) \restriction \alpha \in T' \) for all \( \alpha < \kappa \), so $f^{-1}(x)\in \Nbhd_t\cap [T']$. In particular, this implies that \( T' \) is pruned. We now prove that 
\( T' \) is also superclosed. Let \( t \in \pre{\gamma}{\kappa} \) for \(  \gamma < \kappa \) 
limit be such that \( t \restriction \alpha \in T' \) for all \( \alpha < \gamma \). Then 
\( \hat{V}_{t \restriction \alpha} \cap X' \neq \emptyset \) for all 
\( \alpha \in \Succ(\gamma) \), hence also \( V_{t \restriction \alpha} \cap X' \neq \emptyset \) by \( \hat{V}_{t \restriction \alpha} \subseteq V_{t \restriction \alpha} \). 
By the choice of $\sigma$, it follows that 
\( \bigcap_{\alpha \in \Succ(\gamma)} V_{ t \restriction \alpha} \cap X' \neq \emptyset \), 
hence \( t \in T' \) since \( V(t) = \bigcap_{\alpha \in \Succ(\gamma)} V_{t \restriction \alpha} \) when $t$ has limit length. 

Finally, we want to show that \( f^{-1}(X') = [T'] \). Given \( x \in X' \), then \( x \) itself 
witnesses \( f^{-1}(x) \in [T'] \). Conversely, if \( t \in [T'] \) then \( V_{t \restriction \alpha} \cap X' \supseteq V(t \restriction \alpha) \cap X' \neq \emptyset \) for all \( \alpha \in \Succ(\kappa) \), hence  
by the choice of $\sigma$ again we have that \( \bigcap_{\alpha \in \Succ(\kappa)} V_{t \restriction \alpha} \cap X' \neq \emptyset  \). 
Since \(  \bigcap_{\alpha \in \Succ(\kappa)} V_{t \restriction \alpha} = V(t) = \{ f(x) \} \), 
it follows that \( f(x) \in X' \) as desired.
\end{proof}

Proposition~\ref{prop:simultaneousembeddings} allows us to considerably extend~\cite[Proposition 1.3]{LuckeSchlMR3430247} from superclosed subsets of \( \pre{\kappa}{\kappa} \) to arbitrary closed \( \SC_\kappa \)-subspaces of a \( \kappa \)-additive \( \SFC_\kappa \)-space.

\begin{corollary} \label{cor:simultaneousembeddings}
Let \( X \) be a \( \kappa \)-additive \( \SFC_\kappa \)-space. Then every closed \( \SC_\kappa \)-subspace \( Y \) of \( X \) is a retract of it.
\end{corollary}

\begin{proof}
By Proposition~\ref{prop:simultaneousembeddings},
without loss of generality we may assume that \( X \) is a closed subspace of \( \pre{\kappa}{\kappa} \) and \( Y \subseteq X \) a superclosed set. By~\cite[Proposition 1.3]{LuckeSchlMR3430247} there is a retraction \( r \) from \( \pre{\kappa}{\kappa} \) onto \( Y \). Then \( r \restriction X \) is a retraction of \( X \) onto \( Y \).
\end{proof}

None of the conditions on \( Y \) can be dropped in the above result: 
every retract of a Hausdorff space is necessarily closed in it, and 
by~\cite[Proposition 1.4]{LuckeSchlMR3430247} the space \( X_0 \) from equation~\eqref{eq:finitelymanyzeroes} is a closed \( \SFC_\kappa \)-subspace of the \( \SC_\kappa \)-space \( \pre{\kappa}{2} \) which is not a retract of it. Notice also that there are even clopen (hence strong \( \kappa \)-Choquet) subspaces of \( \pre{\kappa}{\kappa} \)  which are not superclosed, for example \mbox{\( \{ x \in \pre{\kappa}{\kappa} \mid \exists n < \omega \, (x(n) \neq 0 ) \} \)}. This shows that even in the special case \( X = \pre{\kappa}{\kappa} \), our Corollary~\ref{cor:simultaneousembeddings} properly extends~\cite[Proposition 1.3]{LuckeSchlMR3430247}.

Lemma~\ref{lem:closed_are_G_delta},
Proposition~\ref{prop:G_deltasubspaces} and 
Proposition~\ref{prop:charadditiveSCandSFCkappaspaces} together lead to the 
following characterization of \( \kappa \)-additive \( \SFC_\kappa \)-spaces.

\begin{theorem}
 \label{thm:charadditivefairSCkappaspaces}
For any space\footnote{Recall that all spaces are tacitly assumed to be regular Hausdorff.} \( X \) the following are equivalent:
\begin{enumerate-(a)}
\item \label{thm:charadditivefairSCkappaspaces-a}
\( X \) is a \( \kappa \)-additive  \( \SFC_\kappa \)-space;
\item \label{thm:charadditivefairSCkappaspaces-b}
\( X \) is homeomorphic to a \( G^\kappa_\delta \) subset of \( \pre{\kappa}{\kappa} \);
\item \label{thm:charadditivefairSCkappaspaces-c}
\( X \) is homeomorphic to a closed subset of \( \pre{\kappa}{\kappa} \).
\end{enumerate-(a)}
In particular, \( \pre{\kappa}{\kappa} \) is universal for \( \kappa \)-additive \( \SFC_\kappa \)-spaces, and hence also for \( \kappa \)-additive \( \SC_\kappa \)-spaces.
\end{theorem}

\begin{proof}
The implication from \ref{thm:charadditivefairSCkappaspaces-a} to 
\ref{thm:charadditivefairSCkappaspaces-c} is 
Proposition~\ref{prop:charadditiveSCandSFCkappaspaces}, while 
\ref{thm:charadditivefairSCkappaspaces-c} implies 
\ref{thm:charadditivefairSCkappaspaces-b} by Lemma~\ref{lem:closed_are_G_delta}. 
Finally, \ref{thm:charadditivefairSCkappaspaces-b} $\Rightarrow$ 
\ref{thm:charadditivefairSCkappaspaces-a} because \( \pre{\kappa}{\kappa} \) is trivially a 
\( \kappa \)-additive \( \SFC_\kappa \)-space, and such spaces are closed under 
\( G^\kappa_\delta \) subspaces by Proposition~\ref{prop:G_deltasubspaces}.
\end{proof}

From Proposition~\ref{prop:charadditiveSCandSFCkappaspaces} we also get a 
characterization of \( \kappa \)-additive \( \SC_\kappa \)-spaces. (The fact that every 
superclosed subset of \( \pre{\kappa}{\kappa} \) is an \( \SC_\kappa \)-space is 
trivial.)

\begin{theorem} \label{thm:charadditiveSCkappaspaces}
For any space \( X \) the following are equivalent:
\begin{enumerate-(a)}
\item \label{thm:charadditiveSCkappaspaces-a}
\( X \) is a \( \kappa \)-additive   \( \SC_\kappa \)-space;
\item \label{thm:charadditiveSCkappaspaces-b}
\( X \) is homeomorphic to a superclosed subset of \( \pre{\kappa}{\kappa} \).
\end{enumerate-(a)}
\end{theorem}

\begin{remark}
Since \( \pre{\kappa}{\kappa} \) is \( \kappa \)-additive and the latter is a hereditary 
property, Theorems~\ref{thm:charadditivefairSCkappaspaces} 
and~\ref{thm:charadditiveSCkappaspaces} can obviously be turned into a 
characterization of \( \kappa \)-additivity inside the classes of \( \SFC_\kappa \)-spaces 
and \( \SC_\kappa \)-spaces, respectively.
\end{remark}

Recall that an uncountable cardinal \( \kappa \) is (\markdef{strongly}) 
\markdef{inaccessible} if it is regular and strong limit, that is, \( 2^\lambda < \kappa \) for 
all \( \lambda < \kappa \). An uncountable cardinal \( \kappa \) is \markdef{weakly 
compact} if and only if it is inaccessible and has the tree property: \( [T] \neq \emptyset \) 
for every tree \( T \subseteq \pre{<\kappa}{\kappa} \) satisfying \( 1 \leq |\mathrm{Lev}_\alpha(T)| < \kappa \) for all \( \alpha < \kappa \). A topological space \( X \) is 
\markdef{\( \kappa \)-Lindel\"of} if all its open coverings admit a subcovering of size \( < \kappa \). (Thus \(\omega\)-Lindel\"ofness is ordinary compactness.)
It turns out that the space \( \pre{\kappa}{2} \) is \( \kappa \)-Lindel\"of if and only if \( \kappa \) is weakly
compact~\cite[Theorem 5.6]{MottoRos:2013}, in which case \( \pre{\kappa}{2} \) and 
\( \pre{\kappa}{\kappa} \) are obviously not homeomorphic; if instead \( \kappa \) is not 
weakly compact, then \( \pre{\kappa}{2} \) is homeomorphic to \( \pre{\kappa}{\kappa} \) 
by~\cite[Theorem 1]{HungNegropMR367930}. This implies that if \( \kappa \) is not 
weakly compact, then we can replace \( \pre{\kappa}{\kappa} \) with \( \pre{\kappa}{2} \) 
in both Proposition~\ref{prop:charadditiveSCandSFCkappaspaces} and  
Theorem~\ref{thm:charadditivefairSCkappaspaces}. Moreover, since one can easily show 
that if \( \kappa \) is not weakly compact then there are homeomorphisms between \( \pre{\kappa}{\kappa} \) and \( \pre{\kappa}{2} \) preserving superclosed sets, for such \( \kappa \)'s we can replace \( \pre{\kappa}{\kappa} \) with \( \pre{\kappa}{2} \) in 
Theorem~\ref{thm:charadditiveSCkappaspaces} as well. As for weakly compact cardinals \( \kappa \), the equivalence between~\ref{thm:charadditivefairSCkappaspaces-a} and 
\ref{thm:charadditivefairSCkappaspaces-b} in 
Theorem~\ref{thm:charadditivefairSCkappaspaces} still holds replacing \( \pre{\kappa}{\kappa} \) with \( \pre{\kappa}{2} \) by Fact~\ref{fct:Baire_is_G_delta_of_Cantor}, but 
the same does not apply to part~\ref{thm:charadditivefairSCkappaspaces-c} and 
Theorem~\ref{thm:charadditiveSCkappaspaces} because for such a \( \kappa \) all 
(super)closed subsets of \( \pre{\kappa}{2} \) are \( \kappa \)-Lindel\"of---see 
Theorems~\ref{thm:charK-Polish-Lindelof} and~\ref{thm:char-spherically-complete-K-Polish-Lindelof}.

\medskip

We now move to \( \GG \)-Polish spaces. Our goal is to show that such spaces coincide 
with the \( \kappa \)-additive \( \SFC_\kappa \)-spaces, and thus that the definition is in 
particular independent of the chosen \( \GG \). Along the way, we also generalize some 
results independently obtained in~\cite[Section 2.3]{Gal19} and close some open 
problems and conjectures contained therein, obtaining a fairly complete picture of the 
relationships among all the proposed generalizations of Polish spaces.

In the subsequent results, $\GG$ is a totally ordered Abelian group with 
$\Deg(\GG)=\kappa$. 
Examples of groups of this form are $\ZZ^\kappa$ and $\RR^\kappa$ (and any other $(\GG')^\kappa$ for an Abelian group $\GG'$) equipped with coordinate-wise operations and lexicographic orders,  
the ``$\kappa$-versions of the reals'' proposed in \cite{AsperoTsaprounisMR3780585} or in \cite{Gal19}, or any non-standard model of the reals of degree $\kappa$.

The next lemma was essentially proved 
in~\cite[Theorem (viii)]{Sik50} and it corresponds to~\ref{thm:charK-metrizable-2} 
\( \Rightarrow \) \ref{thm:charK-metrizable-1} in Theorem~\ref{thm:charK-metrizable}. 
We reprove it here for the reader's convenience.

\begin{lemma} \label{lem:K-Polishareadditive}
Every \( \GG \)-metric space \( X \) is \( \kappa \)-additive, hence also zero-dimensional.
\end{lemma}

\begin{proof}
Let \( \gamma < \kappa \) and
\( (U_\alpha)_{\alpha < \gamma} \) be a sequence of nonempy open sets. If  \( \bigcap_{\alpha < \gamma} U_\alpha \neq \emptyset \), consider an arbitrary
\( x \in \bigcap_{\alpha < \gamma} U_\alpha \). 
The family $\{B_d(x,\varepsilon) \mid \varepsilon \in \GG^+\}$ is a local basis of $x$, so for every $\alpha<\gamma$ we may find \(\varepsilon_\alpha \in \GG^+ \) such that
\( B_d(x,\varepsilon_\alpha) \subseteq U_\alpha \). Since \( \mathrm{Deg}(\GG) = \kappa > \gamma \), there is
\( \varepsilon \in \GG^+ \) such that \( \varepsilon \leq_\GG \varepsilon_\alpha \) for all
\( \alpha < \gamma \): thus \( x \in B_d(x, \varepsilon) \subseteq \bigcap_{\alpha< \gamma} B_d(x, \varepsilon_\alpha) \subseteq \bigcap_{\alpha< \gamma} U_\alpha \).
\end{proof}

\begin{lemma} \label{lem:K-PolishareSFC}
Every \( \GG \)-Polish space \( X \) is strong fair \( \kappa \)-Choquet.
\end{lemma}

\begin{proof}
Fix a compatible Cauchy-complete metric \( d \) on \( X \) and a strictly decreasing sequence \( (r_\alpha)_{\alpha < \kappa} \) coinitial in \( \GG^+ \).
Consider the strategy \(\tau\) of \( \pII \) in \( \fCho^s_\kappa(X) \) in which he replies to 
player \( \pI \)'s move \( (U_\alpha, x_\alpha) \) by picking a ball 
\( V_\alpha =  B_d(x_\alpha, \varepsilon_\alpha) \) with \( \varepsilon_\alpha \in \GG^+ \) 
small enough so that \( \varepsilon_\alpha \leq_\GG r_\alpha \) and \( \cl(V_\alpha) \subseteq U_\alpha \). In particular, we will thus have \( \cl(V_{\alpha+1}) \subseteq V_\alpha \). 
Suppose that \( \langle (U_\alpha, x_\alpha), V_\alpha \mid \alpha < \kappa \rangle \) is a 
run in \( \fCho^s_\kappa(X) \) in which \( \bigcap_{\alpha < \gamma} V_\alpha \neq \emptyset \) for every limit \( \gamma < \kappa \). Then
the choice of the \( \varepsilon_\alpha \)'s ensures that \( (x_\alpha)_{\alpha < \kappa} \) 
is a Cauchy sequence, and thus it converges to some \( x \in X \) by Cauchy-completeness 
of \( d \). It follows that \( x \in \bigcap_{\alpha < \kappa} \cl(V_\alpha) = \bigcap_{\alpha< \kappa} V_\alpha \neq \emptyset \), and thus \( \tau \) is a winning strategy for player \( \pII \).
\end{proof}

\begin{theorem} \label{thm:charK-Polish}
For any space \( X \) the following are equivalent:
\begin{enumerate-(a)}
\item \label{thm:charK-Polish-a}
\( X \) is \( \GG \)-Polish;
\item \label{thm:charK-Polish-b}
\( X \) is a \( \kappa \)-additive \( \SFC_\kappa \)-space;
\item \label{thm:charK-Polish-c}
\( X \) is homeomorphic to a \( G^\kappa_\delta \) subset of \( \pre{\kappa}{\kappa} \);
\item \label{thm:charK-Polish-d}
\( X \) is homeomorphic to a closed subset of \( \pre{\kappa}{\kappa} \).
\end{enumerate-(a)}
\end{theorem}

\begin{proof}
The equivalence of~\ref{thm:charK-Polish-b}, \ref{thm:charK-Polish-c}, and~\ref{thm:charK-Polish-d} is Theorem~\ref{thm:charadditivefairSCkappaspaces}, and~\ref{thm:charK-Polish-d} easily implies~\ref{thm:charK-Polish-a} 
(use e.g.\ the $\GG$-metric described in equation~\eqref{eq:metricBaire}).
The remaining implication \ref{thm:charK-Polish-a} $\Rightarrow$~\ref{thm:charK-Polish-b} follows from Lemma~\ref{lem:K-Polishareadditive} and Lemma~\ref{lem:K-PolishareSFC}.
\end{proof}

As usual, when \( \kappa \) is not weakly compact we can replace \( \pre{\kappa}{\kappa} \) with its homeomorphic copy \( \pre{\kappa}{2} \) in conditions~\ref{thm:charK-Polish-c} 
and~\ref{thm:charK-Polish-d} above. When \( \kappa \) is instead weakly compact, by 
Fact~\ref{fct:Baire_is_G_delta_of_Cantor} we can still replace \( \pre{\kappa}{\kappa} \) 
with \( \pre{\kappa}{2} \) in condition~\ref{thm:charK-Polish-c}, 
but the same does not apply to condition~\ref{thm:charK-Polish-d} because of 
\( \kappa \)-Lindel\"ofness---see Theorem~\ref{thm:charK-Polish-Lindelof}. In view of this 
observation, the implication~\ref{thm:charK-Polish-a} \( \Rightarrow \) \ref{thm:charK-Polish-c} in Theorem~\ref{thm:charK-Polish} is just a reformulation of~\cite[Corollary 2.36]{Gal19}, which is thus nicely complemented by the reverse implication \ref{thm:charK-Polish-c} \( \Rightarrow \) \ref{thm:charK-Polish-a} from our result.

Theorem~\ref{thm:charK-Polish} shows in particular that the notion of \( \GG \)-Polish 
space does not depend on the particular choice of the group \( \GG \). 

\begin{corollary} \label{cor:Gisirrelevant}
Let \( \GG, \GG' \) be two totally ordered (Abelian) groups, both of degree \( \kappa \), 
and \( X \) be a space. 
Then \( X \) is \( \GG \)-Polish if and only if it is \( \GG' \)-Polish.
\end{corollary}

For this reason, from now on will systematically avoid to specify which kind of \( \GG \) 
we are considering and freely use the term ``\( \GG \)-Polish'' as a shortcut for ``\( \GG \)-Polish with respect to a(ny) totally ordered (Abelian) group of degree \( \kappa \)''.

\begin{remark}\label{rmk:choice_of_structure_for_GG}
The only property of the metric $d$ required in the proofs of Lemma~\ref{lem:K-Polishareadditive} and Lemma~\ref{lem:K-PolishareSFC}
is that 
\begin{equation}
\text{For all } x\in X, \text{ the family }  \{B_d(x,\varepsilon) \mid \varepsilon \in \GG^+\} \mbox{ is a local basis of } x.\label{eq:local_base_condition_for_metric} 
\end{equation}
Hence, Theorem~\ref{thm:charK-metrizable} and Theorem~\ref{thm:charK-Polish} (and 
Corollary~\ref{cor:Gisirrelevant}) can be extended to metrics taking values in any other 
kind of structure, as long as condition~\eqref{eq:local_base_condition_for_metric} is still 
satisfied. (In particular, commutativity of $\GG$ is not really needed.)
This includes the case of 
completely $S$-quasimetrizable spaces for a totally ordered semigroup $S$ considered in \cite{ReichelMR0458373}, 
or spaces admitting a complete $\kappa$-ultrametric as defined 
in~\cite{CoskSchlMR3453772}.  In particular, the concepts of (complete) metric space 
and (complete) ultrametric space lead to the same class of spaces in generalized 
descriptive set theory. This is in strong contrast to what happens in the classical setting, 
where Polish ultrametric spaces form a proper subclass of arbitrary Polish spaces 
because admitting a compatible ultrametric implies zero-dimensionality.
\end{remark}

Another easy corollary of Theorem~\ref{thm:charK-Polish} is that a \( G^\kappa_\delta \) 
subset of a \( \GG \)-Polish space is necessarily \( \GG \)-Polish as well (see the proof of Theorem~\ref{thm:Kpolishsubspaces} for details). We complement 
this in  Corollary~\ref{cor:subspace}, using an extension result for continuous functions 
(Proposition~\ref{prop:extension}). These results are the natural generalization of the 
classical arguments in~\cite[Theorems 3.8 and 3.11]{KechrisMR1321597}, and already 
appeared  in~\cite[Theorems 2.34 and 2.35]{Gal19} where, as customary in the subject, the fact that \( \GG \) is Abelian is assumed and used. 
However,  we fully reprove both results for the sake of completeness and to confirm that 
also in this case commutativity of \( \GG \) is not required.

\begin{lemma} \label{lem:smallelements}
Let \( \GG \) be a totally ordered (non-necessarily Abelian) group 
such that \( \GG^+ \) has no minimum.
Then for every \( \varepsilon \in \GG^+ \) and every \( n \in \omega \) there is \( \delta \in \GG^+ \) with%
\footnote{As customary, we denote by \( n \delta \) the finite sum \( \underbrace{\delta +_\GG \dotsc +_\GG \delta}_{n \text{ times}} \).}
 \( n \delta \leq_\GG \varepsilon \).
\end{lemma}

\begin{proof}
It is clearly enough to prove the result for \( n = 2 \).
Let \( \varepsilon' \in \GG^+ \) be such that \( 0_\GG <_\GG \varepsilon' <_\GG \varepsilon \) and set \( \delta = \min \{ \varepsilon', -\varepsilon'+_\GG \varepsilon \} \). Since \( \leq_\GG \) is translation-invariant on both sides we get
\[
\delta+_\GG \delta \leq_\GG 
\varepsilon' +_\GG (- \varepsilon' +_\GG \varepsilon) = \varepsilon.
\qedhere
 \]
\end{proof}

\begin{proposition} \label{prop:extension}
Let \( X \) be a $\GG$-metrizable space, and \( (Y,d) \) be a Cauchy-complete  \( \GG \)-metric space.
Let \( A \subseteq X \) be any set and \( f \colon A \to Y \) be continuous. Then there is a \( G^\kappa_\delta \) set \( B \subseteq X \) and a continuous function \( g \colon B \to Y \) such that \( A \subseteq B \subseteq \cl(A) \) and \( g \) extends \( f \), i.e.\ \( g \restriction A = f \).
\end{proposition}

\begin{proof}
Given any \( \varepsilon \in \GG^+ \), let
\( O_\varepsilon \) be the collection of those \( x \in X \) admitting an open neighborhood 
\( U \) such that \( d(f(y), f(z)) <_\GG \varepsilon \) for all \( y,z \in U \cap A \). By 
definition, each \( O_\varepsilon \) is open in \( X \), and since \( f \colon A \to Y \) is 
continuous then \( A \subseteq O_\varepsilon \) for all \( \varepsilon \in \GG^+ \) (here 
we are implicitly using Lemma~\ref{lem:smallelements}). Fix a strictly decreasing 
sequence \( (r_\alpha)_{\alpha < \kappa} \) coinitial in \( \GG^+ \), and set
\[
B = \cl(A) \cap \bigcap_{\alpha < \kappa} O_{r_\alpha},
 \]
so that \( A \subseteq B \subseteq \cl(A) \) and \( B \) is \( G^\kappa_\delta \) by Lemma~\ref{lem:closed_are_G_delta}.
Fix \( x \in B\), and for every \( \alpha < \kappa \) fix an open neighborhood \( U^x_\alpha \) of
\( x \) witnessing \( x \in O_{r_\alpha} \). 
Without loss of generality we
may assume that \( U^x_\beta \subseteq U^x_\alpha \) if \( \alpha \leq \beta  < \kappa \) (if not, then $\tilde{U}^x_\beta=\bigcap_{\zeta\leq \beta } U^x_\zeta$ is as desired by $\kappa$-additivity of $X$).  
Since \( x \in B \subseteq \cl(A) \), for each \( \alpha < \kappa \) we can pick some \( y_\alpha \in U^x_\alpha \cap A \). The sequence
\( (f(y_\alpha))_{\alpha < \kappa} \) is \( d \)-Cauchy by construction, thus it converges to some \( y \in Y \)
by Cauchy-completeness of \( d \): set \( g(x) = y \). By uniqueness of limits, it is easy to check that the
map \( g \) is well-defined (i.e.\ the value \( g(x) \) is independent of the choice of the \( U^x_\alpha \)'s and
\( y_\alpha \)'s), and that \( g(x) = f(x) \) for all \( x \in A \). It remains to show that \( g \) is also
continuous at every \( x \in B \).
Given any \( \varepsilon \in \GG^+ \), we want to find an open neighborhood \( U \) of \( x \) such that
\( g(U \cap B) \subseteq B_d(g(x), \varepsilon) \). Let \( U^x_\alpha \) and \( y_\alpha \) be as in
the definition of \( g(x) \). Using Lemma~\ref{lem:smallelements}, 
find \( \delta \in \GG^+ \)  such that \( 3\delta \leq_\GG \varepsilon \). Let \( \alpha \) be large enough so that
\( d(f(y_\alpha), g(x)) < \delta \) and \( r_\alpha < \delta \), so that
\( f(U^x_\alpha \cap A) \subseteq B_d\left(g(x), 2 \delta \right) \). We claim that
\( U = U^x_\alpha \) is as required. Indeed, if \( z \in U \cap B \), then when defining \( g(z) \) we may
without loss of generality pick \( U^z_\alpha \) so that \( U^z_\alpha \subseteq U^x_\alpha \): it then
follows that
\[
g(z) \in \cl (  f (U^z_\alpha \cap A)  ) \subseteq \cl (f(U^x_\alpha \cap A)) \subseteq \cl \left(B_d\left(g(x), 2\delta \right)\right)  \subseteq B_d(g(x), \varepsilon),
\]
as required.
\end{proof}

\begin{corollary} \label{cor:subspace}
Let \( X \) be a $\GG$-metrizable space, and let \( Y \subseteq X \) be a completely \( \GG \)-metrizable subspace of $X$. Then \( Y \) is a \( G^\kappa_\delta \) subset of \( X \).
\end{corollary}

\begin{proof}
Apply Proposition~\ref{prop:extension} with \( A = Y \) and \( f \) the identity map from \( Y \) to itself. The resulting \( g \colon B \to Y \) 
coincides with the identity map on $B$ on a dense subset of their common domain: since both functions are continuous, they coincide.
It follows that $B = g(B) \subseteq Y$, hence
\( B=Y \) and \( Y \) is \( G^\kappa_\delta \). 
\end{proof}

In~\cite{Gal19} it is asked whether the reverse implication holds, i.e.\ whether 
\( G^\kappa_\delta \) subsets of \( \GG \)-Polish spaces need to be \( \GG \)-Polish as well 
(see the discussion in the paragraph after~\cite[Theorem 2.10]{Gal19}): our 
Theorem~\ref{thm:charK-Polish} already yields a positive answer, and thus it allows us to 
characterize which subspaces of a \( \GG \)-Polish space are still \( \GG \)-Polish.

\begin{theorem} \label{thm:Kpolishsubspaces}
Let \( X \) be \( \GG \)-Polish and \( Y \subseteq X \). Then \( Y \) is \( \GG \)-Polish if and only if \( Y \) is \( G^\kappa_\delta \) in \( X \).
\end{theorem}

\begin{proof}
One direction follows from Corollary~\ref{cor:subspace}. For the other direction,
since \( X \) is homeomorphic to a closed subset of \( \pre{\kappa}{\kappa} \) by 
Theorem~\ref{thm:charK-Polish}, every \( G^\kappa_\delta \) subspace \( Y \subseteq X \) 
is homeomorphic to a \( G^\kappa_\delta \) subset of \( \pre{\kappa}{\kappa} \). Using 
again Theorem~\ref{thm:charK-Polish}, it follows that \( Y \) is \( \GG \)-Polish as well.
\end{proof}

By Theorem~\ref{thm:charK-Polish}, Theorem~\ref{thm:Kpolishsubspaces}
admits a natural counterpart characterizing \( \SFC_\kappa \)-subspaces of \( \kappa \)-additive \( \SFC_\kappa \)-spaces.

To complete the description of how our classes of spaces relate to each other, we just 
need to characterize those spaces which are in all of them and thus have the richest 
structure (this includes e.g.\ the generalized Cantor and Baire spaces). To this aim, we 
need to introduce one last notion inspired by~\cite[Definition 6.1]{CoskSchlMR3453772} 
and~\cite{Gal19}.

\begin{definition}
A \( \GG \)-metric \( d \) on a space \( X \) is called \markdef{spherically complete} if the 
intersection of every \emph{decreasing} (with respect to inclusion) sequence of open balls is nonempty. If in the 
definition we consider only sequences of order type \( \kappa \) (respectively, \( {<} \kappa \) or \( {\leq} \kappa \)) we say that the metric is \markdef{spherically \( \kappa \)-complete} (respectively, \markdef{spherically \( {<} \kappa \)-complete} or 
\markdef{spherically \( {\leq} \kappa \)-complete}).
\end{definition}

\begin{remark} \label{rmk:spherically} 
Let \( (X,d) \) be a \( \GG \)-metric space.
\begin{enumerate-(i)} 
\item
If the space \( X \) has weight $\leq \kappa$,  then the metric \( d \) is spherically complete if and 
only if it is spherically \( {\leq} \kappa \)-complete. For the non trivial direction, fix 
a decreasing sequence \( (\varepsilon_i)_{i < \kappa} \) 
coinitial in \( \GG^+ \) 
and consider 
an arbitrary decreasing chain of balls \( B_\alpha = B_d(x_\alpha,r_\alpha) \) for 
\( \alpha < \lambda \) with \(\lambda \) a regular cardinal greater than \( \kappa \). If for all 
\( i < \kappa \) there is \( \alpha_i < \lambda \) such that \( r_{\alpha_i} < \varepsilon_i \), 
then by spherically \( \kappa \)-completeness we have 
\( \bigcap_{i < \kappa} B_{\alpha_i} = \{ x \} \) for some \( x \). It follows that 
\( B_\alpha = \{ x \} \) for all \( \alpha \geq \sup_{i < \kappa} \alpha_i \), hence 
\( \bigcap_{\alpha < \lambda} B_{\alpha} = \{ x \} \neq \emptyset \). The remaining case is 
when there is \( \delta \in \GG^+ \) such that \(r_\alpha \geq \delta \) for all \( \alpha < \lambda \). 
If \( \bigcap_{\alpha < \lambda} B_\alpha = \emptyset \), then we could recursively construct an 
increasing sequence \( (\alpha_\beta)_{\beta < \lambda} \) of ordinals \( < \lambda \) such that 
\( x_{\alpha_\beta} \notin B_{\alpha_{\beta'}} \) for all \( \beta < \beta' < \lambda \). By the case 
assumption, we thus have \( d(x_{\alpha_\beta}, x_{\alpha_{\beta'}}) \geq \delta \) for all 
distinct \( \beta, \beta' < \lambda \), contradicting the fact that \( X \) has weight 
$\leq\kappa < \lambda$. 
\item
If \( d \) is spherically \( \kappa \)-complete, then it is also Cauchy-complete (independently of the weight of the space). Thus if \( d \) is spherically complete, then it is both spherically \( {<} \kappa \)-complete and Cauchy-complete.
\item \label{rmk:spherically-iii}
The converse does not hold: 
there are examples of \( \GG \)-metric spaces \( (X,d) \) of 
weight \( \kappa \) such that \( d \) is both Cauchy-complete and spherically 
\( {<} \kappa \)-complete, yet it is not spherically \( \kappa \)-complete. 
For example, consider the subspace $X = \{ x_\alpha \in \pre{\kappa}{2} \mid \alpha < \kappa \}$ of $\pre{\kappa}{2}$, where $x_\alpha(\alpha) =1$ and $x_\alpha(\beta) = 0$ for all $\beta \neq \alpha$. Fix a decreasing sequence $(r_\alpha)_{\alpha < \kappa}$ coinitial in $\GG^+$ and some $s \in \GG^+$. 
The ultrametric on $X$ 
defined by $d(x_\alpha,x_\beta) = s + \max \{ r_\alpha, r_\beta \}$ for distinct $\alpha,\beta<\kappa$ is discrete and hence trivially Cauchy-complete. Moreover, it is $< \kappa$-spherically complete. But the decreasing sequence $(B_d(x_\alpha, s + r_\alpha))_{\alpha<\kappa}$ has empty intersection.
Thus for a given 
\( \GG \)-metric \( d \) being Cauchy-complete and spherically \( {<} \kappa \)-complete is 
strictly weaker than being spherically \mbox{(\( {\leq} \kappa \)-)}complete.
 \end{enumerate-(i)}
\end{remark}

\begin{definition} \label{def:Gpolish+sphericallycomplete}
A \( \GG \)-Polish space is \markdef{spherically} (\markdef{\( {<} \kappa \)-})\markdef{complete} if it admits a compatible Cauchy-complete metric which is also spherically (\( {<} \kappa \)-)complete.
\end{definition}

In~\cite{Gal19}, spherically \( {<} \kappa \)-complete \( \GG \)-Polish spaces are also called 
strongly \( \kappa \)-Polish spaces. Although in view of 
Remark~\ref{rmk:spherically}\ref{rmk:spherically-iii} this seems to be the weakest 
among the two possibilities considered in 
Definition~\ref{def:Gpolish+sphericallycomplete}, it will follow from 
Theorem~\ref{thm:SC-space+G-Polish} that they are indeed equivalent: if a space of 
weight \( \leq \kappa \) admits a compatible Cauchy-complete spherically 
\( {<} \kappa \)-complete \( \GG \)-metric, then it also admits a (possibly different) 
compatible Cauchy-complete \( \GG \)-metric which is (fully) spherically complete. We 
point out that the implication~\ref{thm:SC-space+G-Polish-2} \( \Rightarrow \) \ref{thm:SC-space+G-Polish-1} already appeared  in~\cite[Theorem 2.45]{Gal19}, 
although with a different terminology.

\begin{theorem}\label{thm:SC-space+G-Polish}
For any space \( X \) the following are equivalent:
\begin{enumerate-(a)}
\item\label{thm:SC-space+G-Polish-1}
\( X \) is a \( \kappa \)-additive \(\SC_\kappa \)-space;
\item\label{thm:SC-space+G-Polish-4}
\( X \) is both an \( \SC_\kappa \)-space and \( \GG \)-Polish;
\item\label{thm:SC-space+G-Polish-2}
\( X \) is a spherically \( {<} \kappa \)-complete \( \GG \)-Polish space;
\item\label{thm:SC-space+G-Polish-3}
\( X \) is a spherically complete \( \GG \)-Polish space;
\item\label{thm:SC-space+G-Polish-5}
\( X \) is homeomorphic to a superclosed subset of \( \pre{\kappa}{\kappa} \).
\end{enumerate-(a)}
\end{theorem}

\begin{proof}
Item \ref{thm:SC-space+G-Polish-4}
implies \ref{thm:SC-space+G-Polish-1} because all \( \GG \)-Polish spaces are \( \kappa \)-additive (Lemma~\ref{lem:K-Polishareadditive}), while \ref{thm:SC-space+G-Polish-1}
implies \ref{thm:SC-space+G-Polish-5} by Theorem~\ref{thm:charadditiveSCkappaspaces}. Moreover, any
superclosed subset of \( \pre{\kappa}{\kappa} \) is trivially spherically complete 
with respect to the \( \GG \)-metric on \( \pre{\kappa}{\kappa} \) defined in
equation~\eqref{eq:metricBaire}, thus ~\ref{thm:SC-space+G-Polish-5}
implies \ref{thm:SC-space+G-Polish-3}, and~\ref{thm:SC-space+G-Polish-3}
obviously implies \ref{thm:SC-space+G-Polish-2}. Finally, to prove that~\ref{thm:SC-space+G-Polish-2}
implies \ref{thm:SC-space+G-Polish-4}, recall that every \( \GG \)-Polish space \( X \) is an \( \SFC_\kappa \)-space by
Theorem~\ref{thm:charK-Polish}.
Fix a compatible spherically \( {<} \kappa \)-complete $\GG$-metric on \( X \)
and a winning strategy \(\tau \) for \( \pII \) in \( \fCho^s_\kappa(X) \), and observe that by Remark~\ref{rmk:strategyrangeinbasis}
we can assume that \(\tau \) requires \( \pII \) to play only open \( d \)-balls \( V_\alpha \) because the latter
form a basis for the topology of \( X \). Then \( \tau \) is also winning in
\( \Cho^s_\kappa(X) \) because spherically \( {<} \kappa \)-completeness implies that
\( \bigcap_{\alpha < \gamma} V_\alpha \neq \emptyset \) for every limit \( \gamma < \kappa \) .
\end{proof}

Theorems~\ref{thm:charK-Polish} and~\ref{thm:SC-space+G-Polish} allow us to reformulate our Corollary~\ref{cor:simultaneousembeddings} on retractions in terms of \( \GG \)-Polish spaces. (Again, we have that none of the conditions on \( Y \) can be dropped, see the comment after Corollary~\ref{cor:simultaneousembeddings}.)

\begin{corollary}
If \( X \) is \( \GG \)-Polish, then all its closed subspaces \( Y \) which are also spherically complete \( \GG \)-Polish (possibly with respect to a different \( \GG \)-metric) are retracts of \( X \).
\end{corollary}

Moreover, using the results obtained so far, one can easily observe that the classes of \( \SC_\kappa \)-spaces and \( \GG \)-Polish spaces  are incomparable with respect to inclusion.
On the one hand, there are \( \GG \)-Polish spaces which are not \( \SC_\kappa \)-spaces: 
in~\cite[Theorem 2.41]{Gal19} it is observed that Sikorski's \( \kappa \)-\( \RR \) is such an 
example, but it is also enough to consider any closed subset of \( \pre{\kappa}{\kappa} \) 
which is not strong \( \kappa \)-Choquet, such as the space $X_0$ defined in 
equation~\eqref{eq:finitelymanyzeroes}.
Conversely,  there are \( \SC_\kappa \)-spaces which are not \( \GG \)-Polish (to the best 
of our knowledge, examples of this kind were not yet provided in the literature): just 
take any non-$\kappa$-additive \( \SC_\kappa \)-space, such as \( \pre{\kappa}{\kappa} \) 
equipped with the order topology induced by the lexicographical ordering. 

In a different direction, Theorem~\ref{thm:SC-space+G-Polish} allows us to characterize 
inside one given class those spaces which happen to also belong to a different one in a 
very natural way. For example, among \( \SC_\kappa \)-spaces we can distinguish those 
that are also \( \GG \)-Polish by checking \( \kappa \)-additivity. Conversely, working in 
the class of \( \GG \)-Polish spaces we can isolate those spaces \( X \) in which player 
\( \pII \) wins the strong \( \kappa \)-Choquet game \( \Cho^s_\kappa(X) \) by checking 
spherical completeness.

Figure~\ref{fig:sum_up_relationships} sums up the relationship among the various 
classes of (regular Hausdorff) spaces of weight $\leq \kappa$ considered so 
far. At the end of Section~\ref{sec:characterizations_Cantor_Baire} we will further enrich 
this picture by distinguishing the class of $\kappa$-Lindel\"of spaces---see 
Theorems~\ref{thm:charK-Polish-Lindelof} and~\ref{thm:char-spherically-complete-K-Polish-Lindelof}.

\begin{figure}[h]
    \centering

\tikzset{every picture/.style={line width=0.75pt}} 

\begin{tikzpicture}[x=0.75pt,y=0.75pt,yscale=-1,xscale=1]

\draw   (39.34,81.13) .. controls (39.34,62.86) and (54.15,48.05) .. (72.41,48.05) -- (331.26,48.05) .. controls (349.53,48.05) and (364.34,62.86) .. (364.34,81.13) -- (364.34,203.98) .. controls (364.34,222.25) and (349.53,237.05) .. (331.26,237.05) -- (72.41,237.05) .. controls (54.15,237.05) and (39.34,222.25) .. (39.34,203.98) -- cycle ;
\draw   (58.34,129.72) .. controls (58.34,116.65) and (68.93,106.05) .. (82,106.05) -- (435.67,106.05) .. controls (448.74,106.05) and (459.34,116.65) .. (459.34,129.72) -- (459.34,200.71) .. controls (459.34,213.78) and (448.74,224.38) .. (435.67,224.38) -- (82,224.38) .. controls (68.93,224.38) and (58.34,213.78) .. (58.34,200.71) -- cycle ;
\draw   (76.34,161.25) .. controls (76.34,154.52) and (81.8,149.05) .. (88.54,149.05) -- (437.14,149.05) .. controls (443.88,149.05) and (449.34,154.52) .. (449.34,161.25) -- (449.34,197.85) .. controls (449.34,204.59) and (443.88,210.05) .. (437.14,210.05) -- (88.54,210.05) .. controls (81.8,210.05) and (76.34,204.59) .. (76.34,197.85) -- cycle ;

\draw (100,53) node [anchor=north west][inner sep=0.75pt]  [xslant=-0.03] [align=left] {\begin{minipage}[lt]{170.30736pt}\setlength\topsep{0pt}
\begin{center}
\textbf{$\displaystyle \mathbb{G}${-metrizable}} or, 

equivalently, $\displaystyle \kappa $-additive

\begin{small}
(Up to homeomorphisms: subsets of $\displaystyle \pre{\kappa}{\kappa}$)
\end{small}
\end{center}

\end{minipage}};
\draw (85,111) node [anchor=north west][inner sep=0.75pt]  [xslant=-0.03] [align=left] {\begin{minipage}[lt]{191.62572000000003pt}\setlength\topsep{0pt}
\begin{center}
\textbf{$\displaystyle \mathbb{G}$-Polish}

\begin{small}
(Up to homeomorphism: closed subsets of $\displaystyle \pre{\kappa}{\kappa} $)
\end{small}
\end{center} 
\end{minipage}};
\draw (74,163) node [anchor=north west][inner sep=0.75pt]  [xslant=-0.03] [align=left] {\begin{minipage}[lt]{216.67088000000004pt}\setlength\topsep{0pt}
\begin{center}
\textbf{Spherically (\( {<} \kappa \)-)complete $\displaystyle \mathbb{G}$-Polish} 

\begin{small}
(Up to homeomorphism: superclosed subsets of $\displaystyle \pre{\kappa}{\kappa} $)
\end{small}
\end{center}

\end{minipage}};
\draw (393.76,172.38) node [anchor=north west][inner sep=0.75pt]  [xslant=-0.03] [align=left] {\textbf{SC$\displaystyle _{\kappa }$}};
\draw (388.75,121) node [anchor=north west][inner sep=0.75pt]  [xslant=-0.03] [align=left] {\textbf{$\displaystyle \boldsymbol{f}$SC$\displaystyle _{\kappa }$}};

\end{tikzpicture}

    \caption{Relationships among different Polish-like classes.}
    \label{fig:sum_up_relationships}
\end{figure}
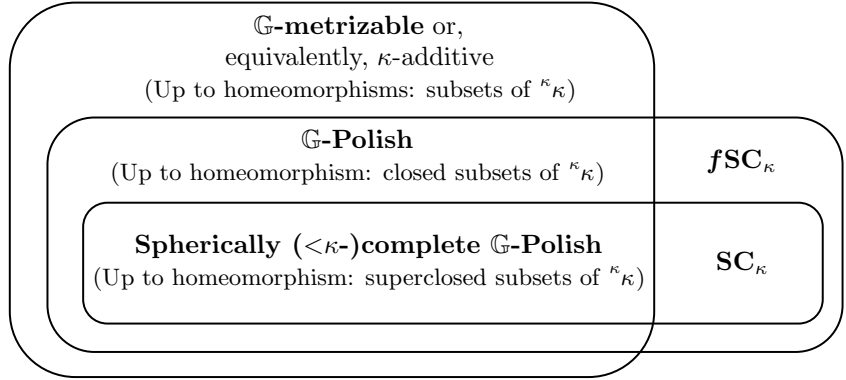

Despite the fact that the classes we are considering are all different from each other, we 
now show that 
one can still pass from one to the other by changing (and sometimes even refining) the 
underlying topology, yet maintaining the same notion of \( \kappa \)-Borelness.

\begin{proposition}\label{prop:refintetoGpolish}
Let \( (X,\tau) \) be an \( \SFC_\kappa \)-space (respectively,  \( \SC_\kappa \)-space). Then there is \( \tau' \supseteq \tau \) such that \( \Bor_\kappa(X,\tau') = \Bor_\kappa(X,\tau) \) and \( (X,\tau') \) is a \( \kappa \)-additive \( \SFC_\kappa \)-space (respectively, \( \SC_\kappa \)-space).
\end{proposition}

\begin{proof}
It is enough to let \( \tau' \) be the topology generated by the \( {<} \kappa \)-sized 
intersections of \( \tau \)-open sets. Arguing as 
in~\cite[Proposition 4.3 and Lemma 4.4]{CoskSchlMR3453772}, player \(\pII \)  still has a 
winning strategy in the relevant Choquet-like game on \( (X,\tau') \). Moreover the 
weight of \( (X,\tau') \) is still \( \leq \kappa \) because we assumed 
\( \kappa^{< \kappa}= \kappa \). Finally, \( \kappa \)-Borel sets do not change because by 
definition \( \tau \subseteq \tau' \subseteq \Bor_\kappa(X,\tau) \).
\end{proof}

This allows us to strengthen~\cite[Theorem 3.3]{CoskSchlMR3453772} and extend it to \( \SFC_\kappa \)-spaces. 

\begin{corollary} \label{cor:refintetoGpolish-}
If \( X \) is an \( \SFC_\kappa \)-space, then there is a pruned tree \( T \subseteq \pre{<\kappa}{\kappa} \) and a continuous bijection \( f \colon [T] \to X \). Moreover, if \( X \) is an \( \SC_\kappa \)-space then \( T \) can be taken to be superclosed.
\end{corollary}

\begin{proof}
Refine the topology \( \tau \) of \( X \) to a topology \( \tau' \supseteq \tau \) as in 
Proposition~\ref{prop:refintetoGpolish}. Then use 
Theorem~\ref{prop:charadditiveSCandSFCkappaspaces} to find a pruned (superclosed, if \( X \) was \( \SC_\kappa \)) tree \(  T \subseteq \pre{<\kappa}{\kappa} \) and a 
homeomorphism \( f \colon [T] \to (X,\tau') \). Since \( f \) remains a continuous bijection 
when stepping back to \( \tau \), we get that \( T \) and \( f \) are as required.
\end{proof}

By Proposition~\ref{prop:refintetoGpolish} (together with Theorem~\ref{thm:charK-Polish}), every \( \SFC_\kappa \)-space, and thus every \( \SC_\kappa \)-space, can be 
turned into a \( \GG \)-Polish space sharing the same \( \kappa \)-Borel structure  by 
suitably refining its topology.
In contrast, it is not always possible to \emph{refine} the topology \( \tau \) of an 
\( \SFC_\kappa \)-space \( X \) to turn it into an \( \SC_\kappa \)-space, even if we start 
with a \( \kappa \)-additive (hence \( \GG \)-Polish) one and we further allow to change its 
\( \kappa \)-Borel structure. Indeed, as shown in the next example, there are 
$\kappa$-additive strong fair \( \kappa \)-Choquet (i.e.\ $\GG$-Polish) spaces $(X,\tau)$ such that for 
every topology $\tau'\supseteq\tau$, the space $(X,\tau')$ is not an $\SC_\kappa$-space.

\begin{example}\label{xmp:no_top_refinement_from_SFC_to_SC}
Consider a closed (hence \( \GG \)-Polish) subspace \( C  \subseteq \pre{\kappa}{\kappa} \) which is not a continuous image of $\pre{\kappa}{\kappa}$. Such a set exists by~\cite[Theorem 1.5]{LuckeSchlMR3430247}: we can e.g.\ let \( C \) be the set of
well-orders on $\kappa$ (coded as elements of 
$\pre{\kappa}{2} \subseteq \pre{\kappa}{\kappa}$ via the usual G\"odel pairing function). 
If one could find a refinement \( \tau' \) of the bounded topology $\tau$ on \( C \) such that 
\( (C, \tau') \) is an \( \SC_\kappa \)-space (recall that any \( \SC_\kappa \)-space has weight $\leq\kappa$ by definition), then \( (C,\tau') \) would be a continuous image 
of \( \pre{\kappa}{\kappa} \) by~\cite[Theorem 3.5]{CoskSchlMR3453772} and thus so 
would be \( (C,\tau) \), contradicting the choice of \( C \).
\end{example}

Nevertheless, if we drop the requirement that \( \tau' \) refines the original topology 
\( \tau \) of \( X \), then we can get a result along the lines above. This is due to the next 
technical lemma, which will be further extended in Section~\ref{sec:standardBorel} (see 
Corollary~\ref{cor:superclosed}).

\begin{lemma}\label{lem:superclosed}
Every closed set \( C \subseteq \pre{\kappa}{\kappa} \) is \( \kappa \)-Borel isomorphic to a superclosed set \( C' \subseteq \pre{\kappa}{\kappa} \).
\end{lemma}

\begin{proof}
If \( C \) has \( {\leq} \kappa \)-many points, then any bijection between \( C \) and
\( C' = \{ \alpha \conc 0^{(\kappa)} \mid \alpha < |C| \} \),
where \( 0^{(\kappa)} \) is the constant sequence with length \( \kappa \) and value \( 0 \), 
is a \( \kappa \)-Borel isomorphism between \( C \) and the superclosed set
\( C' \), hence we may assume without loss of generality that \( | C| >  \kappa \). Let \( T \subseteq \pre{< \kappa}{\kappa} \) be a pruned tree
such that \( C = [T] \). Let \( L(T) \) be the set of sequences \( s \in \pre{< \kappa}{\kappa} \) 
of limit length such that \( s \notin T \) but \( s \restriction \alpha \in T \) for all \( \alpha < \leng(s) \). (Clearly, the set \( L(T) \) is empty if and only if \( C \) is already superclosed). 
Set \( C' = [T'] \) with
\[
T' = T \cup	\{ s \conc  0^{(\alpha)} \mid s \in L(T) \wedge \alpha < \kappa \},
 \]
 where \( 0^{(\alpha)} \) denotes the sequence of length \( \alpha \) constantly equal to \( 0 \). 
 The tree \( T' \) is clearly pruned and \( {<} \kappa \)-closed, hence \( C' \) is 
 superclosed.
 Notice also that \( C' \setminus C = \{ s {}^\smallfrown{} 0^{(\kappa)} \mid s \in L(T) \} \)
 has size \( \leq \kappa \).
Pick a set \( A \subseteq C \) of size \( \kappa \) and fix any bijection
\( g \colon A \to A \cup (C' \setminus C) \).
Since both \( C \) and \( C' \) are Hausdorff, it is easy to check that the map
\[
f \colon C \to C', \qquad x \mapsto
\begin{cases}
g(x) & \text{if } x \in A \\
x & \text{otherwise}
\end{cases}
 \]
 is a \( \kappa \)-Borel isomorphism.
\end{proof}

Combining this lemma with Proposition~\ref{prop:refintetoGpolish} and 
Theorem~\ref{thm:charadditivefairSCkappaspaces} we thus get

\begin{proposition}
Let \( (X,\tau) \) be an \( \SFC_\kappa \)-space. Then there is a topology \( \tau' \)  on 
\( X \) such that \( \Bor_\kappa(X,\tau') = \Bor_\kappa(X,\tau) \) and \( (X,\tau') \) is a 
\( \kappa \)-additive \( \SC_\kappa \)-space (equivalently, a spherically complete 
\( \GG \)-Polish space).
\end{proposition}

As a corollary, we finally obtain:

\begin{theorem}\label{thm:char_up_to_kBorel_iso}
Up to \( \kappa \)-Borel isomorphism, the following classes of spaces are the 
same:
\begin{enumerate-(1)}
\item
\( \SFC_\kappa \)-spaces;
\item
\( \SC_\kappa \)-spaces;
\item
\( \GG \)-Polish spaces;
\item
\( \kappa \)-additive \( \SC_\kappa \)-spaces or, equivalently, spherically complete \( \GG \)-Polish spaces.
\end{enumerate-(1)}
\end{theorem}

Theorem~\ref{thm:char_up_to_kBorel_iso} shows that, as we already claimed after 
Definition~\ref{def:standardBorel}, we can consider any class of Polish-like spaces to 
generalize~\ref{intro:stbor1}: they all yield the same notion, and it is thus not necessary 
to formally specify one of them. Furthermore, in Section~\ref{sec:standardBorel} we will 
prove that the class of \( \kappa \)-Borel spaces obtained in this way coincide with the 
class of all standard \( \kappa \)-Borel spaces as defined in 
Definition~\ref{def:standardBorel}, so we do not even need to introduce a different 
terminology.

The sweeping results obtained so far allow us to improve some results from the 
literature and solve 
 some open problems contained therein, so let us conclude this 
section with a brief discussion on this matter. 
In~\cite[Theorem 2.51]{Gal19} it is proved that, in our terminology, if \( X \) is a 
spherically \( {<} \kappa \)-complete \( \GG \)-Polish space and \( \kappa \) is weakly 
compact, then every \( \SC_\kappa \)-subspace \( Y \subseteq X \) is \( G^\kappa_\delta \) 
in \( X \).
By Theorem~\ref{thm:charK-Polish} and Corollary~\ref{cor:subspace}, we actually have 
that every \( \SC_\kappa \)-subspace \( Y \) of a $\GG$-metrizable space $X$ is 
\( G^\kappa_\delta \) in \( X \): hence the further hypotheses on \( \kappa \) and \( X \) 
required in \cite[Theorem 2.51]{Gal19} are not necessary.
Furthermore, in~\cite[Lemma 2.47]{Gal19} the converse is shown to hold assuming that 
\( X \) is a \( \GG \)-metric \( \SC_\kappa \)-space (which through \( \kappa \)-additivity 
implies that \( X \) is \( \GG \)-Polish by Theorem~\ref{thm:SC-space+G-Polish}) and 
\( Y \) is spherically \( {<} \kappa \)-complete. Theorems~\ref{thm:Kpolishsubspaces} 
and~\ref{thm:SC-space+G-Polish} show that we can again weaken the hypotheses on 
\( X \) by dropping the requirement that \( X \) be a \( \SC_\kappa \)-space: 
if \( X \) is \( \GG \)-Polish and \( Y \subseteq X \) is spherically \( {<} \kappa \)-complete 
and \( G^\kappa_\delta \), then \( Y \) is a \( \SC_\kappa \)-space.
Finally, Theorem~\ref{thm:SC-space+G-Polish} shows that \cite[Theorem 2.53]{Gal19} 
and \cite[Proposition 3.1]{CoskSchlMR3453772} deal with the same phenomenon: if $X$ 
is a $\kappa$-perfect $\SC_\kappa$-space, then there is a continuous injection $f$ from the 
generalized Cantor space into $X$, and if furthermore $X$ is $\kappa$-additive, then $f$ 
can be taken to be a homeomorphism on the image. The latter fact will be slightly improved in 
Theorem~\ref{thm:PSP for perfect SC_kappa kappa-additive}, where we show that in the 
latter case the range of \( f \) can be taken to be superclosed.

Summing up the results above, one can now complete and improve the diagram from~\cite[p.\ 25]{Gal19}, which corresponds to Arrows 1--7 of Figure~\ref{fig:subsets_G_Polish} (although~\cite{Gal19} sometimes requires additional assumption on the space \( Y \) or on the cardinal \( \kappa \), see the discussion below).

{

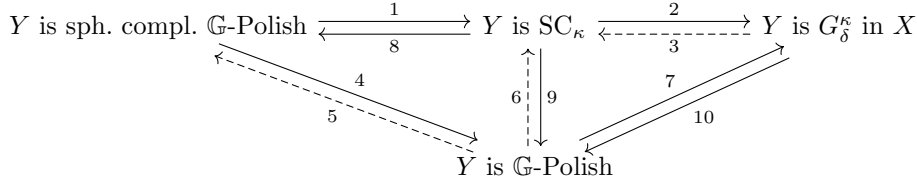
\begin{figure}[h]
    \centering
    
\begin{tikzcd}
\mbox{$Y$ is sph.\ compl.\ $\GG$-Polish} \arrow[rr, shift left, "1"] \arrow[rrdd, shift left, "4"] &  & \mbox{$Y$ is $\SC_\kappa$} \arrow[rr, shift left, "2"] \arrow[ll, shift left, "8"] \arrow[dd, shift left, "9"]                                    &  & \mbox{$Y$ is $G^\kappa_\delta$ in $X$} \arrow[lldd, shift left, "10"] \arrow[ll, 
dashed, shift left, "3"] \\
                                                   &  &                                                                                                              &  &                                                                     \\
                                                   &  & \mbox{$Y$ is $\GG$-Polish} \arrow[lluu, 
                                                   dashed, shift left, "5"] \arrow[rruu, shift left, "7"] \arrow[uu, 
                                                   dashed, shift left, "6"] &  &                                                                    
\end{tikzcd}

    \caption{Properties of subspaces \( Y \subseteq X \) for \( X \)  a $\GG$-Polish space. A line means implication without further assumptions, while a dotted line means that the implication holds under the further assumption that $Y$ is spherically complete or, equivalently, an $\SC_\kappa$-space.}
    \label{fig:subsets_G_Polish}
\end{figure}
}

Here is a list of our improvements:
\begin{itemizenew}
\item First of all, the ambient space $X$ can be any $\GG$-Polish space, and need not be spherically complete as assumed in \cite{Gal19}.
\item
The new Arrows 8 and 9 hold because by \( \GG \)-metrizability \( Y \) is \( \kappa \)-additive, and hence a spherically \( {<} \kappa \)-complete \( \GG \)-Polish space.
\item
Arrow 10 holds as well by Theorem~\ref{thm:Kpolishsubspaces}.
\item
The implication Arrow 2 holds unconditionally (\( \kappa \) does not need to be weakly compact, as originally required in~\cite{Gal19}).
\item
The requirement that, in our terminology, \( Y \) be spherically \( {<} \kappa \)-complete cannot instead be 
dropped in the implication Arrow 3: indeed, there are even closed subsets of 
\( X = \pre{\kappa}{\kappa} \) which are not homeomorphic to a superclosed subset of 
\( \pre{\kappa}{\kappa} \), and hence they are not strong \( \kappa \)-Choquet. Thus in this case the hypothesis in~\cite{Gal19} were already optimal.
\item 
We now obtained that Arrows 5 and 6, which were forbidden in~\cite{Gal19}, hold when additionally requiring that \( Y \) be spherically \( < \kappa \)-complete (the same hypothesis as in Arrow 3): taking into account Galeotti's counterexamples, such a hypothesis cannot be dropped.
\end{itemizenew}


\section{Characterizations of \( \pre{\kappa}{\kappa} \) and \( \pre{\kappa}{2} \)}\label{sec:characterizations_Cantor_Baire}

The (classical) Cantor and Baire spaces play a central role in classical descriptive set 
theory. It is remarkable that they admit a purely topological characterization 
(see~\cite[Theorems 7.4 and 7.7]{KechrisMR1321597}).

\begin{theorem} \label{thm:classicalchar} \ 
\begin{enumerate-(1)}
\item \label{thm:classicalchar-1}
(Brouwer) Up to homeomorphism, the Cantor space \( \pre{\omega}{2} \) is the unique nonempty perfect compact metrizable zero-dimensional space.
\item \label{thm:classicalchar-2}
(Alexandrov-Urysohn) Up to homeomorphism, the Baire space \( \pre{\omega}{\omega} \) is the unique nonempty Polish zero-dimensional space such that
all its compact subsets have empty interior.
\end{enumerate-(1)}
\end{theorem}

Our next goal is to find analogous characterizations of the generalized Baire and Cantor 
spaces. To this aim, we first have to generalize the above mentioned topological notions 
to our setup.

First of all, we notice that a special feature of \( \pre{\kappa}{\kappa} \) and 
\( \pre{\kappa}{2} \) which is not shared by some of the other \( \SC_\kappa \)-spaces is 
\( \kappa \)-additivity: since this condition already implies that the space be 
zero-dimensional, the latter will always be absorbed by \( \kappa \)-additivity and will not  
explicitly appear in our statements.
As for compactness, it is natural to replace it with the property of being 
\( \kappa \)-Lindel\"of. Notice that this condition may play a role in the characterization 
of \( \pre{\kappa}{2} \) only when \( \kappa \) is weakly compact, as otherwise 
\( \pre{\kappa}{2} \) is not \( \kappa \)-Lindel\"of. However, this is not a true limitation, 
because if \( \kappa \) is not weakly compact, then the spaces \( \pre{\kappa}{2} \) and 
\( \pre{\kappa}{\kappa} \) are homeomorphic, and thus the characterization of 
\( \pre{\kappa}{\kappa} \) takes care of both. In view of the Hurewicz 
dichotomy~\cite[Theorem 7.10]{KechrisMR1321597}, which 
in~\cite{LuckeMottSchlMR3557473} has been analyzed in detail in the context of 
generalized descriptive set theory, we will also consider \markdef{\( K_\kappa \)-sets}, 
i.e.\ sets in a topological space which can be written as unions of \( \kappa \)-many 
\( \kappa \)-Lindel\"of sets.

We now come to perfectness.
The notion of an isolated point may be transferred to the generalized context in (at least) two natural ways:
\begin{itemizenew}
\item
keeping the original definition: a point \( x \) is \markdef{isolated} in \( X \) if there is an open set \( U \subseteq X \) such that \( U = \{ x \} \);
\item
allowing short intersections of open sets 
(see e.g.\ \cite[Section 3]{CoskSchlMR3453772}):  a point \( x \) is \markdef{\( \kappa \)-isolated} in \( X \) if there are \( {<} \kappa \)-many open sets whose intersection is \( \{ x \} \).
\end{itemizenew}
A topological space is then called (\markdef{\( \kappa \)-})\markdef{perfect} if it has no (\( \kappa \)-)isolated points.

If we restrict the attention to \( \kappa \)-additive spaces, as we do in this section, the two notions coincide.
However, the notion of \( \kappa \)-perfectness is in a sense preferable when the space 
\( X \) is not $\kappa$-additive because it implies that \( X \) has weight at least $\kappa$ 
and that all its nonempty open sets have size \( \geq \kappa \) (use the regularity of 
\( \kappa \) and the fact that all our spaces are Hausdorff). If we further require \( X \) to 
be strong \( \kappa \)-Choquet, we get the following strengthening of the last property.

\begin{lemma}  \label{lem:equivperf1}
Let \( X \) be an \( \SC_\kappa \)-space. If \( X \) is \( \kappa \)-perfect, then every open set \( U \subseteq X \) has size \( 2^\kappa \).
\end{lemma}

\begin{proof}
If \( X \) is \( \kappa \)-perfect, then so is every open \( U \subseteq X \). Since \( U \) is 
strong \( \kappa \)-Choquet as well, there is a continuous injection from 
\( \pre{\kappa}{2} \) into \( U \) by~\cite[Proposition 3.1]{CoskSchlMR3453772}, hence 
\( |U| = 2^\kappa  \). 
\end{proof}

In the statement of Lemma~\ref{lem:equivperf1} one could further replace the open set 
\( U \) with a \( {<} \kappa \)-sized intersection of open sets.
The lemma is instead not true for arbitrary \( \SFC_\kappa \)-spaces, even when requiring 
\( \kappa \)-additivity (and thus it does not work for arbitrary \( \GG \)-Polish spaces either).  
For a counterexample,
consider the closed subspace \( X_0 \) of  $\pre{\kappa}{2}$ defined in 
equation~\eqref{eq:finitelymanyzeroes}:
by Theorem~\ref{thm:charK-Polish}, \( X_0 \) is a \( \kappa \)-additive \( \SFC_\kappa \)-space (equivalently, a \( \GG \)-Polish space), it is clearly \( \kappa \)-perfect, yet it has 
size \( \kappa \).

In the next lemma we crucially use the fact that we assumed
 \( \kappa^{< \kappa} = \kappa \).

\begin{lemma} \label{lem:sizetoweight}
If \( Y \) is a \( T_0 \)-space of size \( > \kappa \), then \( Y \) has weight \( \geq \kappa \).
\end{lemma}

\begin{proof}
Let \( \mathcal{B} \) be any basis of \( Y \). Then the map sending each point of \( Y \) into the set of its basic open neighborhoods is an injection into
\( \pow(\mathcal{B}) \). Thus if there is such a \( \mathcal{B} \) of size \( \nu < \kappa \) then \( |Y| \leq 2^\nu \leq \kappa^{< \kappa} =  \kappa \).
\end{proof}

A tree \( T \subseteq \pre{< \kappa}{\kappa} \) is \markdef{splitting} if for every \( s \in T \) there are incomparable \( t,t' \in T \) extending \( s \) (without loss of generality we can further require that \( \leng(t) = \leng(t') \)). 
We now show that the splitting condition captures the topological notion of perfectness 
for \( \kappa \)-additive \( \SC_\kappa \)-spaces. (Notice that the equivalence between 
items~\ref{char_perfectness_additive_notweaklycompact - kappa_perfect} 
and~\ref{char_perfectness_additive_notweaklycompact - splitting tree} in 
Lemma~\ref{lem:equivperf2} may be seen as the analogue of 
Theorem~\ref{thm:charadditiveSCkappaspaces} for (\( \kappa \)-)perfect 
\( \kappa \)-additive \( \SC_\kappa \)-spaces.)

\begin{lemma}\label{lem:equivperf2}
Let \( X \) be a nonempty $\kappa$-additive \( \SC_\kappa \)-space. The following are equivalent:
\begin{enumerate-(a)}
\item \label{char_perfectness_additive_notweaklycompact - kappa_perfect}
\( X \) is ($\kappa$-)perfect;
\item \label{char_perfectness_additive_notweaklycompact - big_size}
every nonempty open subset of \( X \) has size $>\kappa$;
\item \label{char_perfectness_additive_notweaklycompact - big_weight}
every nonempty open subspace of \( X \) has weight $\kappa$; 
\item \label{char_perfectness_additive_notweaklycompact - splitting tree+homeo}
every superclosed \( T \subseteq \pre{<\kappa}{\kappa} \) such that \( X \) is homeomorphic to \( [T] \) is splitting;
\item \label{char_perfectness_additive_notweaklycompact - splitting tree}
there is a splitting superclosed%
\footnote{This is a bit redundant: if \( T \) is splitting and \( {<} \kappa \)-closed, then it is automatically pruned.}
 tree $T \subset \pre{<\kappa}{\kappa}$ with $[T]$ homeomorphic to  $X$. 
\end{enumerate-(a)}
\end{lemma}

\begin{proof}
The implication
\ref{char_perfectness_additive_notweaklycompact - kappa_perfect}~\( \Rightarrow \)~\ref{char_perfectness_additive_notweaklycompact - big_size} follows from
Lemma~\ref{lem:equivperf1}, while the implication
\ref{char_perfectness_additive_notweaklycompact - big_size}~\( \Rightarrow \)~\ref{char_perfectness_additive_notweaklycompact - big_weight} follows from
Lemma~\ref{lem:sizetoweight} and the fact that $X$ has weight (at most) $\kappa$.
In order to prove \ref{char_perfectness_additive_notweaklycompact - big_weight}~\( \Rightarrow \)~\ref{char_perfectness_additive_notweaklycompact - splitting tree+homeo}, notice that if
\( s \in T \) then $\Nbhd_s\cap[T] \neq \emptyset$ because \( T \) is superclosed. Thus $s$ 
must have two incomparable extensions, since otherwise \( \Nbhd_s \cap [T] \) would be 
a nonempty open set of weight (and size) \( 1 \). The 
implication \ref{char_perfectness_additive_notweaklycompact - splitting tree+homeo}~\( \Rightarrow \)~\ref{char_perfectness_additive_notweaklycompact - splitting tree} 
follows from Theorem~\ref{thm:charadditiveSCkappaspaces}, which ensures the 
existence of a superclosed \( T \subseteq \pre{<\kappa}{\kappa} \) with \( [T] \) 
homeomorphic to \( X \): such a \( T \) is then necessarily splitting by 
condition~\ref{char_perfectness_additive_notweaklycompact - splitting tree+homeo}. 
Finally, for the implication \ref{char_perfectness_additive_notweaklycompact - splitting tree}~\( \Rightarrow \)~\ref{char_perfectness_additive_notweaklycompact - kappa_perfect}
notice that if \( T \) is splitting and superclosed,  then for every two incomparable 
extensions \( t,t' \in T \) of a given \( s \in T \) we have 
\( \Nbhd_t \cap [T] \neq \emptyset \) and \( \Nbhd_{t'} \cap [T] \neq \emptyset \) but 
\( \Nbhd_t \cap \Nbhd_{t'} = \emptyset \), hence \( |\Nbhd_s \cap [T]| > 1 \) for all 
\( s \in T \). 
\end{proof}

\begin{remark} \label{rmk:equivperf2}
Notice that if \( \kappa \) is inaccessible, then the splitting condition on the superclosed 
tree \( T \) in items~\ref{char_perfectness_additive_notweaklycompact - splitting tree+homeo} and~\ref{char_perfectness_additive_notweaklycompact - splitting tree} 
above can be strengthened to
\begin{equation}\label{eq:strongsplitting}
\forall s \in T \, \forall \nu < \kappa \, \exists  \alpha < \kappa \, (\alpha > \leng(s) \wedge | \mathrm{Lev}_\alpha(T_s)| \geq \nu ).
\end{equation}
Notice also that if \( \alpha < \kappa \) witnesses~\eqref{eq:strongsplitting} for given \( s \in T\) and \( \nu < \kappa \), then every \( \alpha \leq \alpha' < \kappa \) witnesses the same fact because \( T \) is pruned.
\end{remark}

Lemma~\ref{lem:equivperf2} allows us to prove the following strengthening of~\cite[Proposition 3.1]{CoskSchlMR3453772} and \cite[Theorem~2.53]{Gal19},  answering in particular~\cite[Question 3.2]{CoskSchlMR3453772} for the case of \( \kappa \)-additive spaces.

\begin{theorem} \label{thm:PSP for perfect SC_kappa kappa-additive}
Let \( X \) be a nonempty \( \kappa \)-additive \( \SC_\kappa \)-space. If \( X \) is (\( \kappa \)-)perfect, then there is a superclosed \( C \subseteq X \) which is homeomorphic to \( \pre{\kappa}{2} \).
\end{theorem}

\begin{proof}
By Lemma~\ref{lem:equivperf2} we may assume that \( X = [T] \)  with \( T \subseteq \pre{<\kappa}{\kappa} \) superclosed and splitting.
Recursively define a map \( \varphi \colon \pre{<\kappa}{2} \to T \) by setting 
\( \varphi(\emptyset) = \emptyset \) and then  letting \( \varphi(t \conc 0) \) and 
\( \varphi ( t \conc  1) \) be incomparable extensions in \( T \) of the sequence of 
\( \varphi(t) \). At limit levels we set $\varphi(t)=\bigcup_{\alpha<\leng(t)}\varphi(t \restriction \alpha)$, which is still an element of \( T \) because the latter is 
\( {<} \kappa \)-closed.

By construction, \( \varphi \) is a tree-embedding from \( \pre{< \kappa}{2} \) into \( T \), 
i.e.\ \( \varphi \) is monotone and preserves incomparability. Moreover, 
\( \leng(\varphi(t)) \geq \leng(t) \) for every \( t \in \pre{<\kappa}{2} \).
Let \( T' \) be the subtree of \( T \) generated by \( \varphi(\pre{<\kappa}{2}) \), that is
\[
T' = \{ s \in T \mid s \subseteq \varphi(t) \text{ for some } t \in \pre{<\kappa}{2} \}.
 \]
It is easy to see that \( T' \) is pruned. We now want to check that it is also \( < \kappa \)-closed by showing that if \( s \notin  T' \) for some \( s \) of limit length, then there is \( \alpha < \leng(s) \) such that \( s \restriction \alpha \notin T' \). 
(We present a detailed argument because the claim uses in an essential way that \( \pre{<\kappa}{2} \) is finitely splitting, and would instead fail if e.g.\ $\pre{<\kappa}{2}$ is replaced by  $\pre{<\kappa}{\omega}$.) 
Set \( A =  \{ t \in \pre{< \kappa}{2} \mid \varphi(t) \subseteq s \} \). 
Since \( \varphi \) preserves incomparability, all sequences in \( A \) are comparable and 
thus the sequence \( \bar{t} = \bigcup \{ t \mid t \in A \} \in \pre{< \kappa}{2} \) is 
well-defined and such that \( \varphi(\bar{t}) \subsetneq s \) (here we use that \(\varphi\) is 
defined in a continuous way at limit levels and \( s \notin T' \)). Since \( s \notin T' \), the sequences 
\( \varphi(\bar{t} {}^\smallfrown{} 0 ) \) and \( \varphi(\bar{t} {}^\smallfrown{}  1) \) are 
both incomparable with \( s \) by the choice of \( \bar{t} \), and since \( \leng(s) \) is limit 
there is \(  \leng(\varphi(\bar{t})) < \alpha< \leng(s) \) such that the above sequences are incomparable with 
\( s \restriction \alpha \) as well:%
\footnote{Here is where we use that there are finitely many (in fact, two) extensions of $\bar{t}$ of length \( \leng(\bar{t})+1 \).}
 we claim 
that such \( \alpha \) is as required. Given an arbitrary \( t \in \pre{<\kappa}{2} \), we 
distinguish various cases. If \( t \) is incomparable with \( \bar{t} \), then \( \varphi(t) \) is 
incomparable with \( \varphi(\bar{t}) \) and thus with \( s \restriction \alpha \) as well 
because by construction \( \varphi(\bar{t}) \subseteq s \restriction \alpha \). If 
\( t \subseteq \bar{t} \), then by monotonicity  of \( \varphi \) we have that 
\( \varphi(t) \subseteq \varphi(\bar{t}) = s \restriction \leng(\varphi(\bar{t})) \) and thus \( \varphi(t) \) is a proper initial segment 
of \( s \restriction \alpha \) by \( \alpha > \leng(\varphi(\bar{t})) \). Finally, if \( t \) properly extends \( \bar{t} \), then 
\( t \supseteq \bar{t} \conc i \) for some \( i \in \{ 0,1 \} \): but then 
\( \varphi(t) \supseteq \varphi(\bar{t} {}^\smallfrown{} i) \) is incomparable with 
\( s \restriction \alpha \) again. So in all cases we get that 
\( s \restriction \alpha \not\subseteq \varphi(t) \), and since \( t \) was arbitrary this 
entails \( s \restriction \alpha \notin T' \), as required.

This shows that \( T' \) is a superclosed subtree of \( T \). Moreover, \( \varphi \) canonically induces the function \( f_\varphi \colon \pre{\kappa}{2} \to C = [T'] \) where
\[
f_\varphi ( x ) =  \bigcup\nolimits_{\alpha < \kappa} \varphi(x \restriction \alpha),
\]
which is well-defined by monotonicity of \( \varphi \) and  
\( \leng(\varphi(x \restriction \alpha)) \geq \alpha \). Moreover \( f_\varphi \) is a bijection 
because \( \varphi \) is a tree-embedding, and by construction 
\( f_\varphi(\Nbhd_t) = \Nbhd_{\varphi(t)} \cap C \) for all \( t \in \pre{<\kappa}{2} \). Since 
\( \{ \Nbhd_{\varphi(t)} \cap C \mid t \in \pre{<\kappa}{2} \} \) is clearly a basis for \( C \), 
this shows that \( f_\varphi \)  is a homeomorphism between \( \pre{\kappa}{2} \) and \( C  \).
\end{proof}

The previous theorem can be turned into the following characterization: a topological 
space contains a closed homeomorphic copy of \( \pre{\kappa}{2} \) if and only if it 
contains a nonempty closed (\( \kappa \)-)perfect \( \kappa \)-additive \( \SC_\kappa \)-
subspace.

Finally, we briefly discuss \( \kappa \)-Lindel\"of sets and \( K_\kappa \)-sets.
The Alexandrov-Urysohn characterization  of the Baire space 
(Theorem~\ref{thm:classicalchar}\ref{thm:classicalchar-2}) implicitly deals with Baire 
category. In fact, compact sets are closed, thus requiring that they have empty interior is 
equivalent to requiring that they are nowhere dense. The latter notion makes sense also 
in the generalized setting, but the notion of meagerness needs to be replaced with 
\( \kappa \)-meagerness, where a subset \( A \subseteq X \) is called 
\markdef{\( \kappa \)-meager} if it can be written as a union of \( \kappa \)-many 
nowhere dense sets. A topological space is \markdef{\( \kappa \)-Baire} if no nonempty open subset of \( X \) is 
\( \kappa \)-meager. It is not difficult to see that if \( \kappa \) is regular, then 
\( \pre{\kappa}{\kappa} \) is \( \kappa \)-Baire (see e.g.\ \cite{Friedman:2011nx,AM}), so 
the next lemma applies to it.

\begin{lemma} \label{lem:lindelKkappa}
Suppose that \( X \) is a \( \kappa \)-additive \( \kappa \)-Baire space. Then the following are equivalent:
\begin{enumerate-(a)}
\item \label{lem:lindelKkappa-1}
all \( \kappa \)-Lindel\"of subsets of \( X \) have empty interior;
\item \label{lem:lindelKkappa-2}
all \( K_\kappa \) subsets of \( X \) have empty interior.
\end{enumerate-(a)}

\end{lemma}

\begin{proof}
The nontrivial implication \ref{lem:lindelKkappa-1} \( \Rightarrow \) \ref{lem:lindelKkappa-2} follows from the fact that if 
\( A = \bigcup_{\alpha < \kappa} A_\alpha \subseteq X \) with all \( A_\alpha \)'s 
\( \kappa \)-Lindel\"of, then \( A \) is \( \kappa \)-meager because in a \( \kappa \)-additive (Hausdorff) 
space all \( \kappa \)-Lindel\"of sets are necessarily closed and thus, 
by~\ref{lem:lindelKkappa-1}, each \( A_\alpha \) is nowhere dense;  thus the interior of 
\( A \), being \( \kappa \)-meager as well, must be the empty set.
\end{proof}

Finally, observe that if a space \( X \) can be partitioned into \( \kappa \)-many nonempty 
clopen sets, then it is certainly not \( \kappa \)-Lindel\"of. The next lemma shows that 
the converse holds as well if \( X \) is \( \kappa \)-additive and of weight at most \( \kappa \). 

\begin{lemma} \label{split closed set into kappa pieces}
Let  $X$ be a nonempty $\kappa$-additive space of weight $\leq \kappa$. If \( X \) is not 
\( \kappa \)-Lindel\"of, then it can be partitioned into $\kappa$-many nonempty clopen 
subsets.
\end{lemma}

\begin{proof}
Since \( X \) is zero-dimensional of weight $\leq \kappa$ but not \( \kappa \)-Lindel\"of, there is a clopen covering 
\( \{ U_\alpha \mid \alpha < \kappa \} \) of it which does not admit a \( {<} \kappa \)-sized 
subcover. Without loss of generality, we may assume that 
\( U_\alpha \not\subseteq \bigcup_{\beta < \alpha} U_\beta \). Then the sets 
\( V_\alpha = U_\alpha \setminus \bigcup_{\beta < \alpha} U_\beta \) form a 
\( \kappa \)-sized partition of \( X \). Since by \( \kappa \)-additivity the \( V_\alpha \)'s are 
clopen, we are done.
\end{proof}

We are now ready to characterize the generalized Baire space \( \pre{\kappa}{\kappa} \) 
in the class of \( \SC_\kappa \)-spaces (compare it with 
Theorem~\ref{thm:classicalchar}\ref{thm:classicalchar-2}).

\begin{theorem} \label{thm:charBairespaceallcardinals}
\label{thm:characterizations of kappa^kappa among superclosed spaces}
Up to homeomorphism, the generalized Baire space \( \pre{\kappa}{\kappa} \) is the 
unique nonempty \( \kappa \)-additive \( \SC_\kappa \)-space for which all \( \kappa \)-Lindel\"of  subsets (equivalently: all \( K_\kappa \)-subsets) have empty interior.
\end{theorem}

\begin{proof}
Clearly, \( \pre{\kappa}{\kappa} \) is a \( \kappa \)-additive \( \SC_\kappa \)-space. 
Moreover, every \( \kappa \)-Lindel\"of subset of \( \pre{\kappa}{\kappa} \) has empty 
interior as otherwise for some
\( s \in \pre{<\kappa}{\kappa} \) the basic clopen set \( \Nbhd_s \) would be 
\( \kappa \)-Lindel\"of as
well, which is clearly false because \( \{ \Nbhd_{s \conc \alpha} \mid \alpha < \kappa \} \) 
is a \( \kappa \)-sized
clopen partition of \( \Nbhd_s \). By Lemma~\ref{lem:lindelKkappa} and the fact that
\( \pre{\kappa}{\kappa} \) is \( \kappa \)-Baire we get that the \( K_\kappa \)-subsets of \( \pre{\kappa}{\kappa} \) have empty interior too.

Conversely, let \( X \) be any nonempty \( \kappa \)-additive \( \SC_\kappa \)-space all of whose \( \kappa \)-Lindel\"of subsets have empty interior.
By Theorem~\ref{thm:charadditiveSCkappaspaces} we may assume that \( X = [T] \) for some superclosed tree \( T \subseteq \pre{< \kappa}{\kappa} \): our aim is to define a homeomorphism 
between \( \pre{\kappa}{\kappa} \) and \( [T] \).
We recursively define a map \( \varphi \colon \pre{<\kappa}{\kappa} \to T \) by setting 
\( \varphi(\emptyset) = \emptyset \) and 
$\varphi(t)=\bigcup_{\alpha<\leng(t)}\varphi(t \restriction \alpha)$ if $\leng(t)$ is limit 
(this is still a sequence in \( T \) because the latter is \( {<} \kappa \)-closed). For the 
successor step, assume that $\varphi(t)$ has already  been defined. Notice that 
$\Nbhd_{\varphi(t)}\cap [T]$ is open and nonempty (because \( T \) is superclosed), hence 
it is not $\kappa$-Lindel\"of by assumption. 
By Lemma~\ref{split closed set into kappa pieces} there is a \( \kappa \)-sized partition of 
\( \Nbhd_{\varphi(t)}\cap [T] \) into nonempty clopen sets, which can then be refined to a partition in nonempty sets
of the form \( \{ \Nbhd_{t_\alpha} \cap [T] \mid \alpha < \kappa \} \), so that in particular $t_\alpha \in T$: set 
\( \varphi( t \conc \alpha) = t_\alpha \).
It is now easy to see that the function
\[
f_\varphi \colon \pre{\kappa}{\kappa} \to X, \quad x \mapsto \bigcup\nolimits_{\alpha < \kappa} \varphi(x \restriction \alpha)
\]
induced by \( \varphi \) is a homeomorphism between $\pre{\kappa}{\kappa}$ and $X$.
\end{proof}

Theorem~\ref{thm:characterizations of kappa^kappa among superclosed spaces} can be 
used to get an easy proof of the fact that \( \pre{\kappa}{2} \) is homeomorphic to 
\( \pre{\kappa}{\kappa} \) when \( \kappa \) is not weakly compact, i.e.\ when 
\( \pre{\kappa}{2} \) is not \( \kappa \)-Lindel\"of itself. Indeed, \( \pre{\kappa}{2} \) is 
clearly a nonempty \( \kappa \)-additive \( \SC_\kappa \)-space, so it is enough to check 
that all its \( \kappa \)-Lindel\"of subsets have empty interior. But for zero-dimensional 
spaces this is equivalent to the fact that every nonempty open subspace is not 
\( \kappa \)-Lindel\"of, which in this case is true because all basic open subsets of 
\( \pre{\kappa}{2} \) are homeomorphic to it, and thus they are not \( \kappa \)-Lindel\"of.

\bigskip

We next move to the characterization(s) of \( \pre{\kappa}{2} \). When \( \kappa \) is not 
weakly compact, 
Theorem~\ref{thm:characterizations of kappa^kappa among superclosed spaces}  
already does the job, but we are anyway seeking a generalization along the lines  of 
Brouwer's  characterization of \( \pre{\omega}{2} \) from Theorem~\ref{thm:classicalchar}\ref{thm:classicalchar-1} (thus involving perfectness and suitable compactness properties). 
Since \( \pre{\kappa}{2} \) is \( \kappa \)-Lindel\"of if and only if \( \kappa \) is weakly 
compact, we distinguish between the corresponding two cases and first concentrate  on 
the case when  \( \kappa \) is not weakly compact.
In this situation,
 there is no space at all sharing all (natural generalizations of) the conditions appearing 
 in Theorem~\ref{thm:classicalchar}\ref{thm:classicalchar-1}.

\begin{proposition}\label{prop:no_perfect_kappa_additive_SC_Lindelof}
Let $\kappa$ be a non weakly compact cardinal. Then there is no nonempty  
$\kappa$-additive (\( \kappa \)-)perfect $\kappa$-Lindel\"of $\SC_\kappa$-space.
\end{proposition}

\begin{proof}
Suppose towards a contradiction that there is such a space $X$. By 
Theorem~\ref{thm:PSP for perfect SC_kappa kappa-additive}, we could then find a 
homeomorphic copy \( C \subseteq X \) of \( \pre{\kappa}{2} \) with \( C \) closed in \( X \). 
But then \( C \), and hence also \( \pre{\kappa}{2} \), would be \( \kappa \)-Lindel\"of, 
contradicting the fact that \( \kappa \) is not weakly compact.
\end{proof}

Proposition~\ref{prop:no_perfect_kappa_additive_SC_Lindelof} seems to suggest the 
we already reached a dead end in our attempt to generalize Brouwer's theorem for 
non-weakly compact cardinals. However, this is quite not true: we are now going to show 
that relaxing even just one of the conditions on the space give a compatible set of 
requirements.
For example, if we restrict the attention to \( \kappa \)-Lindel\"of \( \SC_\kappa \)-spaces, 
then \( \kappa \)-additivity and \( \kappa \)-perfectness cannot coexists by 
Proposition~\ref{prop:no_perfect_kappa_additive_SC_Lindelof}, but they can be 
satisfied separately. 
Indeed, the space
\[
X = \{ x \in \pre{\kappa}{2} \mid x(\alpha) = 0 \text{ for at most one } \alpha < \kappa \}
 \]
is a $\kappa$-additive $\kappa$-Lindel\"of $\SC_\kappa$-space, while endowing 
\( \pre{\kappa}{2} \) with the product topology (rather than the bounded topology) 
we get a $\kappa$-perfect $\kappa$-Lindel\"of (in fact, compact) $\SC_\kappa$-space.
If instead we weaken the Choquet-like condition to being just a \( \SFC_\kappa \)-space, 
then we have the following example.

\begin{proposition}
There exists a nonempty $\kappa$-additive (\( \kappa \)-)perfect $\kappa$-Lindel\"of $\SFC_\kappa$-space.
\end{proposition}

\begin{proof}
Consider the tree $T_0=\{s\in \pre{<\kappa}{2}\mid |\{\alpha\mid s(\alpha)=0\}|<\omega\}$ and the space $X_0=[T_0]$ from equation~\eqref{eq:finitelymanyzeroes}, which is clearly  a $\kappa$-additive (\( \kappa \)-)perfect $\SFC_\kappa$-space. 
Suppose towards a contradiction that $X_0$ is not $\kappa$-Lindel\"of, and let $\F$ be a \( \kappa \)-sized clopen  partition of \( X_0 \) in nonempty pieces, which exists by Lemma~\ref{split closed set into kappa pieces}.
Without loss of generality, we may assume that each set in $\F$ is of the form $\Nbhd_s\cap [T_0]$ for some $s\in T_0$. Set $F=\{s\in T_0\mid \Nbhd_s\cap [T_0]\in \F\}$: then $F$ is a maximal antichain in $T_0$, i.e.\ distinct $s,t\in F$ are incomparable and for each $x\in [T_0]$ there is $s\in F$ such that $s\subseteq x$.
By definition, each sequence $s\in F$ has only a finite number of coordinates with value $0$: for each $n\in\omega$, let $F_n$ be the set of those $s\in F$ that have exactly $n$-many zeros.
Since $|F|=\kappa$ and $\{F_n\mid n\in\omega\}$ is a partition of $F$, there exists some $n$ such that $|F_n|=\kappa$: let $\ell$ be the smallest natural number with this property, 
and set $F_{<\ell}=\bigcup_{n<\ell} F_n$. Then $|F_{<\ell}|<\kappa$ and $\gamma=\sup \{\leng(s)\mid s\in F_{<\ell}\}<\kappa$ by regularity of \( \kappa \).

We claim that there is $s\in F_\ell$ such that $s(\beta)=0$ for some $\gamma\leq \beta<\leng(s)$. If not, the map \( s \mapsto \{ \alpha < \gamma \mid s(\alpha) = 0 \} \) would be an injection (because \( F \) is an antichain) from \( F_\ell \) to \( \{ A \subseteq \gamma \mid |A| = \ell \}\), contradicting $|F_\ell|=\kappa$.
Given now \( s \) as above, let $x=(s \restriction \gamma)\conc 1^{(\kappa)}$. Then $x\in X_0$ and $|\{\alpha < \kappa \mid x(\alpha)=0\}|<\ell$, thus there is $t\in F_{<\ell}$ such that $x\in \Nbhd_{t}\cap [T_0]$. Since \( t \in F_{< \ell} \) implies \( \leng(t) \leq \gamma \),  this means that $t\subseteq x \restriction \gamma =  s \restriction \gamma \subseteq s$, contradicting the fact that $F$ is an antichain.
\end{proof}

The remaining option is to drop the condition of being \( \kappa \)-Lindel\"of. In a sense, 
this is the most promising move, as we are assuming that \( \kappa \) is not weakly 
compact and thus \( \pre{\kappa}{2} \), the space we are trying to characterize, thus not 
satisfy such a property. Indeed, we are now going to show that dropping such a (wrong) 
requirement, we already get the desired characterization.

\begin{lemma}\label{lem:char_perfectness_additive_notweaklycompact} 
Suppose that $\kappa$ is not weakly compact and \( X \) is a $\kappa$-additive \( \SC_\kappa \)-space. Then \( X \) is ($\kappa$-)perfect if and only if every \( \kappa \)-Lindel\"of subsets of \( X \) has empty interior.
\end{lemma}

\begin{proof}
It is clear that if all \( \kappa \)-Lindel\"of subsets of \( X \) have empty interior, then \( X \) has no isolated point because if \( x \in X \) is isolated then \( \{ x \} \) is open and trivially \( \kappa \)-Lindel\"of. 
Suppose now that \( X \) is perfect but has a \( \kappa \)-Lindel\"of subset with nonempty 
interior. By zero-dimensionality, it would follow that there is a nonempty clopen set 
\( O \subseteq X \) which is \( \kappa \)-Lindel\"of. But then \( O \) would be a nonempty 
\( \kappa \)-additive perfect \( \kappa \)-Lindel\"of \( \SC_\kappa \)-space, contradicting 
Proposition~\ref{prop:no_perfect_kappa_additive_SC_Lindelof}.
\end{proof}

Lemma~\ref{lem:char_perfectness_additive_notweaklycompact} allows us to replace the 
last condition in the characterization of \( \pre{\kappa}{\kappa} \) from 
Theorem~\ref{thm:characterizations of kappa^kappa among superclosed spaces} with 
(\( \kappa \)-)perfectness. Together with the fact that \( \pre{\kappa}{\kappa} \) is 
homeomorphic to \( \pre{\kappa}{2} \) when \( \kappa \) is not weakly compact, this leads 
us to the following analogue of Theorem~\ref{thm:classicalchar}\ref{thm:classicalchar-1} 
(which of course can also be viewed as an alternative characterization of \( \pre{\kappa}{\kappa} \)).

\begin{theorem} \label{thm:charCantornonweaklycompact}
Let $\kappa$ be a non weakly compact cardinal. Up to homeomorphism, the generalized Cantor space \( \pre{\kappa}{2} \) (and hence also \( \pre{\kappa}{\kappa} \)) is the unique nonempty \( \kappa \)-additive (\( \kappa \)-)perfect \( \SC_\kappa \)-space.
\end{theorem}

We now move to the case when \( \kappa \) is weakly compact.
In contrast to the previous situation,  the condition of being \( \kappa \)-Lindel\"of 
obviously becomes relevant (and necessary) because \( \pre{\kappa}{2} \) now has  that 
property---this is the only difference between 
Theorem~\ref{thm:charCantornonweaklycompact} and 
Theorem~\ref{thm:charCantorweaklycompact}.

\begin{theorem} \label{thm:charCantorweaklycompact}
\label{characterizations of Cantor among superclosed spaces}
Let $\kappa$ be a weakly compact cardinal.
Up to homeomorphism, the generalized Cantor space \( \pre{\kappa}{2} \) is the unique 
nonempty \( \kappa \)-additive ($\kappa$-)perfect  \( \kappa \)-Lindel\"of \( \SC_\kappa \)-space.
\end{theorem}

\begin{proof}
For the nontrivial direction, let \( X \) be any nonempty perfect \( \kappa \)-additive 
\( \kappa \)-Lindel\"of \( \SC_\kappa \)-space. By 
Lemma~\ref{lem:equivperf2} 
we may assume that \( X =  [T] \) for some splitting superclosed tree \( T \subseteq \pre{\kappa}{\kappa} \). 
Notice that the fact that \( X \) is \( \kappa \)-Lindel\"of entails that \( |\mathrm{Lev}_\alpha(T)| < \kappa \) for all \( \alpha < \kappa \): this will be used in combination with the strong form of the splitting condition from equation~\eqref{eq:strongsplitting} in Remark~\ref{rmk:equivperf2} to prove the following claim.

\begin{claim} \label{claimcharcantor}
For every \( \alpha< \kappa \) there is a nonzero \( \beta < \kappa \) such that \( | \mathrm{Lev}_{\alpha+\beta}(T_t)| = |\pre{\beta}{2} | \) for all \( t \in \mathrm{Lev}_\alpha(T) \).
\end{claim}

\begin{proof}
Recursively define a sequence of ordinals \( (\gamma_n)_{ n \in \omega} \), as follows. Set \( \gamma_0  = 0 \). Suppose now that the \( \gamma_i \) have been defined for all \( i \leq n \), and set \( \bar{\gamma}_n = \sum_{i \leq n} \gamma_i \).  Then choose \( \gamma_{n+1} <\kappa \) large enough to ensure
\begin{enumerate-(1)}
\item \label{con1inclaim}
\( \gamma_{n+1} \geq \max \left\{  2^{|\gamma_n|}, |\mathrm{Lev}_{\alpha + \bar{\gamma}_n}(T)| \right\} \);
\item \label{con2inclaim}
\( |\mathrm{Lev}_{\alpha+\bar{\gamma}_n+\gamma_{n+1}}(T_s)| \geq |\gamma_n| \) for all \( s \in \mathrm{Lev}_{\alpha+\bar{\gamma}_n}(T) \).
\end{enumerate-(1)}
Such a \( \gamma_{n+1} \) exists because \(  |\mathrm{Lev}_{\alpha + \bar{\gamma}_n}(T)| < \kappa \) (because \( X \) is \( \kappa \)-Lindel\"of) and \( 2^{|\gamma_n|} < \kappa \) (because \( \kappa \) is inaccessible). Set \( \beta = \sum_{n \in \omega} \gamma_n =  \sup_{n \in \omega} \bar{\gamma}_n \).
By construction, \( |\pre{\beta}{2}| = \prod_{n \in \omega} 2^{|\gamma_n|} = \prod_{n \in \omega} |\gamma_n| \). On the other hand for every \( t \in \mathrm{Lev}_\alpha(T) \) we have
\[ 
\prod\nolimits_{n \in \omega} |\gamma_n| \leq | \mathrm{Lev}_{\alpha+\beta}(T_t)| \leq | \mathrm{Lev}_{\alpha+\beta}(T)| \leq \prod\nolimits_{n \in \omega} |\gamma_n|,
 \] 
where the first inequality follows from~\ref{con2inclaim} while the last one follows from~\ref{con1inclaim}.
\end{proof}

Using Claim~\ref{claimcharcantor} we can easily construct a club 
\( 0 \in  C \subseteq \kappa \) such that if \( (\alpha_i)_{i < \kappa } \) is the increasing 
enumeration of \( C \) and \( \beta_i \) is such that \( \alpha_{i+1} = \alpha_i + \beta_i \), 
then there is a bijection \( \varphi_t \colon \mathrm{Lev}_{\alpha_{i+1}}(T_t) \to \pre{\beta_i}{2} \) for every \( i < \kappa \) and  \( t \in \mathrm{Lev}_{\alpha_i}(T) \). 

Define
\( \varphi \colon T \to \pre{< \kappa}{2} \) 
by recursion on \( \leng(s) \) as follows. Set \( \varphi(\emptyset) = \emptyset \). For an 
arbitrary \( s \in T \setminus \{ \emptyset \} \), let  \( j < \kappa \) be largest such that 
\( \alpha_j \leq \leng(s) \) (here we use that \( C \) is a club). If \( \alpha_j < \leng(s) \), set 
\( \varphi(s) = \varphi(s \restriction \alpha_j) \). If instead \( \alpha_j = \leng(s) \), then we 
distinguish two cases. If \( j = i+1 \) we set \( \varphi(s) = \varphi(s \restriction \alpha_i) \conc \varphi_{s \restriction \alpha_i}(s) \); if \( j \) is a limit ordinal (so that \( \leng(s) \) is a 
limit as well), we set \( \varphi(s) = \bigcup_{ \beta < \leng(s)}  \varphi( s \restriction \beta) \).

It is clear that \( \varphi \) is 
\( \subseteq \)-monotone and for all \( \alpha  \in C \) the restriction of \( \varphi \) to 
\( \mathrm{Lev}_\alpha(T) \) is a bijection with \( \pre{\alpha}{2} \). It easily follows that
\[ 
f_\varphi \colon [T] \to \pre{\kappa}{2}, \quad x \mapsto \bigcup\nolimits_{\alpha < \kappa} \varphi(x \restriction \alpha)
 \] 
is a homeomorphism, as required.
\end{proof}

The proof of the nontrivial direction in Theorem~\ref{thm:charCantorweaklycompact} 
just requires \( \kappa \) to be inaccessible (and not necessarily weakly compact). The 
stronger hypothesis on \( \kappa \) in the statement is indeed due to the other direction: 
if \( \kappa \) is not weakly compact, then \( \pre{\kappa}{2} \) is not 
\( \kappa \)-Lindel\"of and, indeed, by 
Proposition~\ref{prop:no_perfect_kappa_additive_SC_Lindelof} there are no spaces at 
all as in the statement.

\begin{remark} \label{rmk:transferofsuperclosed}
It is easy to check that the function \( f_\varphi \) constructed in the previous proof 
preserves superclosed sets, that is, it is such that \( C \subseteq [T] \) is superclosed if 
and only if \( f_\varphi(C) \subseteq \pre{\kappa}{2} \) is superclosed. This follows from 
the fact that if \( S \) is a superclosed subtree of \( T \), then the 
\( \subseteq \)-downward closure of \( \varphi(S) \) is a superclosed subtree \( S' \) of 
\( \pre{\kappa}{2} \); conversely, if \( S' \subseteq \pre{< \kappa}{2} \) is a superclosed 
tree, then \( S = \{ t \in T \mid \varphi(t) \in S' \} \) is a superclosed subtree of \( T \).
\end{remark}

In view of Theorem~\ref{thm:SC-space+G-Polish}, most of the characterizations provided 
so far can equivalently be rephrased in the context of \( \GG \)-Polish spaces. For 
example, the following is the characterization of the generalized Cantor and Baire 
spaces in terms of $\GG$-metrics.

\begin{theorem} \ 
\begin{enumerate-(1)} 
\item
Up to homeomorphism, the generalized Cantor space \( \pre{\kappa}{2} \) is the unique 
nonempty  \mbox{(\( \kappa \)-)perfect} (\( \kappa \)-Lindel\"of, if \( \kappa \) is weakly 
compact) spherically complete \( \GG \)-Polish space.
\item
Up to homeomorphism, the generalized Baire space \( \pre{\kappa}{\kappa} \) is the 
unique nonempty spherically complete \( \GG \)-Polish space for which all 
\( \kappa \)-Lindel\"of  subsets (equivalently: all \( K_\kappa \)-subsets) have empty interior.
\end{enumerate-(1)}
\end{theorem}

In this section we studied in detail the $\kappa$-Lindel\"of property in relation with the 
generalized Cantor space: it turns out that this property has important consequences  
for other spaces as well.
For example, as it happens in the classical case, compactness-like properties always bring with them some
form of completeness. 

\begin{proposition}\label{prop:K_Lindelof_implies_SFC}
Let $X$ be a space of weight $\leq \kappa$. If $X$ is $\kappa$-Lindel\"of, then it is an $\SFC_\kappa$-space.
\end{proposition}

\begin{proof}
Define a strategy $\sigma$ for $\pII$ such that when $\pI$ plays a (relatively) open set 
$U$ and a point $x\in U$, then $\sigma$ answers with any (relatively) open set $V$ 
satisfying $x\in V$ and $\cl(V)\subset U$ (such a \( V \) exists by regularity).
Suppose that $\langle (U_\alpha, x_\alpha), V_\alpha\mid \alpha<\kappa\rangle$ is a run 
of the strong fair $\kappa$-Choquet game played accordingly to $\sigma$
in which player \( \pI \) ensured $\bigcap_{\alpha< \gamma} V_\alpha \neq \emptyset$ for all limit $\gamma < \kappa$. 
If 
$\bigcap_{\alpha<\kappa} V_\alpha=\emptyset$, then the family $\{X\setminus \cl(V_\alpha) \mid \alpha<\kappa\}$ would be an open cover of $X$
because \( \bigcap_{\alpha < \kappa} \cl(V_\alpha) = \bigcap_{\alpha< \kappa} U_\alpha = \bigcap_{\alpha < \kappa} V_\alpha = \emptyset \), and thus it would have a subcover of size 
$<\kappa$ because \( X \) is \( \kappa \)-Lindel\"of.  But then there would be \( \delta < \kappa \) 
such that \( \bigcap_{\alpha < \delta'} \cl(V_\alpha) =\emptyset \) for all \( \delta \leq \delta' < \kappa \). Considering any limit ordinal $\gamma < \kappa$ with $\gamma \geq \delta$, we then would get 
\( \bigcap_{\alpha < \delta'} V_\alpha = \bigcap_{\alpha < \delta'} \cl(V_\alpha) = \emptyset \), 
against our assumptions on the given run.
\end{proof}

The following is the analogue in our context of the standard fact that compact metrizable spaces are automatically Polish.

\begin{corollary}
Every \( \kappa \)-Lindel\"of \( \GG\)-metrizable space is \( \GG \)-Polish.
\end{corollary}

\begin{proof}
Choose a strictly decreasing sequence \( (\varepsilon_\alpha)_{\alpha < \kappa} \) 
that is coinitial in \( \GG^+ \). By \( \kappa \)-Lindel\"ofness, for each \( \alpha < \kappa \) there 
is a covering \( \mathcal{B}_\alpha \) of \( X \) of size \( < \kappa \) consisting of open 
balls of radius \( \varepsilon_\alpha \). It follows that \( \mathcal{B} = \bigcup_{\alpha < \kappa} \mathcal{B}_\alpha \) is a basis for \( X \) of size \( \leq \kappa \). By 
Proposition~\ref{prop:K_Lindelof_implies_SFC} the space \( X \) is then strong fair \( \kappa \)-Choquet, 
and since \( \GG \)-metrizability implies \( \kappa \)-additivity we get that \( X \) is 
\( \GG \)-Polish by Theorem~\ref{thm:charK-Polish}.
\end{proof}

Using Proposition~\ref{prop:K_Lindelof_implies_SFC}, many statements of 
Section~\ref{sec:relationship_SC_SFC_Polish} can be reformulated for the special case of 
weakly compact cardinals and $\kappa$-Lindel\"of spaces.
For example, the next proposition is a reformulation of 
Proposition~\ref{prop:charadditiveSCandSFCkappaspaces} in this special case.

\begin{proposition}\label{prop:charadditiveSCandSFCkappaspaces+k_Lindelof}
Let $X$ be a \( \kappa \)-additive $\kappa$-Lindel\"of space of weight $\leq \kappa$ (in 
which case \( X \) is automatically an $\SFC_\kappa$-space by 
Proposition~\ref{prop:K_Lindelof_implies_SFC}). Then \( X \) is homeomorphic to a 
closed set \( C \subseteq \pre{\kappa}{2} \).
If furthermore $X$ is an \( \SC_\kappa \)-space, then \( C \) can be taken to be 
superclosed.
\end{proposition}

\begin{proof}
First notice that if $\kappa$ is not weakly compact, then the result trivially holds by 
Proposition~\ref{prop:charadditiveSCandSFCkappaspaces} since in this case 
$\pre{\kappa}{2}$ and $\pre{\kappa}{\kappa}$ are homeomorphic (via a homeomorphism 
which preserves superclosed sets). Thus we may assume that $\kappa$ is weakly 
compact. 
By Theorems~\ref{thm:charadditivefairSCkappaspaces} 
and~\ref{thm:charadditiveSCkappaspaces} we can further assume that \( X = [T] \) 
for some (superclosed, in the case of an \( \SC_\kappa \)-space) tree 
\( T \subseteq \pre{< \kappa}{\kappa} \). Since \( X \) is \( \kappa \)-Lindel\"of, 
by~\cite[Lemma 2.6(1)]{LuckeMottSchlMR3557473} the set \( [T] \) is bounded, i.e.\ there 
is \( y \in \pre{\kappa}{\kappa} \) such that \( x(\alpha) \leq y(\alpha) \) for all \( x \in [T] \) 
and \( \alpha < \kappa \). 
We can assume $y(\alpha)>0$ for all $\alpha<\kappa$. 
Consider the space 
\( Z = \{ z \in \pre{\kappa}{\kappa} \mid \forall \alpha < \kappa \, (z(\alpha) \leq y(\alpha) )  \} \). 
It is clearly a nonempty \( \kappa \)-additive \( \kappa \)-perfect \( \SC_\kappa \)-space. 
Moreover, since by definition it is bounded by \( y \) and \( \kappa \) is weakly compact, 
by~\cite[Lemma 2.6(1)]{LuckeMottSchlMR3557473} and the fact that \( Z \) is closed in 
\( \pre{\kappa}{\kappa} \) it follows that \( Z \) is also \( \kappa \)-Lindel\"of. By 
Theorem~\ref{thm:charCantorweaklycompact} there is a homeomorphism \( h \colon Z \to \pre{\kappa}{2} \), which moreover preserves superclosed subsets of \( Z \) by 
Remark~\ref{rmk:transferofsuperclosed}.
Since by definition \( X \subseteq Z \), it follows that \( h(X) \) is a (super)closed subset of 
\( \pre{\kappa}{2} \) homeomorphic to \( X \), as required.
\end{proof}

Using Proposition~\ref{prop:charadditiveSCandSFCkappaspaces+k_Lindelof}, we can 
restate Theorems~\ref{thm:charK-Polish} and~\ref{thm:SC-space+G-Polish} for the 
special case of $\kappa$-Lindel\"of spaces, further refining the picture given in 
Figure~\ref{fig:sum_up_relationships} with one more dividing line, namely \( \kappa \)-Lindel\"ofness.

\begin{theorem} \label{thm:charK-Polish-Lindelof}
For any space \( X \) the following are equivalent:
\begin{enumerate-(a)}
\item \label{thm:charK-Polish-Lindelof-1}
\( X \) is a \( \kappa \)-Lindel\"of \( \kappa \)-additive space of weight \( \leq \kappa \);
\item \label{thm:charK-Polish-Lindelof-3}
\( X \) is a \( \kappa \)-Lindel\"of \( \GG \)-metrizable space;
\item \label{thm:charK-Polish-Lindelof-4}
\( X \) is a \( \kappa \)-Lindel\"of \( \GG \)-Polish space;
\item \label{thm:charK-Polish-Lindelof-2}
\( X \) is a \( \kappa \)-Lindel\"of \( \kappa \)-additive $\SFC_\kappa$-space.
\end{enumerate-(a)}
If \( \kappa \) is weakly compact, the above conditions are also equivalent to:
\begin{enumerate}[label={\upshape (e)}, leftmargin=2pc]
\item \label{thm:charK-Polish-Lindelof-5}
\( X \) is homeomorphic to a closed subset of \( \pre{\kappa}{2} \).
\end{enumerate}
\end{theorem}

\begin{theorem} \label{thm:char-spherically-complete-K-Polish-Lindelof}
For any space \( X \) the following are equivalent:
\begin{enumerate-(a)}
\item \label{thm:char-spherically-complete-K-Polish-Lindelof-1} \( X \) is a \( \kappa \)-Lindel\"of \( \kappa \)-additive $\SC_\kappa$-space;
\item \label{thm:char-spherically-complete-K-Polish-Lindelof-2} \( X \) is a \( \kappa \)-Lindel\"of spherically \( {<} \kappa \)-complete \( \GG \)-metrizable space;
\item \label{thm:char-spherically-complete-K-Polish-Lindelof-3} \( X \) is a \( \kappa \)-Lindel\"of spherically complete \( \GG \)-Polish space.
\end{enumerate-(a)}
If \( \kappa \) is weakly compact, the above conditions are also equivalent to:
\begin{enumerate}[label={\upshape (d)}, leftmargin=2pc]
\item \label{thm:char-spherically-complete-K-Polish-Lindelof-4} \( X \) is homeomorphic to a superclosed subset of \( \pre{\kappa}{2} \).
\end{enumerate}
\end{theorem}

\section{Characterizations of standard $\kappa$-Borel spaces} 
\label{sec:standardBorel}

In this section we deal with the $\kappa$-Borel structure of topological spaces, and 
show how standard $\kappa$-Borel spaces (Definition~\ref{def:standardBorel}) are 
exactly the $\kappa$-Borel spaces obtained from  Polish-like spaces in any of the classes 
considered so far by forgetting their topology.
For the sake of definiteness,
throughout the section we work with $\SFC_\kappa$-spaces and $\SC_\kappa$-spaces, 
but all results can be reformulated in terms of $\GG$-Polish and spherically complete 
$\GG$-Polish spaces---see Section~\ref{sec:relationship_SC_SFC_Polish}.

We start by proving some results about changes of topology, which might be of 
independent interest. The next proposition shows how to change the topology of an 
$\SFC_\kappa$-space while preserving its $\kappa$-Borel structure. This 
generalizes~\cite[Theorem 13.1]{KechrisMR1321597} to our setup.

\begin{proposition}\label{prop:changeoftopology}
Let \( (B_\alpha)_{\alpha < \kappa} \) be a family of \( \kappa \)-Borel subsets of an 
$\SFC_\kappa$-space \( (X,\tau) \). Then there is a topology \( \tau' \) on $X$ such that:
\begin{enumerate-(1)}
\item \label{prop:changeoftopology-1}
\( \tau' \) refines \( \tau \);
\item \label{prop:changeoftopology-2}
each \( B_\alpha \) is \( \tau' \)-clopen,
\item \label{prop:changeoftopology-3}
\( \Bor_\kappa(X,\tau') = \Bor_{\kappa}(X, \tau) \), and
\item \label{prop:changeoftopology-4}  \( (X,\tau') \) is a $\kappa$-additive $\SFC_\kappa$-space.
\end{enumerate-(1)}
\end{proposition}

\begin{proof}
Let \( \mathscr{A} \) be the collection of those \( A \subseteq X \) for which there is a 
topology \( \tau' \) which satisfies~\ref{prop:changeoftopology-1}--\ref{prop:changeoftopology-4} above (where in~\ref{prop:changeoftopology-2} the set \( B_\alpha \) is replaced by \( A \)).
Notice that \( \mathscr{A} \) is trivially closed under complements. We first show that $\mathscr{A}$ contains all closed subsets of $X$.

\begin{claim}\label{claim:closed}
Let $C$ be a closed subset of an $\SFC_\kappa$-space $(X,\tau)$. Then there is a topology \( \tau' \) which satisfies~\ref{prop:changeoftopology-1}--\ref{prop:changeoftopology-4} above (where in~\ref{prop:changeoftopology-2} the set \( B_\alpha \) is replaced by \( C \)).
\end{claim}

\begin{proof}[Proof of the Claim]
Let \( \bar{\tau} \) be the smallest topology generated by $\tau\cup \{C\}$. 
Then
\( \bar{\tau} \supseteq \tau \) and \( C \) is \( \bar{\tau} \)-clopen. Moreover, 
\( \bar{\tau} \subseteq \Bor_\kappa(X, \tau) \) (and hence \( \Bor_\kappa(X, \bar{\tau}) = \Bor_\kappa(X,\tau) \)) because
each \(\bar{\tau} \)-open set is of the form \( (U \cap C) \cup (V \setminus C) \) for some \( U,V \in \tau \). This also shows that the identity map is a homeomorphism between $(X, \bar{\tau})$ and 
the sum of the spaces \( C \) 
and \(  X \setminus C \) (endowed with the relative topologies inherited from $X$). Since 
both $C$ and $X \setminus C$ are $\SFC_\kappa$-spaces by
Theorem~\ref{thm:charadditivefairSCkappaspaces}, and since the class of 
$\SFC_\kappa$-spaces is trivially closed under (\( \leq \kappa \)-sized) sums, then $X$ is 
an $\SFC_\kappa$-space as well. Applying 
 Proposition~\ref{prop:refintetoGpolish} to \( (X, \bar{\tau} ) \) we then get  a topology 
 \( \tau' \supseteq \bar{\tau} \supseteq \tau \) which satisfies all 
 of~\ref{prop:changeoftopology-1}--\ref{prop:changeoftopology-4}.
\end{proof}

\begin{claim}\label{claim:family}
Let \( (A_\alpha)_{\alpha<\kappa} \) be a family of sets in \( \mathscr{A} \).  
Then there is a topology $\tau'_\infty$ simultaneously witnessing \( A_\alpha \in \mathscr{A} \) for all \( \alpha < \kappa \).
\end{claim}

\begin{proof}[Proof of the Claim]
For each \( \alpha < \kappa \) let \( \tau'_\alpha \) be a topology witnessing \( A_\alpha \in \mathscr{A} \). Define \( \tau'_\infty \) as the smallest $\kappa$-additive topology containing \( \bigcup_{\alpha < \kappa} \tau'_\alpha \).
Then \ref{prop:changeoftopology-1}--\ref{prop:changeoftopology-3} are obvious, since 
\( \tau'_\infty \) refines  each \( \tau'_\alpha \supseteq \tau \) and $\tau'_\infty\subset \Bor_\kappa(X,\tau)$.
(The last inclusion holds because $\tau'_\infty$ is generated by sets of the form $\bigcap_{\beta < \gamma} U_\beta$ with $\gamma < \kappa$ and $U_\beta \in \bigcup_{\alpha < \kappa} \mathcal{B}_\alpha$ for $\mathcal{B}_\alpha$ a $\leq \kappa$-sized basis for $\tau'_\alpha$: each of these $\kappa$-many sets is thus $\kappa$-Borel with respect to $\tau$ by the properties of the chosen topologies $\tau'_\alpha$.)
To prove~\ref{prop:changeoftopology-4},  for each \( \alpha < \kappa \)
fix a closed \( C_\alpha \subseteq \pre{\kappa}{\kappa} \) and a homeomorphism
\( h_\alpha \colon C_\alpha \to (X, \tau'_\alpha) \) as given by 
Theorem~\ref{thm:charadditivefairSCkappaspaces}. Endow \( {}^{\kappa}(\pre{\kappa}{\kappa}) \)
with the \( \kappa \)-supported product topology, i.e.\ the topology generated by the sets
\( \prod_{\alpha< \kappa} U_\alpha \), where each \( U_\alpha \) is open in the bounded topology of $\pre{\kappa}{\kappa}$, and only \( {<} \kappa \)-many of them differ from
\( \pre{\kappa}{\kappa} \). 
Then \( \prod_{\alpha < \kappa} C_\alpha \) is closed in ${}^{\kappa}(\pre{\kappa}{\kappa})$, and since the maps \( h_\alpha \) are continuous, 
the set
\[
\Delta = \left\{ (x_\alpha)_{\alpha < \kappa} \in \prod_{\alpha < \kappa} C_\alpha \Bigm| \forall \alpha, \beta < \kappa \, \Big(h_\alpha(x_\alpha) = h_\beta(x_\beta)\Big) \right\}
 \]
is closed as well.
It is then easy to check that the map \( h \colon \Delta \to (X, \tau'_\infty) \) sending
\( (x_\alpha)_{\alpha<\kappa} \in \Delta \) to \( h_0(x_0) \) is a homeomorphism. 
Therefore the desired result follows from 
Theorem~\ref{thm:charadditivefairSCkappaspaces} and the fact that the spaces
\(  {}^{\kappa}(\pre{\kappa}{\kappa}) \) and \( \pre{\kappa}{\kappa} \) are clearly homeomorphic. 
\end{proof}

Claim~\ref{claim:family} in particular reduces our task of proving the theorem for a 
whole family \( (B_\alpha)_{\alpha < \kappa } \) to  showing that \( B \in \mathscr{A} \) for 
every single \( \kappa \)-Borel set \( B \subseteq X \). To this aim,
by Claim~\ref{claim:closed} and closure of \( \mathscr{A} \) under complements it is
enough to show that \( \mathscr{A} \) is closed under intersections of length 
\( \leq \kappa \). So let \( A = \bigcap_{\alpha < \kappa} A_\alpha \) be such that 
\( A_\alpha \in \mathscr{A} \) for every $\alpha<\kappa$. By Claim~\ref{claim:family}, 
there is a topology \( \tau'_\infty \) simultaneously witnessing 
\( A_\alpha \in \mathscr{A} \) for all \( \alpha < \kappa \). Then $A$ is closed in the 
\( \kappa \)-additive $\SFC_\kappa$-space $(X,\tau'_\infty)$. Therefore 
Claim~\ref{claim:closed} applied to $A$, viewed as a subset of $(X,\tau'_\infty)$, 
yields the desired topology $\tau' \supseteq \tau'_\infty \supseteq \tau$.
\end{proof}

Proposition~\ref{prop:changeoftopology} provides an alternative proof
of~\cite[Lemma 1.11]{LuckeSchlMR3430247}:  
Every $\kappa$-Borel subset of \( \pre{\kappa}{\kappa} \) is a continuous injective image of a closed subset of \( \pre{\kappa}{\kappa} \). 
To see this, let \( B \subseteq \pre{\kappa}{\kappa} \) be \( \kappa \)-Borel, and let \( \tau' \) be the topology
obtained by applying Proposition~\ref{prop:changeoftopology} with \( B_\alpha = B \) 
for all \( \alpha < \kappa \). 
Let $D$ be a closed subset of $\pre{\kappa}{\kappa}$ and \( h \colon (D, \tau_b) \to (\pre{\kappa}{\kappa},\tau')  \) be a homeomorphism as given by 
Proposition~\ref{prop:charadditiveSCandSFCkappaspaces}. Then $C = h^{-1}(B)$ is 
closed in $D$ and hence in $\pre{\kappa}{\kappa}$.
Moreover, the map \( h' \colon (D, \tau_b) \to (\pre{\kappa}{\kappa},\tau_b)  \) obtained 
by composing \( h \) with the identity function $(\pre{\kappa}{\kappa}, \tau') \to (\pre{\kappa}{\kappa}, \tau_b)$ is still a continuous bijection because 
\( \tau' \supseteq \tau_b \).
Therefore, \( h' \restriction C \) is a continuous injection from the closed set 
\( C  \subseteq \pre{\kappa}{\kappa} \) onto \( B \). Notice also that, by construction, 
\( h' \) is actually a \( \kappa \)-Borel isomorphism because \( \Bor_\kappa(\pre{\kappa}{\kappa}, \tau') = \Bor_\kappa(\pre{\kappa}{\kappa}, \tau_b) \). 
More generally, the same argument shows 
that~\cite[Lemma 1.11]{LuckeSchlMR3430247} can be extended to arbitrary 
$\SFC_\kappa$-spaces.

\begin{corollary} \label{cor:borelasinjimage}
For every \( \kappa \)-Borel subset \( B \) of an $\SFC_\kappa$-space there is a 
continuous \( \kappa \)-Borel isomorphism from a closed 
\( C \subseteq \pre{\kappa}{\kappa} \) onto \( B \).
\end{corollary}

The space \( C \) in the previous corollary is an \( \SFC_\kappa \)-space by 
Theorem~\ref{thm:charadditivefairSCkappaspaces}, hence applying 
Theorem~\ref{thm:char_up_to_kBorel_iso} we further get

\begin{corollary} \label{cor:superclosed}
Each \( \kappa \)-Borel subset \(B \) of an $\SFC_\kappa$-space  is \( \kappa \)-Borel 
isomorphic to a $\kappa$-additive $\SC_\kappa$-space.
\end{corollary}

The following is the counterpart of Proposition~\ref{prop:changeoftopology} in terms of 
functions and can be proved by applying it to the preimages of the open sets in any 
\( \leq \kappa \)-sized basis for the topology of \( Y \).

\begin{corollary} \label{cor:changeoftopology}
Let \( (X,\tau) \) be an $\SFC_\kappa$-space and \( Y \) be any space of weight 
\(\leq \kappa \). Then for every \( \kappa \)-Borel function \( f \colon X \to Y \) there is a 
topology \( \tau' \) on \( X \) such that:
\begin{enumerate-(1)}
\item \label{cor:changeoftopology-1}
\( \tau' \) refines \( \tau \);
\item \label{cor:changeoftopology-2}
\( f \colon (X,\tau') \to Y \) is continuous,
\item \label{cor:changeoftopology-3}
\( \Bor_\kappa(X,\tau') = \Bor_{\kappa}(X, \tau) \), and
\item \label{cor:changeoftopology-4}
\( (X,\tau') \) is a $\kappa$-additive $\SFC_\kappa$-space.
\end{enumerate-(1)}
\end{corollary}

Finally, combining the results obtained so far we get that all the proposed 
generalizations of~\ref{intro:stbor1} and~\ref{intro:stbor2} give rise to the same class of 
spaces. In particular, up to \( \kappa \)-Borel isomorphism such class coincide with any of 
the classes of Polish-like spaces we analyzed in the previous sections. (Notice also that 
Theorem~\ref{th:standardBorel} substantially 
strengthens ~\cite[Corollary 3.4]{CoskSchlMR3453772}.)

\begin{theorem} \label{th:standardBorel}
A \( \kappa \)-Borel space \( (X,\mathscr{B}) \) is standard if and only if there is a topology 
\( \tau \) on \( X \) such that
\begin{enumerate-(1)}
\item \label{th:standardBorel1}
\( (X,\tau) \) is an $\SFC_\kappa$-space, and
\item
\(  \Bor_\kappa(X,\tau) = \mathscr{B} \).
\end{enumerate-(1)}
Moreover, condition~\ref{th:standardBorel1} can equivalently be replaced by 
\begin{enumerate}[label={\upshape (1\( ' \))}, leftmargin=2pc]
\item
\( (X,\tau) \) is a \( \kappa \)-additive $\SC_\kappa$-space.
\end{enumerate}
\end{theorem}

\begin{remark} \label{rmk:standardBorel}
Since \( \kappa \)-additive \( \SC_\kappa \)-spaces and \( \SFC_\kappa \)-spaces form, 
respectively, the smallest and largest class of Polish-like spaces considered in this paper, 
in Theorem~\ref{th:standardBorel} we can further replace those classes with any of the 
other ones: \( \kappa \)-additive \( \SFC_\kappa \)-spaces, \( \SC_\kappa \)-spaces, 
\( \GG \)-Polish spaces, spherically complete \( \GG \)-Polish spaces, and so on.
\end{remark}

A natural question is whether  Proposition~\ref{prop:changeoftopology} can be 
extended in some way or another.
As in the classical case, the answer is mostly negative and thus 
Proposition~\ref{prop:changeoftopology} is essentially optimal. In fact:
\begin{enumerate-(a)}
\item
We cannot in general consider more than \( \kappa \)-many (even closed, or open) 
subsets, since this could force \( \tau' \) to have weight greater than \(  \kappa \)---think 
about 
turning more than \(  \kappa \)-many singletons into clopen sets.
\item \label{limitationtochangeoftopology}
We obviously cannot turn a set which is not \( \kappa \)-Borel into a clopen (or even just 
\( \kappa \)-Borel) one pretending to maintain the same \( \kappa \)-Borel structure. 
Notice however that, in contrast to the classical case, one can consistently have that 
there are non-\( \kappa \)-Borel sets \( B \subseteq \pre{\kappa}{\kappa} \) for which 
there is a \( \kappa \)-additive \( \SFC_\kappa \) topology \( \tau' \supseteq \tau_b \) 
turning \( B \) into a \( \tau' \)-clopen set, so that all conditions in 
Proposition~\ref{prop:changeoftopology} except for~\ref{prop:changeoftopology-3} are 
satisfied with respect to such \( B \) (see 
Corollary~\ref{cor:limitationchangesoftopologies} for more details and limitations). 
\item
By Example~\ref{xmp:no_top_refinement_from_SFC_to_SC}, we cannot require that the 
topology $\tau'$ be $\SC_\kappa$ (instead of just \( \SFC_\kappa \)). The same remains 
true if we consider a single \( \kappa \)-Borel set \( B \) (instead of a whole family 
\( (B_\alpha)_{\alpha < \kappa} \)), we start from the stronger hypothesis that $(X,\tau)$ 
is already a $\kappa$-additive $\SC_\kappa$-space, and we weaken the conclusions by 
dropping condition~\ref{prop:changeoftopology-3} and relaxing 
condition~\ref{prop:changeoftopology-2} to ``\( B \) is \( \tau' \)-open'' (or ``\( B \) is \( \tau' \)-closed'').
\end{enumerate-(a)} 

As it is clear from the discussion, in the last item the problem arises from the fact that 
there is a tension between condition~\ref{prop:changeoftopology-1} and our desire to 
strengthen condition~\ref{prop:changeoftopology-4} from \( \SFC_\kappa \) to 
\( \SC_\kappa \). However, we are now going to show that if we drop the problematic 
condition~\ref{prop:changeoftopology-1}, then it is possible to obtain the desired 
strengthening, at least when we just consider a few \( \kappa \)-Borel sets at a time.

\begin{proposition} \label{prop:changeoftopologyweak}
For every \( \kappa \)-Borel subset \( B \) of an $\SFC_\kappa$-space \(  (X,\tau) \) there is a topology \( \tau'' \) on \( X \) such that:
\begin{enumerate-(1)}
\item \label{prop:changeoftopologyweak-1}
\( B \) is \( \tau'' \)-clopen,
\item \label{prop:changeoftopologyweak-2}
\( \Bor_\kappa(X,\tau'') = \Bor_{\kappa}(X, \tau) \), and
\item \label{prop:changeoftopologyweak-3}
\( (X,\tau'') \)  is a \( \kappa \)-additive $\SC_\kappa$-space (hence so are its subspaces \( B \) and \( X \setminus B \) because they are \( \tau'' \)-open).
\end{enumerate-(1)}
\end{proposition}

\begin{proof}
By Corollary~\ref{cor:superclosed}, there are $\kappa$-additive $\SC_\kappa$ topologies $\tau_1$ and \( \tau_2 \) on, respectively, $B$ and \( X \setminus B \) such that
\( \Bor_\kappa(B, \tau_1) = \Bor_\kappa(X, \tau) \restriction B \) and
\( \Bor_\kappa(X \setminus B, \tau_2) = \Bor_\kappa(X, \tau) \restriction (X \setminus B) \).
 Let $\tau''$ be the topology on \( X \) construed as
 the sum of $(B,\tau_1)$ and $(X\setminus B,\tau_2)$: then $\tau''$ is as required.
\end{proof}

The proof of Proposition~\ref{prop:changeoftopologyweak} can easily be adapted to work with 
\( \kappa \)-many \emph{pairwise disjoint} \( \kappa \)-Borel subsets of \( X \). This in 
turn implies that the proposition can e.g.\  be extended to deal with
\( {<} \kappa \)-many \( \kappa \)-Borel sets simultaneously, even when such sets are not pairwise disjoint. Indeed, if \( (B_\alpha)_{\alpha < \nu} \) with \( \nu < \kappa \) is such a family, then for each \(s \in \pre{\nu}{2} \) we can set
\[ 
C_s = \{ x \in X \mid \forall \alpha < \nu \, (x \in B_\alpha \iff s(\alpha) =1) \}.
 \] 
Since \( 2^\nu \leq \kappa^{< \kappa} = \kappa \), the family \( (C_s)_{s \in \pre{\nu}{2}} \) 
is a partition of \( X \) into \( \leq \kappa \)-many \( \kappa \)-Borel sets, and any topology 
\( \tau'' \) working simultaneously for all the \( C_s \) will work for all sets in the family 
\( (B_\alpha)_{\alpha < \nu} \) as well.
In contrast, Proposition~\ref{prop:changeoftopologyweak} cannot be extended to arbitrary \( \kappa \)-sized families of \( \kappa \)-Borel sets, even when we restrict to \( X = \pre{\kappa}{\kappa} \).
Indeed, let $C\subset \pre{\kappa}{\kappa}$ be as in Example~\ref{xmp:no_top_refinement_from_SFC_to_SC}
 and let $(B_\alpha)_{\alpha<\kappa}$ be an enumeration of $\{C\} \cup \{ \Nbhd_s \cap C \mid s \in \pre{<\kappa}{\kappa} \}$. 
Then $(B_\alpha)_{\alpha<\kappa}$ is a family of Borel subsets of 
$\pre{\kappa}{\kappa}$ 
such that there is no $\SC_\kappa$ topology $\tau''$ on $\pre{\kappa}{\kappa}$ making 
each $B_\alpha$ a $\tau''$-open subset of $\pre{\kappa}{\kappa}$, since otherwise 
$\tau'' \restriction C$ would be an $\SC_\kappa$ topology on $C$ refining $\tau_b \restriction C$.

From a different perspective, it might be interesting to understand which
subspaces of a Polish-like space inherit a standard \( \kappa \)-Borel structure form their 
superspace. 
 Of course this includes all \( \kappa \)-Borel sets, as standard \( \kappa \)-Borel spaces 
 are closed under \( \kappa \)-Borel subspaces, and we are now going to show that no 
 other set has such property. We begin with a preliminary result which is of independent 
 interest, as it shows that if a (regular Hausdorff) topology of weight \( \leq \kappa \) induces a standard 
 \( \kappa \)-Borel structure, then it can be refined to a Polish-like topology with the 
 same \( \kappa \)-Borel sets.

\begin{proposition}\label{prop:condition_for_std_Borel_topology}
Let $(X,\tau)$ be a space of weight $\leq \kappa$. Then
$(X,\Bor_\kappa(X,\tau))$ is a standard $\kappa$-Borel space if and only if there is a 
topology $\tau'\supseteq \tau$ such that $(X,\tau')$ is a $\kappa$-additive 
$\SFC_\kappa$-space and $\Bor_\kappa(X,\tau)=\Bor_\kappa(X,\tau')$.
\end{proposition}

\begin{proof}
The backward implication follows from Theorem~\ref{th:standardBorel}.
For the forward implication, suppose that $(X,\Bor_\kappa(X,\tau))$ is standard 
$\kappa$-Borel. By Theorem~\ref{th:standardBorel}, there is a topology $\hat{\tau}$ 
such that $(X,\hat{\tau})$ is an $\SFC_\kappa$-space and 
$\Bor_\kappa(X,\hat{\tau})=\Bor_\kappa(X,\tau)$. The identity function 
$i\colon (X,\hat{\tau})\to (X,\tau)$ satisfies the hypothesis of 
Corollary~\ref{cor:changeoftopology}, hence there is a $\kappa$-additive $\SFC_\kappa$ 
topology $\tau'$ such that \( \Bor_\kappa(X,\tau')=\Bor_\kappa(X,\hat{\tau})=\Bor_\kappa(X,\tau) \) and \( i \colon (X, \tau') \to (X,\tau) \) is continuous, 
which implies $\tau\subseteq \tau'$.
\end{proof}

\begin{theorem}\label{thm:char_complexity_std_Borel_subspaces}
Let $(X,\mathscr{B})$ be a standard $\kappa$-Borel space, and let $A\subset X$. Then $(A,\mathscr{B} \restriction A)$ is a standard $\kappa$-Borel space if and only if $A\in \mathscr{B}$.
\end{theorem}

\begin{proof}
Since $(X,\mathscr{B})$ is standard $\kappa$-Borel, by definition it is $\kappa$-Borel isomorphic to a $\kappa$-Borel subset $B\subset \pre{\kappa}{\kappa}$. Given \( A \in \mathscr{B} \),  
the subspace $(A,\mathscr{B} \restriction A)$ is then $\kappa$-Borel isomorphic to a set in $\Bor_\kappa(B) \subseteq \Bor_\kappa(\pre{\kappa}{\kappa})$, hence $(A,\mathscr{B} \restriction A)$ is standard $\kappa$-Borel.

Conversely, assume that $(A,\mathscr{B} \restriction A)$ is standard $\kappa$-Borel. 
Let $\tau$ be 
a $\kappa$-additive $\SFC_\kappa$ topology on $X$ with 
$\mathscr{B}=\Bor_\kappa(X,\tau)$, as given by Theorem~\ref{th:standardBorel}, and 
use Proposition~\ref{prop:condition_for_std_Borel_topology} to refine the topology 
\( \tau \restriction A \) on \( A \) (which obviously generates 
\( \mathscr{B} \restriction A \)) to a $\kappa$-additive $\SFC_\kappa$ topology  $\tau_A$ 
on $A$ such that $\mathscr{B} \restriction A=\Bor_\kappa(A,\tau_A)$. 
Let $\mathcal{B}$ be a clopen basis for $(A,\tau_A)$ of size $\leq \kappa$. Since 
$\mathcal{B} \subseteq \Bor_\kappa(A,\tau_A)=\Bor_\kappa(X,\tau) \restriction A$, for 
every $U\in \mathcal{B}$ we can find $C_U\in \Bor_\kappa(X,\tau)$ such that 
$U=C_U\cap A$. Without loss of generality, we may assume that the family 
$\mathcal{C}=\{C_U\mid U\in \mathcal{B}\}$ is closed under complements. Define 
$\hat{\tau}$ to be the smallest $\kappa$-additive topology on $X$ containing 
$\tau\cup \mathcal{C}$. Then $\hat{\tau}$ is Hausdorff because it refines $\tau$, it is 
zero-dimensional (and hence regular) because $\tau$ is zero-dimensional and 
$\mathcal{C}$ is closed under complements, and it has weight at most 
$(\kappa+ |\mathcal{C}|)^{<\kappa}=\kappa^{<\kappa}=\kappa$. Hence, by 
Theorem~\ref{thm:charK-metrizable} we have that the space $(X,\hat{\tau})$ is 
$\GG$-metrizable. Furthermore, $\Bor_\kappa(X,\hat{\tau})=\Bor_\kappa(X,\tau)$ 
because \( \tau \subseteq \hat{\tau} \subseteq \Bor_\kappa(X, \tau) \), and $A$ is a  
\( \GG \)-Polish subspace of $(X,\hat{\tau})$ because by construction 
$\hat{\tau} \restriction A=\tau_A$ and \( \tau_A \) is a \( \kappa \)-additive 
\( \SFC_\kappa \) topology. Therefore, by Corollary~\ref{cor:subspace} we have that $A$ 
is a $G^\kappa_\delta$ subset of $(X,\hat{\tau})$, so in particular 
$A\in\Bor_\kappa(X,\hat{\tau})=\Bor_\kappa(X,\tau) = \mathscr{B}$.
\end{proof}

\begin{corollary} \label{cor:borelimages}
Let \( X,Y \) be standard \( \kappa \)-Borel spaces. If \( A \subseteq X \) is \( \kappa \)-Borel and \( f \colon A \to Y \) is a \( \kappa \)-Borel embedding, then \( f(A) \) is \( \kappa \)-Borel in \( Y \).
\end{corollary}

Corollary~\ref{cor:borelimages} is the analogue of the classical fact that an injective 
Borel image of a Borel set is still Borel (see~\cite[Section 15.A]{KechrisMR1321597}). 
Notice however that in the generalized version the hypothesis on \( f \) is stronger: we 
need it to be a \( \kappa \)-Borel embedding, and not just an injective \( \kappa \)-Borel 
map. This is mainly due to the fact that in the generalized context we lack the analogue 
of Luzin's separation theorem. Indeed, one can even 
prove~\cite[Corollary 1.9]{LuckeSchlMR3430247} that there are non-\( \kappa \)-Borel 
sets which are continuous injective images of the whole \( \pre{\kappa}{\kappa} \), hence 
our stronger requirement cannot be dropped.

We finally come to the problem of characterizing which topologies  induce a standard 
\( \kappa \)-Borel structure. Of course this class is larger than the collection of all 
\( \SFC_\kappa \) topologies, even when restricting to the \( \kappa \)-additive case. 
Indeed, the relative topology on any \( \kappa \)-Borel  non-\( G^\kappa_\delta \) 
subspace \( B \subseteq  \pre{\kappa}{\kappa} \) generates a standard \( \kappa \)-Borel 
structure, yet it is not \( \SFC_\kappa \) itself because of  
Theorems~\ref{thm:charK-Polish} and~\ref{thm:Kpolishsubspaces}. On the other hand, if 
a space \( (X,\tau) \) is homeomorphic to a \( \kappa \)-Borel subset of 
\( \pre{\kappa}{\kappa} \), then it clearly generates a standard \( \kappa \)-Borel 
structure by definition. Theorems~\ref{thm:charK-metrizable} 
and~\ref{thm:char_complexity_std_Borel_subspaces} allow us to reverse the implication, 
yielding the desired characterization in the case of \( \kappa \)-additive topologies.
(For the nontrivial direction, use the fact that by Theorem~\ref{thm:charK-metrizable} every \( \kappa \)-additive space of weight \( \leq \kappa \) is, up to homeomorphism, a subspace of \( \pre{\kappa}{\kappa} \).)

\begin{corollary} \label{cor:topgrneratingstandardBorel}
Let $(X,\tau)$ be a $\kappa$-additive space of weight $\leq \kappa$. Then 
$(X,\Bor_\kappa(X,\tau))$ is a standard $\kappa$-Borel space if and only if $(X,\tau)$ is 
\emph{homeomorphic} to a $\kappa$-Borel subset of $\pre{\kappa}{\kappa}$ (or, 
equivalently, of \( \pre{\kappa}{2} \)).
\end{corollary}

In~\cite[Definition 3.6]{MottoRos:2013}, the author considered \emph{topological} spaces \( (X,\tau) \) with weight \( \leq \kappa \) such that the induced \( \kappa \)-Borel structure is \( \kappa \)-Borel isomorphic to a \( \kappa \)-Borel subset of \( \pre{\kappa}{\kappa} \). By Corollary~\ref{cor:topgrneratingstandardBorel} it turns out that when \( \tau \) is regular Hausdorff and \( \kappa \)-additive, a space \( (X,\tau) \) 
satisfies~\cite[Definition 3.6]{MottoRos:2013} if and only if it is homeomorphic (and not just \( \kappa \)-Borel isomorphic) to a \( \kappa \)-Borel subset of \( \pre{\kappa}{\kappa} \).

\section{Final remarks and open questions}

In the classical setup, Polish spaces are closed under countable sums, countable 
products, and \( G_\delta \) subspaces. Moving to the generalized context, 
all classes considered so far are trivially closed under sums of size \( \leq \kappa \). 
However,
by Theorem~\ref{thm:SC-space+G-Polish} the class of \( \SC_\kappa \)-spaces is already 
lacking closure with respect to closed subspaces (even when restricting the attention to  
\( \kappa \)-additive spaces or, equivalently, to spherically complete \( \GG \)-Polish 
spaces). In view of Proposition~\ref{prop:G_deltasubspaces}, the class of 
\( \SFC_\kappa \)-spaces is a more promising option. Indeed, since such class is also 
straightforwardly closed under \( \leq \kappa \)-products, where the product is naturally  
endowed by the \( < \kappa \)-supported product topology, we easily get:

\begin{theorem} \label{thm:closurepropertiesSFC}
The class of \( \SFC_\kappa \)-spaces is closed under \( G^\kappa_\delta \) subspaces and \( \leq \kappa \)-sized sums and products.
\end{theorem}

Moving to \( \GG \)-Polish spaces, by Theorem~\ref{thm:Kpolishsubspaces} we still have 
closure under \( G^\kappa_\delta \) subspaces. However, it is then not transparent how 
to achieve closure under \( \leq \kappa \)-sized products. The na\"ive attempt of mimicking 
what is done in the classical case would require to first develop a theory of convergent 
\( \kappa \)-indexed series in some suitable group \( \GG \), and then use it to try to 
define the complete \( \GG \)-metric on the product. 
Theorem~\ref{thm:charK-Polish} provides an elegant bypass to these difficulties and 
directly leads us the the following theorem.

\begin{theorem} \label{thm:closureproperties}
The class of \( \GG \)-Polish spaces (equivalently: \( \kappa \)-additive \( \SFC_\kappa \)-spaces) is closed under \( G^\kappa_\delta \)-subspaces and \( \leq \kappa \)-sized sums and products.
\end{theorem}

\begin{proof}
For \( \leq \kappa \)-sized products, just notice that both the property of being 
\( \kappa \)-additive and the property of being strong fair \( \kappa \)-Choquet are straightforwardly preserved by such an operation.
\end{proof}

Moreover, we also get the analogue of Sierpi\'nski's theorem~\cite[Theorem 8.19]{KechrisMR1321597}: 
the classes of \( \GG \)-Polish spaces and  \( \SFC_\kappa \)-spaces are both closed under 
continuous open images. (Notice that a similar result holds for \( \SC_\kappa \)-spaces, as 
observed in~\cite[Proposition 2.7]{CoskSchlMR3453772}.)

\begin{theorem} \label{thm:closurecontopenimages}
Let \( X \) be \( \GG \)-Polish, and \( Y \) be a space of weight \( \leq \kappa \). If there is a continuous open surjection \( f \) from \( X \) onto \( Y \), then \( Y \) is \( \GG \)-Polish. 

The same is true is we replace \( \GG \)-Polishness by the (weaker) property of being an \( \SFC_\kappa \)-space.
\end{theorem}

\begin{proof}
By Theorem~\ref{thm:charK-Polish}, it is enough to show that the properties of being 
strong fair \( \kappa \)-Choquet and being \( \kappa \)-additive are preserved by \( f \). 
The former is straightforward. For the latter, let \( (U_\alpha)_{\alpha < \nu} \) be a  
sequence of open subsets of \( Y \), for some \( \nu < \kappa \). If 
\( \bigcap_{\alpha < \nu} U_\alpha \neq \emptyset \), let \( y \) be arbitrary in 
\( \bigcap_{\alpha < \nu} U_\alpha \) and, using surjectivity of \( f \), let \( x \in X \) be such 
that \( f(x) = y \). Since \( x \in \bigcap_{\alpha < \nu} f^{-1}(U_\alpha) \) and the latter set 
is open by \( \kappa \)-additivity of \( X \), there is \( V \subseteq X \) open such that 
\( x \in V \subseteq \bigcap_{\alpha < \nu} f^{-1}(U_\alpha) \). It follows that \( f(V) \) is an 
open neighborhood of \( y \) such that \( f(V) \subseteq \bigcap_{\alpha < \nu} U_\alpha \), as desired.
\end{proof}

There is still one interesting open question related to \( \SFC_\kappa \)-subspaces of a 
given space of weight \( \leq \kappa \). By 
Corollary~\ref{cor:subspace}, if \( X \) is also \( \kappa \)-additive and \( Y \subseteq X \) is 
an \( \SFC_\kappa \)-subspace of it, then \( Y \) is \( G^\kappa_\delta \) in \( X \). We do 
not know if the same remains true if we drop \( \kappa \)-additivity. The following 
corollary is the best result we have in this direction: it follows from 
Theorem~\ref{thm:charK-metrizable} and the fact that 
by \( \kappa^{< \kappa} = \kappa \) and the proof of Proposition~\ref{prop:refintetoGpolish}, 
every (regular Hausdorff) topology of weight \( \leq \kappa \) 
can be naturally refined to a \( \kappa \)-additive one  in such a way that 
the new open sets are $F^\kappa_\sigma$ 
in the old topology.

\begin{corollary} \label{cor:not_k_additive_subspace}
Let \( X \) be a space of weight $\leq \kappa$, and \( Y \subseteq X \) 
be an \( \SFC_\kappa \)-subspace of it. Then \( Y \) is a \( \leq \kappa \)-sized intersection 
of \( F^\kappa_\sigma \) subsets of \( X \).
\end{corollary}

It is then natural to ask whether the above computation can be improved.

\begin{question} 
In the same hypotheses of Corollary~\ref{cor:not_k_additive_subspace}, is $Y$ a 
$G_\delta^\kappa$ subset of $X$? What if we assume that \( X \) be \( \SFC_\kappa \)?
\end{question}

In the literature on generalized descriptive set theory, the notion of an analytic set is 
usually generalized as follows.

\begin{definition} \label{def:kappa-analytic}
A subset of a space%
\footnote{Since \( \SFC_\kappa \)-spaces have been introduced in the present paper, the definition of \( \kappa \)-analytic sets given in the literature is of course usually restricted to the spaces \( \pre{\kappa}{\kappa} \) and \( \pre{\kappa}{2} \) and their powers. The only exception is~\cite{MottoRos:2013}, where it is given for all \( \leq \kappa \)-weighted topologies generating a standard \( \kappa \)-Borel structure (see~\cite[Definitions 3.6 and 3.8]{MottoRos:2013}).}
 of weight \( \leq \kappa \)
 is \markdef{\( \kappa \)-analytic} if and only if it is a continuous image of a closed subset 
of \( \pre{\kappa}{\kappa} \). A set is \markdef{\( \kappa \)-coanalytic} if its complement 
is \( \kappa \)-analytic, and it is \markdef{\( \kappa \)-bianalytic} if it is both 
\( \kappa \)-analytic and \( \kappa \)-coanalytic.
\end{definition}
Although the definition works for a larger class of spaces, in this paper we will concentrate on subsets of \( \SFC_\kappa \)-spaces.
Analogously to what happens 
in the classical case, one can then prove that Definition~\ref{def:kappa-analytic} is equivalent to several 
other variants: for example, a set \( A \subseteq \pre{\kappa}{\kappa} \) is 
\( \kappa \)-analytic if and only if it is the projection of a closed 
\( C \subseteq (\pre{\kappa}{\kappa})^2 \), if and only if%
\footnote{This reformulation involves only \( \kappa \)-Borel sets and functions, thus the notion of a \( \kappa \)-analytic set is independent on the actual topology. This allows us to naturally extend this concept to subsets of arbitrary (standard) \( \kappa \)-Borel spaces.}
 it is a \( \kappa \)-Borel image of 
some set \( B \in \Bor_\kappa(\pre{\kappa}{\kappa}) \) (see \cite[Corollary 7.3]{AM} 
and~\cite[Proposition 3.11]{MottoRos:2013}). As explained 
in~\cite[Theorem 1.5]{LuckeSchlMR3430247}, a major difference from the classical setup 
is instead that we cannot add among the equivalent reformulations of 
\( \kappa \)-analyticity that of being a continuous image of the whole 
\( \pre{\kappa}{\kappa} \)---this condition defines a properly smaller class when 
\( \kappa \) is uncountable (and, as usual, \( \kappa^{< \kappa} = \kappa \)).

The reason for using Definition~\ref{def:kappa-analytic} instead of directly 
generalizing~\cite[Definition 14.1]{KechrisMR1321597} is precisely that we were still 
lacking an appropriate notion of generalized Polish-like space. We can now fill this gap.

\begin{proposition} \label{prop:kappa-analytic}
Let \( X \) be an \( \SFC_\kappa \)-space. For any \( A \subseteq X \),
the following are equivalent:
\begin{enumerate-(a)}
\item \label{prop:kappa-analytic-a}
\( A \) is \( \kappa \)-analytic (i.e.\ a continuous image of a closed subset of \( \pre{\kappa}{\kappa} \));
\item \label{prop:kappa-analytic-b}
\( A \) is a continuous image of a \( \GG \)-Polish space;
\item \label{prop:kappa-analytic-c}
\( A \) is a continuous image of an \( \SFC_\kappa \)-space.
\end{enumerate-(a)}
\end{proposition}

\begin{proof}
The implications~\ref{prop:kappa-analytic-a} \( \Rightarrow \) \ref{prop:kappa-analytic-b} 
and \ref{prop:kappa-analytic-b} \( \Rightarrow \) \ref{prop:kappa-analytic-c} follow from 
Theorem~\ref{thm:charK-Polish}. For the remaining 
implication~\ref{prop:kappa-analytic-c} \( \Rightarrow \) \ref{prop:kappa-analytic-a}, 
suppose that \( Y \) is an \( \SFC_\kappa \)-space and that \( g \colon Y \to X \) is 
continuous and onto \( A \). Use
 Proposition~\ref{prop:refintetoGpolish} to refine the topology \( \tau \) of \( Y \) to a 
 topology \( \tau' \) such that \( (Y,\tau') \) is \( \kappa \)-additive and  still 
 \( \SFC_\kappa \). 
Notice that $g \colon (Y,\tau') \to X$ is still continuous because $\tau' \supseteq \tau$.
 Use Theorem~\ref{thm:charK-Polish} again to find a closed set 
 \( C \subseteq \pre{\kappa}{\kappa} \) and a homeomorphism \( f \colon C \to (Y,\tau') \): 
 then \( g \circ f \) is a continuous surjection from \( C \) onto \( A \).
\end{proof}

Clearly, in  Proposition~\ref{prop:kappa-analytic}\ref{prop:kappa-analytic-c} we can
equivalently consider \( \kappa \)-additive \( \SFC_\kappa \)-spaces. We
 instead cannot restrict ourselves to  \( \SC_\kappa \)-spaces, even when further 
 requiring \( \kappa \)-additivity. Indeed,  by Theorem~\ref{thm:SC-space+G-Polish} 
 and~\cite[Proposition 1.3]{LuckeSchlMR3430247}
every such space is a continuous image of the whole \( \pre{\kappa}{\kappa} \): it follows 
that the collection of all continuous images of \( \kappa \)-additive 
\( \SC_\kappa \)-spaces coincides with the collection of continuous images of 
\( \pre{\kappa}{\kappa } \), and it is thus strictly smaller than the class of all 
\( \kappa \)-analytic sets by the mentioned~\cite[Theorem 1.5]{LuckeSchlMR3430247}.

A variant of Definition~\ref{def:kappa-analytic} considered 
in~\cite{LuckeSchlMR3430247} is the class \( I^\kappa_{\mathrm{cl}} \) of continuous 
\emph{injective} images of closed subsets of \( \pre{\kappa}{\kappa} \) (clearly, all such 
sets are in particular \( \kappa \)-analytic). When \( \kappa = \omega \) the class 
\( I^\kappa_{\mathrm{cl}} \) coincides with Borel sets, but when 
\( \kappa > \omega \) the class \( I^\kappa_{\mathrm{cl}} \) is strictly 
larger than \( \Bor_\kappa(\pre{\kappa}{\kappa}) \) 
by~\cite[Corollary 1.9]{LuckeSchlMR3430247}. Moreover, if 
\( \mathrm{V} = \mathrm{L}[x] \) with \( x \subseteq \kappa \), then 
by~\cite[Corollary 1.14]{LuckeSchlMR3430247} all \( \kappa \)-analytic subsets of 
\( \pre{\kappa}{\kappa} \) belong to \( I^\kappa_{\mathrm{cl}} \). This result can be extended to \( \kappa \)-analytic subsets of arbitrary \( \SFC_\kappa \)-spaces.

\begin{corollary} \label{cor:Ikappaclforsfc}
Assume that \( \mathrm{V} = \mathrm{L}[x] \) with \( x \subseteq \kappa \), and let \( X \) be an arbitrary \( \SFC_\kappa \)-space. Then every \( \kappa \)-analytic \( A \subseteq X \) is a continuous injective image of a closed subset of \( \pre{\kappa}{\kappa} \).
\end{corollary}

\begin{proof}
By Corollary~\ref{cor:refintetoGpolish-} there is a closed \( C \subseteq \pre{\kappa}{\kappa} \) and a continuous bijection \( f \colon C \to X \). Notice that \( f^{-1}(A) \) is \( \kappa \)-analytic in \( C \) because the class of \( \kappa \)-analytic sets is easily seen to be closed under continuous preimages, hence it is \( \kappa \)-analytic in \( \pre{\kappa}{\kappa} \) as well. By~\cite[Corollary 1.14]{LuckeSchlMR3430247} there is a continuous injection from some closed \( D \subseteq \pre{\kappa}{\kappa} \) onto \( f^{-1}(A) \), which composed with \( f \) gives the desired result.
\end{proof}

We are now going to 
show that the class \( I^\kappa_{\mathrm{cl}} \) can be characterized through changes of 
topology.

\begin{theorem} \label{thm:Ikappacl}
Let \( (X,\tau) \) be an \( \SFC_\kappa \)-space and \( A  \subseteq X \). Then the following are equivalent:
\begin{enumerate-(a)}
\item
\( A \in I^\kappa_{\mathrm{cl}} \);
\item
there is an \( \SFC_\kappa \) topology \( \tau' \) on \( A \) such that \( \tau' \supseteq \tau \restriction A \).
\end{enumerate-(a)}
\end{theorem}

\begin{proof}
Suppose first that \( C \subseteq \pre{\kappa}{\kappa} \) is closed and 
\( f \colon C \to X \) is a continuous injection with range \( A \). Let \( \tau' \) be obtained 
by pushing forward along \( f \) the (relative) topology of \( C \), so that \( (A, \tau') \) and 
\( C \) are homeomorphic. Then \( (A,\tau') \) is an \( \SFC_\kappa \)-space by 
Theorem~\ref{thm:charK-Polish}, and \( \tau' \) refines \( \tau \restriction A \) because 
\( f \) was continuous.

Conversely, if \( (A, \tau' ) \) is an \( \SFC_\kappa \)-space then by Theorem~\ref{thm:charK-Polish} 
again  there is a closed \( C \subseteq \pre{\kappa}{\kappa} \) and a homeomorphism 
\( f \colon C \to (A, \tau') \). Since \( \tau' \supseteq \tau \restriction A \), it follows that 
\( C \) and \( f \) witness \( A \in I^\kappa_{\mathrm{cl}} \).
\end{proof}

This also allows us to precisely determine to what extent the technique of change of 
topology discussed in Section~\ref{sec:standardBorel} can be applied to 
non-\( \kappa \)-Borel sets.

\begin{corollary} \label{cor:limitationchangesoftopologies}
Let \( (X,\tau) \) be an \( \SFC_\kappa \)-space.
\begin{enumerate-(1)}
\item \label{cor:limitationchangesoftopologies-a}
Let \( A \subseteq X \).
If there is an \( \SFC_\kappa \) topology \( \tau' \supseteq \tau \) on \( X \) such that \( A \) is \( \tau' \)-clopen (or even just \( A \in \Bor_\kappa(X,\tau') \)), then \( A \) is \( \kappa \)-bianalytic.
\item \label{cor:limitationchangesoftopologies-b}
If \( \mathrm{V} = \mathrm{L}[x] \) with \( x \subseteq \kappa \), then for all \( \kappa \)-bianalytic \( A \subseteq X \) there is a \( \kappa \)-additive \( \SFC_\kappa \) topology \( \tau' \supseteq \tau \) on \( X \) such that \( A \) is \(\tau' \)-clopen.
\end{enumerate-(1)}
\end{corollary}

\begin{proof}
For part~\ref{cor:limitationchangesoftopologies-a} observe that since \( A \) is 
\( \tau' \)-clopen, then by Proposition~\ref{prop:G_deltasubspaces} both \( A \) and 
\( X \setminus A \) are \( \SFC_\kappa \)-spaces when endowed with the relativization of 
\( \tau' \). Therefore by Theorem~\ref{thm:Ikappacl} they are in 
\( I^\kappa_{\mathrm{cl}} \), and thus \( \kappa \)-analytic. If instead of \( A \) being 
\(\tau' \)-clopen we just have \( A \in \Bor_\kappa(X,\tau') \), use 
Proposition~\ref{prop:changeoftopology} to further refine \( \tau' \) to a suitable 
\( \tau'' \) turning \( A \) into a \( \tau'' \)-clopen set, and then apply the previous 
argument to \( \tau'' \) instead of \( \tau' \).

We now move to part~\ref{cor:limitationchangesoftopologies-b}. 
By Corollary~\ref{cor:Ikappaclforsfc}, under our assumption all 
\( \kappa \)-analytic subsets of \( X \) are in \( I^\kappa_{\mathrm{cl}} \). It follows that for every 
\( \kappa \)-bianalytic set \( B \subseteq X \) there is a continuous 
bijection \( f \colon C \to X \) with 
\( C \subseteq \pre{\kappa}{\kappa} \) closed and \( f^{-1}(B) \) clopen relatively to  \( C \): 
just fix \( f_0 \colon C_0 \to B \) and \( f_1 \colon C_1 \to X \setminus B \) witnessing 
\( B  \in I^\kappa_{\mathrm{cl}} \) and \( X \setminus B \in I^\kappa_{\mathrm{cl}} \), 
respectively,  let \( C \) be the sum of \( C_0 \) and \( C_1 \), and set \( f = f_0 \cup f_1 \). 
Pushing forward the topology of \( C \) along \( f \) we then get the desired \( \tau' \) (the 
fact that \( \tau' \supseteq \tau_b \) follows again from the continuity of \( f \)). 
\end{proof}

Corollary~\ref{cor:limitationchangesoftopologies} justifies our claim that there might be non-\( \kappa \)-Borel sets that can be turned into clopen sets via a nice change of topology (see item~\ref{limitationtochangeoftopology} on page~\pageref{limitationtochangeoftopology}). Indeed, 
when \( \kappa \) is uncountable there are
\( \kappa \)-bianalytic subsets of \( \pre{\kappa}{\kappa} \) which are not \( \kappa \)-Borel (see e.g.\ \cite[Theorem 18]{Friedman:2011nx}), and Corollary~\ref{cor:limitationchangesoftopologies}\ref{cor:limitationchangesoftopologies-b} applies to them.

Having extended the notion of a \( \kappa \)-analytic set to arbitrary 
\( \SFC_\kappa \)-spaces, it is natural to ask whether the deep analysis carried out 
in~\cite{LuckeSchlMR3430247} can be transferred to such wider context. Some of the results have already been explicitly extended in this paper, see e.g.\ Corollaries~\ref{cor:simultaneousembeddings}, \ref{cor:borelasinjimage}, and~\ref{cor:Ikappaclforsfc},  which extend, respectively,~\cite[Proposition 1.3, Lemma 1.11, and Corollary 1.14]{LuckeSchlMR3430247}. Other results naturally transfer to our general setup using the ideas developed so far. 
 
\begin{question}
Which other results from \cite{LuckeSchlMR3430247} hold for $\kappa$-analytic 
subsets of arbitrary $\SFC_\kappa$-spaces? For example, for which \( \SFC_\kappa \)-spaces \( X \) is there  a closed \( C \subseteq X \) which is not a continuous image of the whole \( \pre{\kappa}{\kappa} \), or a non-\( \kappa \)-Borel set \( A \subseteq X \) which is an injective continuous image of \( \pre{\kappa}{\kappa} \)?
\end{question}

Similar questions can be raised about the  analogue of the Hurewicz dichotomy for 
\( \kappa \)-analytic subsets of $\pre{\kappa}{\kappa}$ studied 
in~\cite{LuckeMottSchlMR3557473}.

We now move to generalizations of the perfect set property.

\begin{definition} \label{def:kappaPSP}
Let \( X \) be a space. A set \( A \) has the \markdef{\( \kappa \)-perfect set property} (\markdef{\( \kappa \)-{\PSP}} for short) if either \( |A| \leq \kappa \) or \( A \) contains a closed set homeomorphic to \( \pre{\kappa}{2} \).
\end{definition}

The \( \kappa \)-Borel version of the \( \kappa \)-{\PSP} would then read as follows: either \( |A| \leq \kappa \) or \( A \) contains a \( \kappa \)-Borel set which is \( \kappa \)-Borel isomorphic to \( \pre{\kappa}{2} \). However, for most applications it is convenient to consider a slightly stronger reformulation.

\begin{definition} \label{def:BorelkappaPSP}
Let \( X \) be a space. A set \( A \) has the \markdef{Borel \( \kappa \)-perfect set property} (\markdef{\bPSP} for short) if either \( |A| \leq \kappa \) or there is a \emph{continuous} \( \kappa \)-Borel embedding  \( f \colon \pre{\kappa}{2} \to A \) with \( f(\pre{\kappa}{2}) \in \Bor_\kappa(X) \).
\end{definition}

By Corollary~\ref{cor:borelimages}, if the \( \kappa \)-Borel structure of \( X \) is standard then the fact that \( f(\pre{\kappa}{2}) \in \Bor_\kappa(X) \) follows from the other conditions.
Notice also that the {\bPSP} is in general strictly weaker than the \( \kappa \)-{\PSP}. For example, consider the space \( X =  \pre{\kappa}{2} \) equipped with the \emph{product} topology. It is a \( \kappa \)-perfect \( \SC_\kappa \)-space, hence \( X \) itself and all its open subsets have the {\bPSP} (see Corollary~\ref{cor:perfectSC_kappa} below). However, \( X \) is compact: thus its clopen subsets cannot contain a closed homeomorphic copy of the generalized Cantor space, which clearly is not compact, and thus they do not have the \( \kappa \)-{\PSP}.

In Definitions~\ref{def:kappaPSP} and~\ref{def:BorelkappaPSP} we are of course allowing the special case \( A = X \). With this 
terminology, Theorem~\ref{thm:PSP for perfect SC_kappa kappa-additive} asserts that 
the \( \kappa \)-{\PSP} holds for all \( \kappa \)-additive \( \kappa \)-perfect 
\( \SC_\kappa \)-spaces.
From this and Proposition~\ref{prop:refintetoGpolish}, we can 
easily infer the following fact, which is just a more precise formulation 
of~\cite[Proposition 3.1]{CoskSchlMR3453772}. (Of course here we are also using that if 
\( \tau \) is \( \kappa \)-perfect, then the topology from the proof of 
Proposition~\ref{prop:refintetoGpolish} is still \( \kappa \)-perfect.)

\begin{corollary} \label{cor:perfectSC_kappa}
If \( X \) is a nonempty \( \kappa \)-perfect \( \SC_\kappa \)-space, then there is a 
continuous \( \kappa \)-Borel embedding from \( \pre{\kappa}{2} \) into \( X \) (with a 
\( \kappa \)-Borel range, necessarily). In particular,  the \emph{\bPSP} holds for 
\( \kappa \)-perfect \( \SC_\kappa \)-spaces.
\end{corollary}

It is instead independent of \( \ZFC \) whether the (Borel) \( \kappa \)-perfect set 
property holds for (\( \kappa \)-additive) \( \SFC_\kappa \)-spaces. Indeed, if there is a 
\( \kappa \)-Kurepa tree \( T \) with \( < 2^\kappa \)-many branches, then no 
\( \kappa \)-{\PSP}-like property can hold for \( [T] \) because of cardinality reasons. 
Moreover, this statement is also independent from $\ZFC+\mathsf{GCH}$. To see this, suppose that  $\mathsf{GCH}$ holds and $T$ is a \emph{slim} $\kappa$-Kurepa tree, that is, a \( \kappa \)-Kurepa tree such that $|\Lev_\alpha(T)|\leq |\alpha|$ for all infinite $\alpha<\kappa$. 
In $L$, such trees exist for uncountable regular $\kappa$ if and only if $\kappa$ is not ineffable \cite[Section 2, Theorems 9 \& 10]{jensenkunen1969}. Since \( T \) is \( \kappa \)-Kurepa, \( [T] \) has more than \( \kappa \)-many branches. 
If there were a continuous $\kappa$-Borel embedding $f \colon \pre{\kappa}{2} \rightarrow [T]$, then by~\cite[Lemma 2.9]{LuckeMottSchlMR3557473} the tree \( T \) would contain a splitting superclosed subtree \( T' \).
By an argument similar to Claim~\ref{claimcharcantor}, we could then find \( \alpha < \kappa \) such that \( | \Lev_\alpha(T')| > |\alpha| \), contradicting the fact that \( T \supseteq T' \) is slim. 
An analogous argument works for \emph{stationary slim} $\kappa$-Kurepa trees as in \cite[Section 1.2]{lucke2020descriptive}, with the upshot that the existence of such trees is consistent with $\kappa$ being supercompact \cite[Section IV.2.3]{Friedman:2011nx}.

On 
the other hand, in~\cite{Schlicht2017ac} the third author constructed a model of 
\( \ZFC \) where all ``definable'' subsets of \( \pre{\kappa}{\kappa} \) (including e.g.\ all 
\( \kappa \)-analytic sets and way more) have the {\bPSP}: combining 
Proposition~\ref{prop:refintetoGpolish} with Theorem~\ref{thm:charK-Polish} we then 
get that such property holds for arbitrary \( \SFC_\kappa \)-spaces and their ``definable'' 
subsets. Indeed, the same reasoning combined with 
Proposition~\ref{prop:condition_for_std_Borel_topology} can be used to show that 
\emph{if} the {\bPSP} holds for all closed subsets of, say, \( \pre{\kappa}{\kappa} \), then 
it automatically propagates to all \( \kappa \)-Borel subsets of all 
\( \SFC_\kappa \)-spaces. Moreover, we can even just start from superclosed sets 
(equivalently, up to homeomorphism, from \( \kappa \)-additive \( \SC_\kappa \)-spaces). 
Indeed, if \( C = [T] \subseteq \pre{\kappa}{\kappa} \) is closed, then arguing as in the 
proof of Lemma~\ref{lem:superclosed} we can construct a superclosed set \( C' = [T'] \) 
such that \( C \subseteq C' \), \( |C'| \leq \max \{ |C|, \kappa \} \), and all points in 
\( C' \setminus C \) are isolated in \( C ' \). It follows that if the {\bPSP} holds for \( C' \) 
then it holds also for \( C \) because if
 \( f \colon \pre{\kappa}{2} \to C' \) is a continuous injection then 
 \( f(\pre{\kappa}{2}) \subseteq C \) (use the fact that \( \pre{\kappa}{2} \) is  perfect).
Summing up we thus have:

\begin{theorem} \label{thm:PSPpropagates}
The following are equivalent:
\begin{enumerate-(a)}
\item
the \emph{\bPSP} holds for superclosed subsets of \( \pre{\kappa}{\kappa} \);
\item
the \emph{\bPSP} holds for closed subsets of \( \pre{\kappa}{\kappa} \);
\item
the \emph{\bPSP} holds for all (\( \kappa \)-additive) \( \SFC_\kappa \)-spaces;
\item
the \emph{\bPSP} holds for all \( \kappa \)-Borel subsets of all \( \SFC_\kappa \)-spaces.
\end{enumerate-(a)}
\end{theorem}

The Borel \( \kappa \)-perfect set property for \( \SFC_\kappa \)-spaces has important consequences for their classification up to \( \kappa \)-Borel isomorphisms. 

\begin{corollary} \label{cor:kappaBorelisom}
Suppose that the \emph{\bPSP} holds for (super)closed subsets of \( \pre{\kappa}{\kappa} \). If \( X \) is an \( \SFC_\kappa \)-space with \( |X| > \kappa \), then \( X \) is \( \kappa \)-Borel isomorphic to \( \pre{\kappa}{2} \).
In particular, any two \( \SFC_\kappa \)-spaces \( X,Y \) are \( \kappa \)-Borel isomorphic if and only if \( |X| = |Y| \).
\end{corollary}

In particular, if the \emph{\bPSP} holds for (super)closed subsets of \( \pre{\kappa}{\kappa} \)
then up to \( \kappa \)-Borel isomorphism
 the generalized Cantor space \( \pre{\kappa}{2} \) is the unique \( \SFC_\kappa \)-space of size \( > \kappa \).

\begin{proof}
By our assumption and Theorem~\ref{thm:PSPpropagates}, \( \pre{\kappa}{2} \) is 
\( \kappa \)-Borel isomorphic to a \( \kappa \)-Borel subsets of \( X \). Conversely, \( X \) is 
\( \kappa \)-Borel isomorphic to a \( \kappa \)-Borel subset of \( \pre{\kappa}{2} \) by 
Theorem~\ref{th:standardBorel} and the fact that \( \pre{\kappa}{2} \) and 
\( \pre{\kappa}{\kappa} \) are \( \kappa \)-Borel isomorphic. Thus the result follows from 
the natural \( \kappa \)-Borel version of the usual Cantor-Schr\"oder-Bernstein argument.
\end{proof}

Using the same arguments and Corollary~\ref{cor:perfectSC_kappa} we also obtain that 
when restricting to \( \kappa \)-perfect \( \SC_\kappa \)-spaces the conclusions of 
Corollary~\ref{cor:kappaBorelisom} hold unconditionally---see~\cite[Corollary 3.7]{CoskSchlMR3453772}.

Without the \bPSP, the $\kappa$-Borel isomorphism types of $ \SFC_\kappa$-spaces are not necessarily uniquely determined by their cardinality. For instance, suppose that $T$ is a slim $\kappa$-Kurepa tree in a model of \( \mathsf{GCH} \), so that \( [T] \) is a \( \kappa \)-additive \( \SFC_\kappa \)-space with the same size  of \( \pre{\kappa}{2} \) (by $2^\kappa = \kappa^+$). Suppose towards a contradiction that there exists a $\kappa$-Borel isomorphism $f\colon  \pre{\kappa}{2}\rightarrow [T]$. 
Since $f$ is continuous on a comeager subset of $\pre{\kappa}{2}$,  there is an injective continuous map $g \colon \pre{\kappa}{2} \to [T]$. 
As in the paragraph after Corollary~\ref{cor:perfectSC_kappa}, this contradicts the fact that \( T \) is slim.

When dealing with topological game theory, one often wonders about what kind of 
winning strategies the players have at disposal in the given game. In this context, one can differentiate 
between perfect information strategies, that need to know all previous moves in order 
to be able to give an answer, and tactics, that instead rely only on the last move to 
determine the answer. 
The two notions do not coincide in general: there are games where a player has a 
winning strategy, but not a winning tactic. For example, ~\cite{DebsMR817083} 
describes a topological space where player $\pII$ has a winning strategy but no winning 
tactic in the classical strong Choquet game (see also \cite{lupton2014}).  Debs' example 
can easily be adapted to show that there exists a (non-$\kappa$-additive) topological 
space of weight $\kappa$ where player $\pII$ has a winning strategy but not a winning 
tactic in $\fCho^s_\kappa(X)$ (or in $\Cho^s_\kappa(X)$), or that there is a 
$\kappa$-additive topological space of weight $>\kappa$ with the same property.
In contrast, Proposition~\ref{prop:charadditiveSCandSFCkappaspaces} implies that 
for $\kappa$-additive spaces of weight $\leq \kappa$
the two notions of winning tactic and winning strategy can be used interchangeably.

\begin{corollary}\label{cor:strategies_and_tactics}
Let $X$ be a $\kappa$-additive space of weight $ \leq \kappa$. Then $\pII$ has a winning strategy in $\fCho^s_\kappa(X)$ (resp. $\Cho^s_\kappa(X)$) if and only if she has a winning tactic.
\end{corollary}

\begin{proof}
For the nontrivial direction,
by Proposition~\ref{prop:charadditiveSCandSFCkappaspaces} we can restrict the 
attention to (super)closed subsets of \( \pre{\kappa}{\kappa} \), so let $X=[T]$ for some
pruned tree $T\subset \pre{<\kappa}{\kappa}$. Then any function $\sigma \colon \tau \to \tau$ 
that associate to every nonempty open set $U \subseteq [T]$ a basic open set 
$\Nbhd_s\cap [T]\subset U$ for some $s\in T$ is a winning tactic for $\pII$ in 
$\fCho^s_\kappa([T])$. Indeed, the answers $N_{s_\alpha}\cap [T]$  of $\sigma$ at every 
round $\alpha$ are such that $s_\alpha\subset s_\beta$ for any $\alpha<\beta<\kappa$. 
Hence, if the game does not stop before $\kappa$-many rounds, then the final 
intersection $\bigcap_{\alpha<\kappa}\Nbhd_{s_\alpha}\cap [T]$ is nonempty, since it 
contains $s=\bigcup_{\alpha<\kappa} s_\alpha$ (or any sequence extending $s$, if $s$ has length $<\kappa$).
A similar argument shows that if $T$ is superclosed, than the tactic described above is winning also for $\Cho^s_\kappa([T])$.
\end{proof}

For more details about perfect information strategies and tactics, and for some interesting problems in the field, see for example \cite{ScheepersMR2367385}.

In this paper we generalized metrics by allowing values in structures different from $\RR$. Another possible generalization of metric spaces is given by uniform spaces.
In this context  we have a notion of completeness as well, which is however strictly weaker than the notions we considered so far.
Indeed, all \( \GG \)-metrizable spaces of weight \( \leq \kappa \) (that is, by Theorem~\ref{thm:charK-metrizable}, all subspaces of \( \pre{\kappa}{\kappa} \)) are paracompact and Hausdorff, and this entails that they are completely uniformizable.
It follows that any non-\( G^\kappa_\delta \) subset of \( \pre{\kappa}{\kappa} \) is a 
completely uniformizable space of weight \( \leq \kappa \) which is not \( \SFC_\kappa \) 
and, more generally, that the class of completely uniformizable spaces of weight 
\( \leq \kappa \) properly extends the class of all $\kappa$-additive spaces with weight 
\( \leq \kappa \) (recall that we are tacitly restricting to regular Hausdorff spaces). Thus 
by Theorem~\ref{thm:char_complexity_std_Borel_subspaces} such a class contains spaces 
which are not even \( \kappa \)-Borel isomorphic to an \( \SFC_\kappa \)-space (that is, 
they are not standard \( \kappa \)-Borel): this seems to rule out the possibility of 
developing a decent (generalized) descriptive set theory in such a generality. 
Nevertheless, from the topological perspective
it would still be interesting to know whether this property also  extends the class of 
non-$\kappa$-additive $\SFC_\kappa$-spaces or at least $\SC_\kappa$-spaces. 

\begin{question}
Is every $\SFC_\kappa$-space completely uniformizable?
\end{question}

\end{document}